\documentclass[a4paper, 12pt, DIV=12, BCOR=0mm, oneside, headinclude=true, footinclude=false, final]{scrartcl}
\usepackage{fixltx2e}					% Korrigiert bugs in LaTeX2e

\usepackage[british]{babel}					%Silbentrennung: ngerman, english
\usepackage{ucs}									%Support für UTF-8 Encoding
\usepackage[utf8x]{inputenc}				%Zeichencodierung: %applemac %latin1
\usepackage{lmodern}							%Latin Modern fonts

\usepackage{graphicx}							%Enhanced support for graphics
\usepackage{amssymb, amsmath}			%AMS Latex
\usepackage{mathtools}							%Trickkiste für Displaymath
\usepackage{mathrsfs}								%Script and Calgraphic
\usepackage{bbm}									%BlackBoard Modified
\usepackage{bm}										%Bold Math
\usepackage{paralist}								%Enhanced Lists (itemize,...)
\usepackage[amsmath,hyperref,thmmarks]{ntheorem}	%Theorem-Umgebungen
\usepackage{multicol}
\usepackage[right]{eurosym}
\usepackage[varumlaut]{yfonts}					% Frakturschriften
\usepackage{suetterl}								% Sütterlin
\usepackage[normalem]{ulem}					% Benutzerdefinierte Textauszeichung (\sout...)

\usepackage[usenames,dvipsnames]{color}		%Farbgebungen
\usepackage[pdftex, bookmarks=true, unicode=true, pdfborder={0 0 0}, colorlinks=true, linkcolor=blue, citecolor=OliveGreen, filecolor=cyan, urlcolor=cyan, hyperindex, linktoc=page]{hyperref} 		%pdftex,
\numberwithin{equation}{section}				% Nummerierung umstellen (ZWINGEND VOR CREF!!)
\usepackage{cleveref}							%Cref
\usepackage[longnamesfirst]{natbib} 	 %numbers	%

\usepackage[automark]{scrpage2}
\clearscrheadfoot
\usepackage{etoolbox}
\addtokomafont{captionlabel}{\footnotesize}
\addtokomafont{caption}{\footnotesize}
\newcommand{\cadlag}{càdlàg}

\newcommand{\eg}{e.\,g.\relax}
\newcommand{\ie}{i.\,e.\relax}
\newcommand{\wrt}{w.\,r.\,t.\relax}

\newcommand{\Iota}{\mathrm{I}}

\newcommand{\dr}{\mathrm{d}r}
\newcommand{\dx}{\mathrm{d}x}
\newcommand{\dt}{\mathrm{d}t}
\newcommand{\ds}{\mathrm{d}s}
\newcommand{\dy}{\mathrm{d}y}
\newcommand{\dz}{\mathrm{d}z}
\newcommand{\du}{\mathrm{d}u}
\newcommand{\dw}{\mathrm{d}w}

\newcommand{\transpose}{{}^{^{\intercal}}}
\renewcommand{\complement}{\mathrm{c}}
\newcommand{\mAnd}{\quad\text{and}\quad}
\newcommand{\cpart}{\complement}
\newcommand{\ball}[2]{B_{#2}(#1)}
\newcommand{\e}{\mathrm{e}}
\newcommand{\loc}{\mathrm{loc}}

\newcommand{\bbn}{\mathbbm{N}}

\newcommand{\bbr}{\mathbbm{R}}

\newcommand{\bbrp}{\mathbbm{R}_+}

\DeclareMathOperator*{\tr}{tr}

\DeclareMathOperator{\E}{\mathbbm{E}}
\DeclareMathOperator{\p}{\mathbbm{P}}

\DeclareMathOperator*{\supp}{supp}

\newcommand{\nto}{{n\to\infty}}
\newcommand{\tto}{{t\to\infty}}

\newcommand{\limN}{\xrightarrow[\nto]{}}

\newcommand{\limT}{\xrightarrow[\tto]{}}

	\newcommand{\limpPiN}{\xrightarrow[\nto]{\p^\pi}}
	
	\newcommand{\limpPiT}{\xrightarrow[\tto]{\p^\pi}}

	\newcommand{\limdN}{\xrightarrow[\nto]{\mathscr{L}}}
	\newcommand{\limdK}{\xrightarrow[k\to\infty]{\mathscr{L}}}

	\newcommand{\limdstT}{\xrightarrow[\tto]{\mathscr{L}\mathrm{-st}}}
	\newcommand{\limdstN}{\xrightarrow[\nto]{\mathscr{L}\mathrm{-st}}}

	\newcommand{\limwT}{\xrightarrow[\tto]{\rm w}}
	\newcommand{\limwN}{\xrightarrow[\nto]{\rm w}}

	\newcommand{\limasT}{\xrightarrow[\tto]{\rm a.s.}}

	\newcommand{\LimdT}{\underset{\tto}{\stackrel{\mathscr{L}}{\Longrightarrow}}}
	\newcommand{\LimdN}{\underset{\nto}{\stackrel{\mathscr{L}}{\Longrightarrow}}}

	\newcommand{\LimdstT}{\underset{\tto}{\stackrel{\mathscr{L}\mathrm{-st}}{\Longrightarrow}}}
	\newcommand{\LimdstN}{\underset{\nto}{\stackrel{\mathscr{L}\mathrm{-st}}{\Longrightarrow}}}

	\newcommand{\LimucpN}{\underset{\nto}{\stackrel{\rm \!ucp}{\Longrightarrow}}}

\newcommand{\coloneqq}{\mathrel{\mathop:}=}
\newcommand{\eqqcolon}{=\mathrel{\mathop:}}

\newcommand{\Halmos}{\quad\hfill\ensuremath{_\diamond}\smallskip}

\crefformat{equation}{(#2#1#3)}
\Crefformat{equation}{(#2#1#3)}

\theoremnumbering{arabic}
\theoremstyle{change}
\RequirePackage{latexsym}
\theoremsymbol{}
\theorembodyfont{\itshape}
\theoremheaderfont{\normalfont\bfseries}
\theoremseparator{.}
\newtheorem{Theorem}{Theorem}[section]

\crefname{Theorem}{Theorem}{Theorems}
\Crefname{Theorem}{Theorem}{Theorems}

\newtheorem{Proposition}[Theorem]{Proposition}

\crefname{Proposition}{Proposition}{Propositions}
\Crefname{Proposition}{Proposition}{Propositions}
\newtheorem{Lemma}[Theorem]{Lemma}

\crefname{Lemma}{Lemma}{Lemmata}
\Crefname{Lemma}{Lemma}{Lemmata}

\newtheorem{Corollary}[Theorem]{Corollary}

\crefname{Corollary}{Corollary}{Corollaries}
\Crefname{Corollary}{Corollary}{Corollaries}

\theoremsymbol{\ensuremath{_\Box}}%
\newtheorem{CorollaryOP}[Theorem]{Corollary}

\crefname{CorollaryOP}{Corollary}{Corollaries}
\Crefname{CorollaryOP}{Corollary}{Corollaries}

\theorembodyfont{\upshape}
\theoremsymbol{\ensuremath{_\diamond}}%
\newtheorem{Definition}[Theorem]{Definition}

\crefname{Definition}{Definition}{Definitions}
\Crefname{Definition}{Definition}{Definitions}
\newtheorem{Assumption}[Theorem]{Assumption}
\crefname{Assumption}{Assumption}{Assumptions}
\Crefname{Assumption}{Assumption}{Assumptions}

\theoremstyle{nonumberplain}
\theoremheaderfont{\itshape}
\theorembodyfont{\normalfont}
\theoremsymbol{}%

\newtheorem{Remark}{Remark}

\crefname{Remark}{Remark}{Remarks}
\Crefname{Remark}{Remark}{Remarks}
\newtheorem{Example}[Theorem]{Example}

\crefname{Example}{Example}{Examples}
\Crefname{Example}{Example}{Examples}

\crefname{Counterexample}{Counter example}{Counter examples}
\Crefname{Counterexample}{Counter example}{Counter examples}

\theoremsymbol{\ensuremath{_\Box}}
\RequirePackage{amssymb}

\newtheorem{proof}{Proof}

\qedsymbol{\ensuremath{_\diamond}}
\theoremclass{LaTeX}

\newcommand{\btheo}{\begin{Theorem}}
\newcommand{\etheo}{\end{Theorem}}
\newcommand{\bpro}{\begin{Proposition}}
\newcommand{\epro}{\end{Proposition}}
\newcommand{\bcor}{\begin{Corollary}}
\newcommand{\ecor}{\end{Corollary}}
\newcommand{\bproof}{\begin{proof}}
\newcommand{\eproof}{\end{proof}}
\newcommand{\blemma}{\begin{Lemma}}
\newcommand{\elemma}{\end{Lemma}}
\newcommand{\beq}{\begin{equation}}
\newcommand{\eeq}{\end{equation}}

\def\addnotation #1: #2#3{#1 \hfill \parbox{5in}{#2 \dotfill  \pageref{#3}}\\}

\makeatletter
\patchcmd{\NAT@citex}
  {\@citea\NAT@hyper@{\NAT@nmfmt{\NAT@nm}\NAT@date}}
  {\@citea\NAT@nmfmt{\NAT@nm}\NAT@hyper@{\NAT@date}}
  {}% Do nothing if patch works
  {}% Do nothing if patch fails

\patchcmd{\NAT@citex}
  {\@citea\NAT@hyper@{%
     \NAT@nmfmt{\NAT@nm}%
     \hyper@natlinkbreak{\NAT@aysep\NAT@spacechar}{\@citeb\@extra@b@citeb}%
     \NAT@date}}
  {\@citea\NAT@nmfmt{\NAT@nm}%
   \NAT@aysep\NAT@spacechar%
   \NAT@hyper@{\NAT@date}}
  {}% Do nothing if patch works
  {}% Do nothing if patch fails

\patchcmd{\NAT@citex}
  {\@citea\NAT@hyper@{%
     \NAT@nmfmt{\NAT@nm}%
     \hyper@natlinkbreak{\NAT@spacechar\NAT@@open\if*#1*\else#1\NAT@spacechar\fi}%
       {\@citeb\@extra@b@citeb}%
     \NAT@date}}
  {\@citea\NAT@nmfmt{\NAT@nm}%
   \NAT@spacechar\NAT@@open\if*#1*\else#1\NAT@spacechar\fi%
   \NAT@hyper@{\NAT@date}}
  {}% Do nothing if patch works
  {}% Do nothing if patch fails
\makeatother
\KOMAoptions{DIV=last}

\title{On Non-parametric Estimation of the~L\'evy~Kernel of Markov processes}
\author{Florian A. J. Ueltzhöfer\thanks{With the support of the Technische Universit\"at M\"unchen---Institute for Advanced Study, funded by the German Excellence Initiative; and support provided by the TUM International Graduate School of Science and Engineering (IGSSE).}~\thanks{\href{http://www-m4.ma.tum.de/personen/doktoranden/florian-ueltzhoefer/}{Lehrstuhl f\"ur mathematische Statistik} and \href{http://www.tum-ias.de/current-focus-groups/risk-analysis/risk-analysis-and-stochastic-modeling.html}{TUM Institute for Advanced Study}, Technische Universit\"at M\"unchen, Boltzmannstra{\ss}e 3, D--85\,748 Garching b.\,M.; \href{mailto:ueltzhoefer@ma.tum.de}{ueltzhoefer@ma.tum.de}}}
\date{}

\begin{document}
\maketitle
%:- Zusammenfassung (Abstraktum)
\vspace{-2\baselineskip}
\begin{abstract}\noindent
We consider a recurrent Markov process which is an Itô semi-martingale. The Lévy kernel describes the law of its jumps. Based on observations $X_0,X_\Delta,\dotsc,X_{n\Delta}$, we construct an estimator for the Lévy kernel's density. We prove its consistency (as $n\Delta\to\infty$ and $\Delta\to0$) and a central limit theorem. In the positive recurrent case, our estimator is asymptotically normal; in the null recurrent case, it is asymptotically mixed normal. 
Our estimator's rate of convergence equals the non-pa\-ra\-me\-tric minimax rate of smooth density estimation. The asymptotic bias and variance are analogous to those of the classical Nadaraya--Watson estimator for conditional densities. Asymptotic confidence intervals are provided.
%{We consider a recurrent Markov process which is an Itô semi-martingale. The law of its jumps is described by the L\'evy kernel. Based on the discrete-time observations $X_0,X_\Delta,\dotsc,X_{n\Delta}$, we construct an estimator for the density of the L\'evy kernel. We prove its consistency (as $n\Delta\to\infty$ and $\Delta\to0$) as well as a central limit theorem. At the core of our research, we also study the fundamental case where we observe a sample path of the process and, in particular, all its jumps.} In the positive recurrent case, our estimator is asymptotically normal; in the null recurrent case, our estimator is asymptotically mixed normal. Our estimator's rate of convergence equals the non-pa\-ra\-me\-tric minimax rate of smooth density estimation. The asymptotic bias and variance are analogous to those of the classical Nadaraya--Watson estimator for conditional densities. Asymptotic confidence intervals are provided.
\end{abstract}
{\small
\emph{AMS Subject Classification 2010}: Primary 62M05; secondary 62G07, 60F05, 60J25\\
\emph{Keywords}: Markov process, {Itô semi-martingale,} L\'evy system, {L\'evy kernel}, {null recurrence}, density estimation, central limit theorem

\smallskip
\footnotesize This is a preprint of a paper which has been accepted for publication in the journal \emph{Stochastic Processes and their Applications} on April 30, 2013.}
% 60J25: cont.-time Markov
% 60F05: CLT
% 62G07: Nonp. Dens. Estim.
% 62M05: Markov proc. Estim.
\allowdisplaybreaks

%:====================================================
\section{Introduction}\label{C Intro}
%:- Teaser
Statistical inference for jumps in continuous-time models has received significant attention in recent years. Due to their well-known tractability properties, a vast amount of literature has been devoted to the class of processes with stationary and independent increments, called \emph{L\'evy processes}. The law of their jumps {is} characterised by their L\'evy measure. Parametric inference for L\'evy measures has a long history. For recent developments in non-pa\-ra\-me\-tric settings, we refer, for instance, to \citet{Comte2011}; to \citet{fig09}; to the special issue \citet{eurandomProc}, which contains a collection of interesting papers; to \citet{Neumann2009}; and to \citet{Ueltzhoefer2011}. Ample references to previous literature can be found within the aforementioned.

%:- Framework (J)
In this paper, we consider a Harris recurrent Markov process $X$ {which is an Itô semi-martingale. Such a process is} a solution of some stochastic differential equation 
\begin{gather}\label{J eq: MarkovSDE}
	\begin{aligned}
		\mathrm{d}X_t  = b(X_t)\dt + \sigma(X_t)\mathrm{d}W_t   & + \int \delta(X_{t-},y) \mathbbm{1}_{\{\lVert \delta(X_{t-},y)\rVert>1\}} \mathfrak{p}(\dt,\dy) \\ 
			& + \int \delta(X_{t-},y)\mathbbm{1}_{\{\lVert \delta(X_{t-},y)\rVert\le1\}} (\mathfrak{p}-\mathfrak{q})(\dt,\dy),
	\end{aligned}
\end{gather}
with coefficients $b$, $\sigma$ and $\delta$; the SDE is driven by some Wiener process $W$ and some Poisson random measure $\mathfrak{p}$ (with intensity measure $\mathfrak{q}(\dt,\dy) = \dt \otimes \lambda(\dy)$); $X_{t-}$ denotes the left-limit. {The law of its jumps is more or less described by the kernel $F$ where,} for each $x$, the measure $F(x,\cdot)$ coincides with the image of the measure $\lambda$ under the map $y \mapsto \delta(x,y)$ restricted to the set $\{y: \delta(x,y) \neq 0 \}$. We call $F$ the \emph{(canonical) L\'evy kernel} of $X$. We assume that the measures $F(x,\dy)$ admit a density $y\mapsto f(x,y)$, and we aim for non-parametric estimation of the function $(x,y)\mapsto f(x,y)$.

%: Statistical problem (D) and Main Results
On an equidistant time grid, we observe a sample $X_0(\omega), X_\Delta(\omega), \dotsc, X_{n\Delta}(\omega)$ of the process; the jumps are latent. We study a kernel density estimator for $f(x,y)$. We show its consistency as $n\Delta\to\infty$ and $\Delta\to0$ under a smoothness hypothesis on the estimated density. In the ergodic case, we obtain asymptotic normality. In the null recurrent case, we impose a condition on the resolvent of the process which goes back to \citet{Darling1957}. Thereunder, we prove asymptotic mixed normality. We also provide a standardised version of our central limit theorem for the construction of asymptotic confidence intervals. 

%:- Bias, Variance and Rate
Our results are comparable to those in classical non-pa\-ra\-me\-tric density estimation. In particular: Our estimator's asymptotic bias and variance resemble those of the Nadaraya--Watson estimator in classical conditional density estimation. Just as in the classical context, moreover, the bandwidth choice is crucial for our estimator's rate of convergence. 
{%By the very nature of our estimator and its underlying dynamics, 
We conjecture that, for instance, a cross-validation method applies here analogously; see \citet{Fan2004} and \citet{Hall2004}.} By an optimal choice, if $\Delta\to0$ fast enough, the rate is $v({n\Delta})^{\alpha_1\alpha_2/[d(\alpha_1+\alpha_2) + 2\alpha_1\alpha_2]}$, where $\alpha_1>0$ (resp., $\alpha_2>0$) stands for the smoothness of $f$ as a function in $x$ (resp., in $y$), and the function $v$ plays the role of an information rate. In the ergodic case, $v(t)=t$; in the null recurrent case with Darling--Kac's condition imposed, $v(t) = t^\delta\ell(t)$ for some $0<\delta\le1$ and some slowly varying function $\ell$. We remark that, in the case $\alpha_1=\alpha_2$, our achieved rate $v({n\Delta})^{\alpha_1/(2\alpha_1+2d)}$ equals the non-pa\-ra\-me\-tric minimax rate of smooth density estimation, related to the smoothness of $f$ as a $2d$-dimensional function and with respect to $v({n\Delta})$. 

%:
{At the core of our statistical problem, we essentially have to study the case first, where the process is observed continuously in time and, in particular, all jumps are discerned. In this case, we can consider a more general class of quasi-left-continuous, strong Markov processes with \cadlag\ sample paths than just Itô semi-martingales. For these, the law of their jumps is again described by their L\'evy kernel. We present a version of our estimator which utilises that the sojourn time of certain sets and the jumps are observed. Under slightly weaker assumptions, we prove the estimator's consistency and asymptotic (mixed) normality. As these results are valid for a quite general class of processes, we believe that they are of independent interest, not only as a benchmark for all possible estimators which are based on some discrete observation scheme. }

{For discrete-time Markov chains, a related result is presented in \citet{Karlsen2001}. We are aware that our final steps of proof appear to be similar. We emphasise that the main difficulties in our context, however, come in two respects: on the one hand, from establishing an appropriate auxiliary framework where related methods apply; on the other hand, from the discrete observation scheme where our primary objects of interest --~the jumps~-- are latent.}

%:- Literature: Markov step processes (Kurzfassung)
{For continuous-time Markov processes,} apart from the L\'evy process case {and} as far as known to us, estimation of {their} L\'evy kernel has been confined to the special case of Markov step processes. For these, there is a one-to-one correspondence between the L\'evy kernel and the infinitesimal generator. Efficient non-pa\-ra\-me\-tric estimation of Markov step process models has been studied by \citet{Greenwood1994}. They assume the mean holding times to be bounded, and the transition kernel to be uniformly ergodic. This excludes the null recurrent case. The work on parametric estimation of Markov step processes is more exhaustive. The null recurrent case has been studied, for instance, by \citet{Hoepfner1993}. There, the process is observed up to a random stopping time such that a deterministic amount of information (or more) has been discerned. Local asymptotic normality is shown in various situations. With a slightly different aim, in contrast, \citet{Hoepfner1990} considers Markov step processes observed up to a deterministic time. Accordingly, the observed amount of information is random. Local asymptotic mixed normality (of statistical experiments) is shown under Darling--Kac's condition. Here, we utilise some of their results and methods. 
%:   Langfassung für Diss:
%{For continuous-time Markov processes,} apart from the L\'evy process case {and} as far as known to us, estimation of {their} L\'evy kernel has been confined to the special case of Markov step processes. For these, there exists a one-to-one correspondence between the L\'evy kernel and the infinitesimal generator. On the one hand, efficient non-pa\-ra\-me\-tric estimation of Markov step process models has been studied by \citet{Greenwood1994}. They assume the mean holding times to be bounded, and the transition kernel to be uniformly ergodic. This excludes the null recurrent case. On the other hand, the work on parametric estimation of Markov step processes is more exhaustive. The null recurrent case has been studied, for instance, by \citet{Hoepfner1993}. There, the process is observed up to a random stopping time such that a deterministic amount of information (or more) has been discerned. Local asymptotic normality is shown in various situations. With a slightly different aim, in contrast, \citet{Hoepfner1990} considers Markov step processes observed up to a deterministic time. Accordingly, the observed amount of information is random. Local asymptotic mixed normality (of statistical experiments) is shown under Darling--Kac's condition. Here, we utilise some of their results and methods; see also \citet{Touati1987} and \citet{Hoepfner2003}.
%:   - Summary
We improve upon the restrictions within the aforementioned literature: First and foremost, we do not restrict ourselves to Markov step processes. Second, we consider processes, null recurrent in the sense of Harris, in a non-pa\-ra\-me\-tric setting. Third, we address the influence of observations on a discrete time grid.

%:- Outline
We briefly outline {our} paper. {In \cref{D Estimation} we study the estimation of the L\'evy kernel based on discrete observations. Split into three subsections, we present the statistical problem with our standing assumptions; we give our estimator along with a bias correction; and state our main results --~the estimator's consistency and the central limit theorem. In \cref{C Estimation}, we study the case where continuous-time observations are available. This section is organised analogously to \cref{D Estimation}. The corresponding proofs are in \cref{C Proofs}. The proofs for our main results of \cref{D Estimation} are in \cref{D Proofs}. Each proofs section comes with its own short outline at its beginning.} Since we bring together potential theoretic aspects of Markov processes with functional and martingale limit theory, we put some of our technical considerations off to \cref{C OnZ}.
\section[Estimation from high-frequency observations]{Density estimation of the L\'evy kernel from high-frequency observations}\label{D Estimation}
%:------------------------------------------------------------------------
\subsection{Preliminaries and assumptions}\label{D Setting}
%:- Observation scheme
On the filtered probability space(s) $(\Omega,\mathscr{F}, (\mathscr{F}_{t})_{t\ge0},(\p^{x})_{x\in E})$, let $X=(X_t)_{t\ge0}$ be a Markovian Itô semi-martingale with values in Euclidean space $E=(\bbr^d, \mathscr{B}^d)$, or a subset thereof, such that $\p^x(X_0=x)=1$ for all $x$. For $n\in\bbn$ and $\Delta>0$, we observe $X_0(\omega)$ and the increments
\begin{gather}\label{D eq: Observations}
	\Delta^n_kX(\omega) \coloneqq X_{k\Delta}(\omega) - X_{(k-1)\Delta}(\omega) \quad k=1,\dotsc,n.
\end{gather}
We emphasise that the jumps of the process are latent. 

%:- Notation
Throughout this paper, we use the following notation: We abbreviate $E^\ast\coloneqq E\setminus\{0\}$. {We denote the Dirac measure at $x$ by $\epsilon_x$.} For $\pi$ an {(initial) probability} on $E$, we denote the expectation \wrt\ the law $\p^\pi\coloneqq \int \pi(\dx)\p^x$ by $\E^\pi$. For $\alpha\ge0$ and $A\subseteq E$, in addition, $\mathcal{C}^\alpha_\loc(A)$ denotes the class of all continuous functions on $A$ which are $\lfloor \alpha\rfloor$-times continuously differentiable such that every $x\in A$ has a neighbourhood on which the function's (partial) $\lfloor \alpha\rfloor$-de\-riv\-a\-tives are uniformly Hölder of order $\alpha-\lfloor\alpha\rfloor$. 

%:- Itô semi-martingale
The characteristics $(B,C,\mathfrak{n})$ of $X$ are absolutely continuous with respect to Lebesgue measure; there are mappings $b: E \to E$ and $c: E \to E\otimes E$ {(with $c=\sigma\sigma\transpose$ in view of  \cref{J eq: MarkovSDE})}, and a kernel $F$ on $E$ with $F(x,\{0\})=0$ such that
\begin{gather}\label{D eq: Ito semimartingale}
	B_t = \int_0^t b(X_s)\ds, \quad C_t = \int_0^t c(X_s)\ds, \quad \text{and}\quad \mathfrak{n}(\dt,\dy) = \dt\otimes F(X_t,\dy).
\end{gather}
The integer-valued random measure $\sum_{\{s:\Delta X_s\neq0\}} \epsilon_{(s,\Delta X_s)}(\dt,\dy)$ on $\bbrp\times\bbr^d$ is called the process's \emph{jump measure}. The random measure {$\mathfrak{n}$} is its predictable compensator: For every Borel function $g: E\times E \to \bbrp$, (inital) probability $\pi$, and $t>0$, we have
\begin{gather}\label{D eq: Levy system}
	\E^\pi \sum_{0<s\le t}g(X_{s-}, \Delta X_{s})\mathbbm{1}_{\{X_{s-}\neq X_{s}\}}  = \E^\pi\int_{0}^t \ds \int_{E}F(X_{s},\dy)g(X_{s},y).
\end{gather}
We call $F$ the \emph{L\'evy kernel}. It is unique outside a set of potential zero. We assume it admits a density $(x,y) \mapsto f(x,y)$ which we want to estimate.

Throughout, we work under the following technical hypothesis on the characteristics: 
%:   ! Assumption: Technical Condition
\begin{Assumption}~\label{D a: BCn}
	\begin{enumerate}
		\item The process $X$ satisfies the following (linear) growth condition: There exists a constant $\zeta<\infty$ and a L\'evy measure $\bar F$ on $E$ such that
		\begin{gather*}
			\lVert b(x) \rVert \le \zeta(1+\lVert x\rVert), \quad  \lVert c(x) \rVert  \le \zeta(1+\lVert x\rVert^2), \mAnd  F(x,A) \le (1+\lVert x\rVert)\bar F(A)
		\end{gather*}
		holds for all $x\in E$ and every Borel set $A\subseteq E$. We denote by $\beta\in[0,2]$ some constant such that $\int \bar F(\dw)(\lVert w\rVert^\beta \wedge 1)) < \infty$.
		\item The L\'evy measure $\bar F$ admits a density $\bar f$ which is continuous on $E^\ast$. 
		\item There exists a constant $\zeta<\infty$ such that $\sup_{\lVert z\rVert > 1} \lVert z\rVert\bar f(z) \le \zeta$.
	\end{enumerate}
\end{Assumption}
\begin{Remark}
	Apart from the growth condition, there is no assumption on $b$ and $c$. Whether $X$ is a weak or a strong solution of \cref{J eq: MarkovSDE} is irrelevant to us.
\end{Remark}

%:- Assumptions
We impose assumptions on the recurrence of~$X$ and on the smoothness of $f$. To obtain consistency for our estimator below, we impose: 
%:   ! Assumption: Harris recurrence
\begin{Assumption}\label{D a: HR}\label{C a: HR}
	The process $X$ is Harris recurrent: {On $E$,} there exists a $\sigma$-finite, {invariant} measure $\mu$ {for $X$ such that}, for every Borel set $A\subseteq E$, we have
	\[
		\mu(A) > 0 \implies \p^x\left( \int\nolimits_{0}^\infty \mathbbm{1}_{A}(X_{s}) \ds = \infty \right) = 1\quad\forall x\in E.
	\]
\end{Assumption}
%:   ! Assumption: smoothness (weak)
\begin{Assumption}\label{D a: C-weak}
	For some $\alpha>0$, the L\'evy kernel admits a density $f\in\mathcal{C}^\alpha_\loc(E\times E^\ast)$; and the invariant measure from \cref{D a: HR} admits a continuous density $\mu'$.
\end{Assumption}
\begin{Remark}
{Harris recurrence can be verified, for instance, by virtue of a Foster--Lyapunov type criteria \citep[see][]{Meyn1993a}. Moreover, for the existence of a smooth density $\mu'$ it is sufficient that the marginal distributions of $X_t$ admit a smooth density. We refer, for instance, to \citet{Picard1996} for criteria for the latter.}
\end{Remark}
To obtain a central limit theorem, we also impose:
%:   ! Assumption: Darling-Kac
\begin{Assumption}\label{D a: DK}\label{C a: DK}
	The process $X$ satisfies the following Darling--Kac condition: For some $0 < \delta \le 1$, there exists a function $v:\bbrp\to\bbrp$ -- at infinity, regularly varying of index $\delta$ -- such that, for every $\mu$-integrable $g$,
	\begin{gather}\label{D eq: Darling Kac}
		\frac{1}{v(1/\lambda)} \int_{0}^\infty \e^{-\lambda t} \E^x[g(X_t)] \dt \to \mu(g) \quad \mu\text{-a.\,e. as } \lambda\downarrow 0.
	\end{gather}
\end{Assumption}
\begin{Remark}
	In the positive recurrent case (that is, when $\mu$ is finite), \cref{D a: DK} indeed is satisfied for $\delta=1$ and with $v(t)=t/\mu(E)$. We refer the interested reader to \citet{Touati1987} and to \citet{Hoepfner2003}.
	% (We also refer to the example at the end of section 2.1).
\end{Remark}
%:   ! Assumption: smoothness (strong)
\begin{Assumption}\label{D a: C-strong}
	For some $\alpha_1,\alpha_2\ge2$, the L\'evy kernel admits a density $f$ which is twice continuously differentiable on $E\times E^\ast$ such that $x\mapsto f(x,y) \in \mathcal{C}^{\alpha_1}_\loc(E)$ for all $y\in E^\ast$, and $y\mapsto f(x,y) \in \mathcal{C}^{\alpha_2}_\loc(E^\ast)$ for all $x \in E$; and the invariant measure from \cref{D a: HR} admits a continuous density $\mu'$ which is $(\lceil\alpha_1\rceil-1)$-times continuously differentiable.
\end{Assumption}

\begin{Example}
	Suppose that $f$ is bounded and vanishes outside $\{\lVert x\rVert\le1,\lVert y\rVert\le1\}$; that is, there are neither jumps with left-limit outside the unit ball nor jumps of size bigger than one. Then our process's recurrence (or transience) is completely determined by drift and volatility. For instance:
\begin{enumerate}
	\item If the volatility $\sigma$ vanishes everywhere and the drift satisfies $b(x) = -x$, then $X$ is positive recurrent.
	\item If the drift $b$ vanishes everywhere, and the volatility satisfies $\sigma(x)=1$, then $X$ is not positive. In fact, $X$ has the recurrence (or transience) of Brownian motion: In the univariate case, $X$ is null recurrent and Darling--Kac's condition holds with $\delta=1/2$; in the bivariate case, $X$ is null recurrent and Darling--Kac's condition fails; and in all other multivariate cases, $X$ is transient.
\end{enumerate}
\end{Example}

%:------------------------------------------------------------------------
\subsection{Kernel density estimator}\label{D Estimator}
In principle, we are free to choose our favourite estimation method, \eg, the method of sieves with projection estimators. Here, however, we introduce a kernel density estimator as it allows for a more comprehensible presentation of the proofs. Also, the method is well-understood in the context of classical (conditional) density estimation.

An outline: First, we choose smooth kernels $g_1$ and $g_2$ with support $B_1(0)$ (the unit ball centred at zero) which are, at least, of order $\alpha_1$ and $\alpha_2$, respectively; that is, for every multi-index $m=(m_1,\dotsc,m_d) \in\bbn^d\setminus\{0\}$ and each $i\in\{1,2\}$, we have
\begin{gather}\label{D eq:Kernel-g1}
	 |m| \coloneqq m_1+\dotsb+m_d < \alpha_i \implies \kappa_m(g_i)\coloneqq \int  x_1^{m_1} \cdot\dotsb\cdot  x_d^{m_d} g_i(x)\dx=0.
\end{gather}
Second, we choose a bandwidth vector $\eta = (\eta_1,\eta_2)>0$. Last, we construct an estimator for $f(x,y)$ using the kernels $g_i^{\eta,x}(z) \coloneqq \eta_i^{-d} g_i((z-x)/\eta_i)$.  If the bandwidth is chosen appropriately, we achieve a consistent estimator which follows a central limit theorem.
%:   ! Definition: Kernel Estimator
\begin{Definition}\label{D def:PE}
For $\eta=(\eta_1,\eta_2)>0$, we call $\hat f^{\Delta,\eta}_{n}$ defined by
\begin{gather}\label{D eq: Estimator}
	\hat f^{\Delta,\eta}_{n}(x,y) \coloneqq
	\begin{cases}
		\frac{\sum_{k=1}^{n} g_1^{\eta,x}(X_{(k-1)\Delta})g_2^{\eta,y}(\Delta^n_k X)}{\Delta\sum_{k=1}^{n}g_1^{\eta,x}(X_{(k-1)\Delta})} & \text{if } \sum_{k=1}^{n}g_1^{\eta,x}(X_{(k-1)\Delta})>0, \\
		0 & \text{otherwise},
	\end{cases}
\end{gather}
the \emph{kernel density estimator} of $f$ (\wrt\ bandwidth $\eta$ based on $X_0,X_\Delta,\dotsc,X_{n\Delta}$).
\end{Definition}

%In the classical context, the (partial) derivatives of a consistent kernel density estimator --~provided they exist~-- are consistent for the (partial) derivatives of the estimated density. 
{In analogy to classical conditional density estimation, we also introduce a bias correction for our estimator.}
%:   ! Definition: Bias correction
\begin{Definition}\label{D def: BC}
For $\eta=(\eta_1,\eta_2)>0$, we call $\hat \gamma^{\Delta,\eta}_n$ defined by 
\begin{gather*}
		\hat \gamma^{\Delta,\eta}_n(x,y) \coloneqq {} 
		\begin{cases}
			\begin{aligned}
				\lefteqn{\eta_{1}^{\alpha_1}\sum_{\substack{|m_1+m_2|=\alpha_1\\|m_2|\neq0}}\frac{\kappa_{m_1+m_2}(g_1)}{m_1!m_2!} \frac{\sum_{k=1}^{n} \frac{\partial^{m_1}}{\partial x^{m_1}} g_1^{\eta,x}(X_{(k-1)\Delta})}{\sum_{k=1}^{n}g_1^{\eta,x}(X_{(k-1)\Delta})}\frac{\partial^{m_2}}{\partial x^{m_2}}\hat f^{\Delta,\eta}_n(x,y)} \\
				&\hspace{0.5em}+ \eta_{2}^{\alpha_2} \sum_{|m|=\alpha_2} \frac{\kappa_{m}(g_2)}{m!}\frac{\partial^m}{\partial y^m}\hat f^{\Delta,\eta}_n(x,y), && \text{if }
				\genfrac{}{}{0pt}{}{\sum_{k=1}^{n}g_1^{\eta,x}(X_{(k-1)\Delta}) > 0}{\alpha_1,\alpha_2\in\bbn^\ast}, \\
				&0, && \text{otherwise},
				 \end{aligned}
		\end{cases}
\end{gather*}
the \emph{bias correction} for $\hat f^{\Delta,\eta}_n$. (The sums in the previous equation are over all multi-indices of {appropriate} length.)
\end{Definition}
%%:   ! Remark: Kernel and bandwidth choice
%\begin{Remark}
%	{As in classical conditional density estimation, the triweight kernel $z \mapsto 35/32 (1-z^2)^3$ is an appropriate, smooth enough, $2$nd-order kernel for the estimation of $f$.}
%	In general, it is favourable to choose different sets of kernels and a different bandwidth vector for the estimation of density derivatives.
%\end{Remark}

%:------------------------------------------------------------------------
\subsection{Consistency and central limit theorem}\label{D Results}
%:- Typographical conventions
Here, we present our main results. % of this section. 
%\emph{Deterministic equivalents} of Markov processes play a crucial role in the limit theory for our estimator; that are non-decreasing functions $v:\bbrp\to\bbrp$ such that the families
%\[
%	\left\{\mathscr{L}(v(t)^{-1}H_t \mid \p^\pi) : t>0\right\} \mAnd \left\{\mathscr{L}(v(t)H^{-1}_t \mid \p^\pi) : t>0\right\}
%\] 
%are tight for every probability $\pi$ on $E$ and every non-decreasing additive functional $H$ of $X$ with $0<\E^\mu\!H_1<\infty$.
%We emphasise the following consequence of Théorème 3 of \citet{Touati1987}:
%Under Darling--Kac's condition (\cref{a: DK}), the function $v$ in \cref{eq: Darling Kac} is a deterministic equivalent of $X$. For every $H$ as before, furthermore, we have that $(v(t)^{-1}H_{st})_{s\ge0}$ converges in law to a non-trivial process as $\tto$. For Markov processes violating Darling--Kac's condition, the latter convergence may not hold. Nevertheless, \citet{Loecherbach2008} showed that some {deterministic equivalent} already exists when $X$ is Harris recurrent.
%
We agree to the following conventions: {Under \cref{D a: HR,D a: DK}, $v$ denotes the regularly varying function given in \cref{D eq: Darling Kac}. Under \cref{D a: HR} only, 
$v$ denotes an arbitrary deterministic equivalent (see \cref{J def: DE} below) of the Markov process $X$. }
For typographical reasons, we may write $v_t$ for $v(t)$ or $X(t)$ for $X_t$ etc.\ as convenient. 

%:- Conditions
We utilise the following conditions as $n\Delta\to\infty$ and $\Delta\to0$, where $0\le\zeta_1,\zeta_2<\infty$:
\begin{align}
		v_{n\Delta}\eta_{1,n}^d\eta_{2,n}^d \to \infty, & \mAnd \eta_{1,n}\to0, \quad\eta_{2,n}\to 0; \label{D eq: LLN}\\
		v_{n\Delta}\eta_{1,n}^{d+2\alpha_1}\eta_{2,n}^d \to \zeta_1^2, &\mAnd    v_{n\Delta}\eta_{1,n}^{d}\eta_{2,n}^{d+2\alpha_2} \to \zeta_2^2; \label{D eq: CLT}
\end{align}
In addition, we also utilise the following conditions due to discretisation, where $\zeta<\infty$ is independent of $n$:
\begin{subequations}\label{D eq: Discrete}
\begin{align}
		\Delta\eta_{1,n}^{-2 - d[(1-2/(\beta+d))\vee0]} \to 0, &\mAnd  \Delta\eta_{2,n}^{-2\vee (\beta+d)}\to 0; \label{D eq: LLNb} \\
	n\Delta^2\eta_{1,n}^d\eta_{2,n}^d \le \zeta, & \quad v_{n\Delta}\Delta^2\eta_{1,n}^{d - 4 - 2d[(1-2/(\beta+d)\wedge0]}\eta_{2,n}^d \to 0, \label{D eq: CLTb} \\
		&\mAnd v_{n\Delta}\Delta^2\eta_{1,n}^d\eta_{2,n}^{d-4\vee 2(\beta+d)}\to 0. \label{D eq: CLTc}
\end{align}
\end{subequations} 	
\vspace{-1.75\baselineskip}
\begin{Remark}
	If $\Delta\to0$ fast enough, then \cref{D eq: LLN} and \cref{D eq: CLT} are the crucial conditions.
\end{Remark}

%:   ! Theorem: Consistency
\begin{Theorem}\label{D t: LLN}
Grant \cref{D a: BCn,D a: C-weak,D a: HR}. Let $\eta_n=(\eta_{1,n},\eta_{2,n})$ be such that \cref{D eq: LLN} and \cref{D eq: LLNb} hold. Moreover, let $(x,y)\in E\times E^\ast$ be such that $\mu'(x)>0$ and $F(x,E)>0$.
\begin{enumerate}
	\item  If $n\Delta^2\to0$, then, under any law $\p^\pi$, we have the following convergence in probability:
\begin{gather}\label{D eq: Consistency}
	\hat f_{n}^{\Delta,\eta_n}(x,y) \limpPiN f(x,y).
\end{gather}
	\item Grant \cref{D a: DK} in addition. If $(n\Delta)^{1-\delta}\Delta\to 0$, then, under any law $\p^\pi$, \cref{D eq: Consistency} holds as well.
\end{enumerate}
\end{Theorem}
%:   ! Remark: Practical limitations
\begin{Remark}
By this theorem, our estimator is consistent for every $x$ and $y \neq 0$ if $n\Delta\to\infty$ and $\Delta\to0$. In practice, however, both $n$ and $\Delta$ are given! Then, for instance, if a continuous martingale component is present, {or if there are infinitely many jumps over finite time intervals,} our estimator is unreliable for all $y$ close to the origin. To illustrate this important point, suppose that $X$ is a univariate process with constant volatility $\sigma^2>0$. Increments with absolute value less than $\zeta\sigma\Delta^{1/2}$, where $\zeta$ is quite a large constant (e.\,g., $\zeta=5$), are predominantly due to the continuous martingale and not due to jumps. On the set $\{y: |y|\le \zeta\sigma\Delta^{1/2}\}$, therefore, our estimator $\hat f^{\Delta,\eta}_n(x,\cdot)$ is unreliable regardless of the chosen bandwidth $\eta$.
\end{Remark}

%:- CLT notation
For the next theorem, we establish additional notation. For $0 < \alpha < 1$, let $K$ denote the $\alpha$-stable L\'evy subordinator with Laplace transform $\E \e^{-\xi K_t} = \e^{-t\xi^\alpha}$ for $\xi,t\ge 0$. Its right inverse $L_{t}\coloneqq\inf\{s>0: K_{s}> t\}$ is called the \emph{Mittag-Leffler process of order $\alpha$}. By abuse of notation, we call $L_t=t$ the \emph{Mittag-Leffler process of order $1$}. On an extension
\begin{gather}\label{C eq:ExtendedPSpace}
	(\tilde\Omega,\tilde{\mathscr{F}}, \tilde\p) \coloneqq (\Omega\times\Omega',\mathscr{F}\otimes\mathscr{F}',\p^\pi\otimes\p')
\end{gather}
of the probability space, let $V=(V(x,y))_{x\in E,y\in E^\ast}$ be a standard Gaussian white noise random field (that is, the finite dimensional marginals of $V$ are i.\,i.\,d.\ standard normal) and let $L=(L_t)_{t\ge0}$ be the Mittag-Leffler process of order $\delta$ (from \cref{D a: DK}) such that $V$, $L$ and $\mathscr{F}$ are independent. In the theorem below, convergence holds \emph{stably in law}; that is,  pre-limiting and limiting random variables are defined on the extended space \cref{C eq:ExtendedPSpace} and we have joint convergence in law of our pre-limiting random variables with any bounded, $\mathscr{F}$-measurable random variable. This notion, labelled $\mathscr{L}\mathrm{-st}$, is due to \citet{Renyi1963}.

%:   ! Theorem: CLT
\begin{Theorem}\label{D t: CLT}
Grant \cref{D a: BCn,D a: C-weak,D a: HR,D a: DK,D a: C-strong}. Let $\eta_n=(\eta_{1,n},\eta_{2,n})$ be such that \cref{D eq: LLN} and \cref{D eq: Discrete} hold, and let $(x_i,y_i)_{i\in I}$ be a finite family of pairwise distinct points in $E\times E^\ast$ such that $\mu'(x_i)>0$ and $F(x_i,E)>0$ for each $i\in I$.  If $(n\Delta)^{1-\delta}\Delta \to 0$, then, under any law $\p^\pi$, we have the following stable convergence in law:
\begin{gather*}
	\left( \sqrt{v_{n\Delta}\eta_{1,n}^d\eta_{2,n}^d} \big(\hat f_{n}^{\Delta,\eta_n}(x_{i},y_{i}) - \frac{\mu(g_1^{\eta_n,x_i}Fg_2^{\eta_n,y_i})}{\mu(g_1^{\eta_n,x_i})}\big) \right)_{i\in I}
		\limdstN
	\left( \frac{\sigma(x_{i},y_{i})}{\sqrt{L_{1}}}V(x_{i},y_{i}) \right)_{i\in I},
\end{gather*}
where the asymptotic variance is given by
\begin{gather}\label{D eq: Variance}
	\sigma(x,y)^2 \coloneqq \frac{f(x,y)}{\mu'(x)}\int g_1(w)^2\dw\int g_2(z)^2\dz.
\end{gather}

In addition, let $\eta_n$ be such that \cref{D eq: CLT} holds as well. Suppose either that $\alpha_1,\alpha_2\in\bbn^\ast$ or that $\zeta_1=\zeta_2=0$ in \cref{D eq: CLT}. Then, under any law $\p^\pi$, we have the following stable convergence in law:
\begin{gather*}\label{D eq: AsymptoticMixedNormal}
	\left(
		\sqrt{v_{n\Delta}\eta_{1,n}^d\eta_{2,n}^d} \big(\hat f^{\Delta,\eta_{n}}_{n}(x_{i},y_{i}) - f(x_{i},y_{i})\big)
	\right)_{i\in I}
	\limdstT
	\left(
		\gamma(x_i,y_i) + \frac{\sigma(x_{i},y_{i})}{\sqrt{L_{1}}}V(x_{i},y_{i})
	\right)_{i\in I},
\end{gather*}
where -- in the former case -- the asymptotic bias $\gamma(x,y)$ is given by
\begin{gather}\label{D eq: Bias}
	\begin{aligned}
		\gamma(x,y) = {} & \frac{\zeta_1}{\mu'(x)}\sum_{\substack{|m_1+m_2|=\alpha_1\\|m_2|\neq0}}\frac{\kappa_{m_1+m_2}(g_1)}{m_1!m_2!}\frac{\partial^{m_1}}{\partial x^{m_1}}\mu'(x)\frac{\partial^{m_2}}{\partial x^{m_2}}f(x,y) \\
		& \hspace{15em} + \zeta_2\sum_{|m|=\alpha_2} \frac{\kappa_{m}(g_2)}{m!}\frac{\partial^m}{\partial y^m}f(x,y),
	\end{aligned}
\end{gather}
and -- in the latter case -- $\gamma(x,y)=0$.
\end{Theorem}
%:   ! Remark: Bias and Variance
\begin{Remark}
The asymptotic bias and variance of our estimator are analogous to those of the Nadaraya--Watson estimator in classical conditional density estimation: $\kappa_m(g_i)$ and $\int g_i(z)^2\dz$ are the relevant \emph{moment} and the \emph{roughness} of the kernel $g_i$, respectively; and $f$ (resp., $\mu'$) plays the role of the conditional (resp., marginal) density.
\end{Remark}

%:- Rate of convergence
We recall that $v$ from \cref{D eq: Darling Kac} satisfies $v_t = t$ in the ergodic case, and $v_t = t^\delta\ell(t)$ for some slowly varying function $\ell$ in the null recurrent case. If we choose $\eta_{i,n} = v_{n\Delta}^{-\xi_i}$ with $\xi_1 = \alpha_2/[d(\alpha_1+\alpha_2) + 2\alpha_1\alpha_2]$ and $\xi_2 = \alpha_1/[d(\alpha_1+\alpha_2) + 2\alpha_1\alpha_2]$, then \cref{D eq: LLN} and \cref{D eq: CLT} hold with $\zeta_1=\zeta_2=1$. {If $\Delta\to0$ fast enough such that $n\Delta^{1 + [d(\alpha_1+\alpha_2)+2\alpha_1\alpha_2]/\zeta}\to0$ in addition, where $\zeta$ denotes the maximum of $(1-\delta)d(\alpha_1+\alpha_2)+2\alpha_1\alpha_2$, $\delta\alpha_1(\alpha_2+2+d)$ and $\delta\alpha_2(\alpha_1+2 + d^2/(2+d))$, then our choice of $\eta_n$ also satisfies \cref{D eq: Discrete} for every $\beta\le2$. Consequently,} our estimator's rate of convergence is
\begin{gather}\label{D eq: rate}
	v_{n\Delta}^{\alpha_1\alpha_2/[d(\alpha_1+\alpha_2) + 2\alpha_1\alpha_2]}.
\end{gather}
In the case  $\alpha_1=\alpha_2 {\eqqcolon \alpha}$, the achieved rate $v_{n\Delta}^{\alpha/(2\alpha+2d)}$ equals the non-pa\-ra\-me\-tric minimax rate of smooth density estimation, related to the smoothness of $f$ as a $2d$-di\-men\-sio\-nal function and \wrt\ $v_{n\Delta}$. 

%%:- Comparison to Benchmark
%{We compare this result to \cref{C theo:CLT}, our benchmark. First, we remark that our estimator's asymptotic bias and variance are equal to those of our benchmark estimator; and they are analogous to those of the Nadaraya--Watson estimator in classical conditional density estimation (see the remark below \cref{C theo:CLT}). Second,} if we choose $\eta_{i,n} = v_{n\Delta}^{-\xi_i}$ with $\xi_1 = \alpha_2/[d(\alpha_1+\alpha_2) + 2\alpha_1\alpha_2]$ and $\xi_2 = \alpha_1/[d(\alpha_1+\alpha_2) + 2\alpha_1\alpha_2]$ {again,} then \cref{D eq: LLN} and \cref{D eq: CLT} hold with $\zeta_1=\zeta_2=1$. If $\Delta\to0$ fast enough such that our choice of $\eta_n$ also satisfies \cref{D eq: Discrete}, then the rate of convergence of our estimator is
%\[
%	v_{n\Delta}^{\alpha_1\alpha_2/[d(\alpha_1+\alpha_2) + 2\alpha_1\alpha_2]}.
%\]
%{We note that this rate is equivalent to the rate of our benchmark estimator, and} equals the non-pa\-ra\-me\-tric minimax rate of smooth density estimation, related to the smoothness of $f$ as a $2d$-di\-men\-sio\-nal function and \wrt\ $v_{n\Delta}$. {Third, we observe that our remarks on the bandwidth choice hold analogously. In particular:}
%:   ! Remark: Bandwidth Selection
\begin{Remark}
{Bandwidth selection has always been a crucial issue in these kind of studies. Although orders of magnitude are crucial from an asymptotic point of view and $\eta_{i,n}=(n\Delta)^{-\xi_i}$ for some $\xi_i>0$ may be a good choice , we note that, in practice, $\eta_{i,n}= \zeta(n\Delta)^{-\xi_i}$ with leading constant $\zeta\neq1$ could be a better one. A detailed analysis would go beyond the scope of this paper. 
We briefly comment on two problems: How to choose the bandwidths manually such that conditions (\ref{D eq: LLN}--\ref{D eq: Discrete}) are satisfied for the unknown $v_{n\Delta}$, $\alpha_1$, $\alpha_2$ and $\beta$? What needs to be considered when employing data-driven methods for selecting optimal bandwidths?}
\begin{enumerate}
	\item Let $\alpha_0\ge2$ and $0<\delta_0\le1$ such that $\delta_0 > d/(d + \alpha_0)$. If we choose $\eta_{i,n} = (n\Delta)^{-1/(2d+2\alpha_0)}$, then \cref{D eq: LLN} and \cref{D eq: CLT} hold for all processes $X$ such that  \cref{D a: C-strong,D a: DK} hold for some $\alpha_1,\alpha_2\ge\alpha_0$ and $\delta_0 < \delta \le 1$. {If $\Delta\to0$ fast enough such that $n\Delta^{1 + 2[\alpha_0+d]/[\alpha_0+(2+d)\vee\alpha_0]}\to0$ in addition, then our {chosen bandwidth} also satisfies \cref{D eq: Discrete}. } 
	\item {The asymptotic bias and variance are proportional to the value of $f$ and its derivatives at the point of interest. The optimal bandwidth choice in terms of the asymptotic mean squared error, therefore, may depend heavily on $x$ and $y$. Especially for processes with infinite activity --~where $y\mapsto f(x,y)$ has a pole at zero~-- this is an important issue in practice; {cf.\  simulations in \cref{sec:Simul}.} In a future study on data-driven bandwidth selection methods like cross-validation, this distinction from estimating a bounded probability density has to be addressed carefully.}
\end{enumerate}
\end{Remark}

%%:   ! Remark: Uniform bandwidth choices 
%\begin{Remark}
%We note that $v_{n\Delta}$ is, in principle, as unknown to the practitioner as the smoothness of $f$ and the Blumenthal--Getoor index of some ``dominating'' L\'evy measure $\bar F$. Let $\alpha_0\ge2$ and $0<\delta_0\le1$ such that $\delta_0 > 2d/( 2d + 2\alpha_0)$. If we choose $\eta_{i,n} = (n\Delta)^{-1/(2d+2\alpha_0)}$, then \cref{D eq: LLN} and \cref{D eq: CLT} hold for all processes $X$ such that  \cref{D a: C-strong,C a: DK} hold for some $\alpha_1,\alpha_2>\alpha_0$ and $\delta_0 \le \delta \le 1$. If $\Delta\to0$ fast enough, then our {chosen bandwidth} also satisfies \cref{D eq: Discrete}.
%\end{Remark}

%:- Confidence Intervals
\cref{D t: CLT} does not allow for a direct construction of confidence intervals. For this purpose, we also obtain the following standardised version. 
%:   ! Corollary: Feasible CLT
\begin{Corollary}\label{D c: CLT}
Grant \cref{D a: BCn,D a: HR,D a: C-weak,D a: DK,D a: C-strong}. Let $\eta_n=(\eta_{1,n},\eta_{2,n})$ be such that (\ref{D eq: LLN}--\ref{D eq: Discrete}) hold. Suppose either that $\alpha_1,\alpha_2\in\bbn^\ast$ or that $\zeta_1=\zeta_2=0$ in \cref{D eq: CLT}.
Then under any law $\p^\pi$, we have the following stable convergence in law:
\begin{gather*}
	\left(
		\sqrt{\frac{\eta_{1,n}^d\eta_{2,n}^d \Delta\sum_{k=1}^{n}g_1^{\eta_n,x_i}(X_{(k-1)\Delta})}{\xi_g^2 \hat f_{n}^{\Delta,\eta_n}(x_{i},y_{i})}}  \big( [\hat f_{n}^{\Delta,\eta_n} - \hat \gamma^{\eta_n}_n - f](x_{i},y_{i})\big)
	\right)_{i\in I}
	\limdstN
	\big( V(x_i,y_i) \big)_{i\in I},
\end{gather*}
where $\xi_g^2=\int g_1(w)^2\dw\int g_2(z)^2\dz$.
\end{Corollary}
%:   ! Remark: On general models
\begin{Remark}
%	In contrast to \cref{C Estimation}, here, we have only dealt with Markovian Itô semi-martingales. 
	In principle, the results {of this section} are extendible to more general Markov models with L\'evy kernel $F$ such that \cref{D eq: Levy system} holds. 
	{In view of our proofs, the assumption that $X$ is an Itô semi-martingale is crucial for the analysis of the influence of discretisation (see \cref{D Discrete}). Suppose that an explicit upper bound for the small-time asymptotic ``error''
	\[
		\left| \frac{1}{\Delta}\E^x \left[g_2^{\eta,y}(\Delta_1^nX)\right] - \int F(x,\dw) g_2^{\eta,y}(w) \right|
	\]
	and an explicit sufficient condition which ensures
	\[
		\sup_{s\le 1} \frac{\xi_n}{v_{n\Delta}\eta_{1,n}^d} \left|{\Delta}\sum_{k=1}^{\lfloor sn\rfloor}h_n(X_{(k-1)\Delta})-\int_0^{\lfloor sn\rfloor\Delta} h_n(X_r)\dr\right| \limpPiN 0
	\]
	for $\xi_n = 1$ or $\xi_n^2 = v_{n\Delta}\eta_{1,n}^d\eta_{2,n}^d$ are available for some Markov process $X$. Then it is straightforward (see \cref{D l: P2a,D eq: P4X} --- \cref{D l: P1,D l: P1b,D l: P4}, respectively) to come up with sufficient conditions for \cref{D t: LLN,D t: CLT}, which replace \cref{D eq: Discrete}.
}
\end{Remark}

%:====================================================
\section[Estimation from continuous-time observations]{Density estimation of the Lévy kernel from continuous-time observations --- A benchmark}\label{C Estimation}
The L\'evy kernel of a Markov process is related with jumps. In fact, our estimator \cref{D eq: Estimator} uses $X_{(k-1)\Delta}$ and $\Delta^n_kX$ as proxies for the pre-jump value $X_{t-}$ and the jump size $\Delta X_t$ if, at a time $t\in[(k-1)\Delta,k\Delta]$, there is a jump from a neighbourhood of $x$ and of size close to $y$. Eventually, such time intervals contain either zero or one such jump; never more. Certainly, the statistical analysis simplifies if we observed the whole path of $X$; introducing proxies would be useless. So, despite observing the whole path of $X$ is somewhat unrealistic, it is theoretically important to study what happens in this case. This section is devoted to this question and can be viewed as a benchmark for what properties are achievable with a more realistic, discrete observation scheme. 
%:------------------------------------------------------------------------
\subsection{Preliminaries and assumptions}\label{C Setting}
%:- Observation scheme
On the filtered probability space(s) $(\Omega,\mathscr{F},(\mathscr{F}_{t})_{t\ge0},(\p^{x})_{x\in E})$, let $X = (X_{t})_{t\ge0}$ be a strong Markov process with values in Euclidean space $E=(\bbr^d,\mathscr{B}^d)$, or a subset thereof. Its sample paths are supposed to be \cadlag. We observe --~continuously in time~-- one sample path $\{X_{s}(\omega): s\in[0,t]\}$ for $t>0$; in particular, we discern all jumps.

%:- Notation
{In addition to the notation introduced before,} we use some classical notation from \citet{Getoor75}: We denote the {shift semi-group} on $\Omega$ by $(\theta_t)_{t\ge0}$ so that  $X_{t+s}=X_t\circ\theta_s$ for all $s,t\ge0$. We denote the {transition semi-group} of $X$ on $E$ by $(P_t)_{t\ge0}$. 
%And for $\pi$ an {(initial) probability}, $\E^\pi$ denotes the expectation \wrt\ the law $\p^\pi\coloneqq \int \pi(\dx)\p^x$. Furthermore, we abbreviate $E^\ast\coloneqq E\setminus\{0\}$ and, for $\alpha\ge0$ and $A\subseteq E$, we denote by $\mathcal{C}^\alpha_\loc(A)$ the class of all continuous functions on $A$ which are $\lfloor \alpha\rfloor$-times continuously differentiable such that every $x\in A$ has a neighbourhood on which the function's (partial) $\lfloor \alpha\rfloor$-de\-riv\-a\-tives are uniformly Hölder of order $\alpha-\lfloor\alpha\rfloor$. 

%:- Additive Functionals, Lévy system
A \emph{(perfect homogeneous) additive functional} $H$ of $X$ is an $\mathscr{F}_t$-adapted process such that $H_{t+s} = H_t \circ \theta_s + H_s$ for all $s,t\ge0$. A \emph{L\'evy system} $(F,H)$ of $X$ (in a wide sense) is a kernel $F$ on $E$ with $F(x,\{0\})=0$ and a non-decreasing additive functional $H$ of $X$ such that, for every Borel function $g: E \times E \to \bbrp$, probability $\pi$ on $E$, and $t>0$,
\begin{gather}\label{C eq:Levy-system}
	\E^\pi \sum_{0<s\le t}g(X_{s-}, \Delta X_{s})\mathbbm{1}_{\{X_{s-}\neq X_{s}\}}  = \E^\pi\int_{0}^t \mathrm{d} H_{s}\int_{E}F(X_{s},\dy)g(X_{s},y).
\end{gather}
The disintegration into $F$ and $H$ is by no means unique. For an appropriate reference function $g_0$ with $Fg_0(x)>0$, nevertheless, ratios of the form $Fg(x)/Fg_0(x)$ are unique outside a set of potential zero. In the cases where $X$ is quasi-left-continuous (that is, when all jump times are totally inaccessible) \citet{BenvenisteJacod73} proved the existence of a L\'evy system $(F,H)$ where $H$ is continuous. Such a process --~\cadlag, strong Markov, quasi-left-continuous~-- is called a \emph{Hunt process}.
%:   ! Remark: Levy system (Watanabe)
\begin{Remark}
	The continuity of the additive functional was included as a part of the original definition of L\'evy systems due to \citet{Watanabe1964}.
\end{Remark}

%:- Standing Assumption
Throughout this section, we work under the following hypothesis:
%:   ! Assumption: dH<<dt
\begin{Assumption}\label{C a:AC}
There exists a L\'evy system $(F,H)$ of $X$ where $H_t=t$.
\end{Assumption}
%:   ! Remark: Ito semi-martingale
Recalling \cref{D eq: Levy system}, we observe that all Markovian Itô semi-martingales satisfy \cref{C a:AC}. {In analogy to the semi-martingale case,} we call this $F$ in \cref{C a:AC} the \emph{(canonical) L\'evy kernel} of $X$. It is unique outside a set of potential zero. {Again,} we assume it admits a density $(x,y) \mapsto f(x,y)$ which we want to estimate.

%:- Assumptions
{Compared to \cref{D Estimation}, we slightly weaken the assumptions imposed on the smoothness of $f$.} To obtain consistency for our estimator below, we impose \cref{D a: HR} and:
%%:   ! Assumption: Harris recurrence
%\begin{Assumption}\label{C a: HR}
%The process $X$ is Harris recurrent: There exists a $\sigma$-finite measure $\mu$ on $E$ which is invariant \wrt\ $(P_t)_{t\ge0}$ and, for every Borel set $A\subseteq E$, we have
%\begin{gather*}
%	\mu(A) > 0 \implies \forall x\in E: \quad	\p^x\left( \int\nolimits_{0}^\infty \mathbbm{1}_{A}(X_{s}) \ds = \infty \right) = 1.
%\end{gather*}
%\end{Assumption}
%:   ! Assumption: Smoothness (weak)
\begin{Assumption}\label{C a: C-weak}
	The canonical L\'evy kernel admits a density $f$, continuous on $E\times E^\ast$; and the invariant measure from \cref{C a: HR} admits a continuous density $\mu'$.
\end{Assumption}
To obtain a central limit theorem, we also impose \cref{D a: DK} and:
%%:   ! Assumption: Darling-Kac
%\begin{Assumption}\label{C a: DK}
%The process $X$ satisfies the following Darling--Kac condition: For some $0 < \delta \le 1$, there exists a function $v:\bbrp\to\bbrp$ -- at infinity, regularly varying of index $\delta$ -- such that, for every $\mu$-integrable $g$,
%\begin{gather}\label{C eq: Darling Kac}
%	\frac{1}{v(1/\lambda)}\int_{0}^\infty \e^{-\lambda t}P_tg(x)\dt \to \mu(g) \quad \mu\text{-a.\,e. as } \lambda\downarrow 0.
%\end{gather}
%\end{Assumption}
%%:   ! Remark: Positive case
%\begin{Remark}
%In the positive recurrent case (that is, when $\mu$ is finite), \cref{C a: DK} indeed is satisfied for $\delta=1$ and with $v(t)=t/\mu(E)$. We refer the interested reader to \citet{Touati1987} and to \citet{Hoepfner2003}.
%\end{Remark}
%:   ! Assumption: Smoothness (strong)
\begin{Assumption}\label{C a: C-strong}
For some $\alpha_1,\alpha_2>0$, the canonical L\'evy kernel admits a density $f$ such that $x\mapsto f(x,y) \in \mathcal{C}^{\alpha_1}_\loc(E)$ for all $y\in E^\ast$, and $y\mapsto f(x,y) \in \mathcal{C}^{\alpha_2}_\loc(E^\ast)$ for all $x \in E$; and the invariant measure from \cref{C a: HR} admits a continuous density $\mu'$ which is $(\lceil\alpha_1\rceil-1)$-times continuously differentiable.
\end{Assumption}

%%:! Example: Darling-Kac vs. Drift / Volatility and Dimension
%\begin{Example}
%	Suppose that $X$ arises as a solution to \cref{J eq: MarkovSDE} such that $f$ is bounded and vanishes outside $\{\lVert x\rVert\le1,\lVert y\rVert\le1\}$; that is, there are neither jumps with left-limit outside the unit ball nor jumps of size bigger than one. Then its recurrence (or transience) is completely determined by drift and volatility. For instance:
%\begin{enumerate}
%	\item If the volatility $\sigma$ vanishes everywhere and the drift satisfies $b(x) = -x$, then $X$ is positive recurrent.
%	\item If the drift $b$ vanishes everywhere, and the volatility satisfies $\sigma(x)=1$, then $X$ is not positive. In fact, $X$ has the recurrence (or transience) of Brownian motion: In the univariate case, $X$ is null recurrent and Darling--Kac's condition holds with $\delta=1/2$; in the bivariate case, $X$ is null recurrent and Darling--Kac's condition fails; and in all other multivariate cases, $X$ is transient.
%\end{enumerate}
%\end{Example}

%:------------------------------------------------------------------------
\subsection{Kernel density estimator}\label{C Estimator}
{In \cref{D Estimator}, we introduced a kernel density estimator and its bias correction based on discrete observations. Here, we present corresponding versions which utilise the continuous-time observation scheme. We recall that $g_1$ and $g_2$ are kernels with support $B_1(0)$ which are, at least, of order $\alpha_1$ and $\alpha_2$, respectively. Given some bandwidth vector $\eta = (\eta_1,\eta_2)>0$, we utilise the kernels $g_i^{\eta,x}(z) = \eta_i^{-d} g_i((z-x)/\eta_i)$. 
}
%
%In principle, we are free to choose our favourite estimation method, \eg, the method of sieves with projection estimators. Here, however, we introduce a kernel density estimator as it allows for a more comprehensible presentation of the proofs. Also, the method is well-understood in the context of classical (conditional) density estimation.
%
%%:- Outline
%An outline: First, we choose kernels $g_1$ and $g_2$, with support $B_1(0)$ (the unit ball centred at zero) which are, at least, of order $\alpha_1$ and $\alpha_2$, respectively; that is, for every multi-index $m=(m_1,\dotsc,m_d) \in \bbn^d\setminus\{0\}$ and each $i\in\{1,2\}$, we have
%\begin{gather}\label{C eq:Kernel-g1}
%	 |m| \coloneqq m_1+\dotsb+m_d < \alpha_i \implies \kappa_m(g_i)\coloneqq \int  x_1^{m_1} \cdot\dotsb\cdot  x_d^{m_d} g_i(x)\dx=0.
%\end{gather}
%Second, we choose a bandwidth vector $\eta = (\eta_1,\eta_2)>0$. Last, we construct an estimator for $f(x,y)$ using the kernels $g_i^{\eta,x}(z) \coloneqq \eta_i^{-d} g_i((z-x)/\eta_i)$. If the bandwidth is chosen appropriately, we achieve a consistent estimator which follows a central limit theorem.
%:   ! Definition: Kernel Estimator
\begin{Definition}\label{C def:PE}
	For $\eta=(\eta_1,\eta_2)>0$, we call $\hat f^{\eta}_{t}$ defined by
	\begin{equation*}
		\hat f^{\eta}_{t}(x,y) \coloneqq
		\begin{cases}
				\frac{\sum_{0<s\le t} g_1^{\eta,x}(X_{s-})g_2^{\eta,y}(\Delta X_s)\mathbbm{1}_{\{X_{s-}\neq X_s\}}}{\int_0^t g_1^{\eta,x}(X_s)\ds} & \text{if } \int_0^t g_1^{\eta,x}(X_s)\ds>0, \\
			0 & \text{otherwise},
		\end{cases}
	\end{equation*}
	the \emph{kernel density estimator} of $f$ (\wrt\ bandwidth $\eta$ up to time $t$).
\end{Definition}

Our estimator in \cref{D def:PE} is the discretised analogue from the one presented here: In the numerator of the former, the jumps $\Delta X_t$ and the pre-jump left-limits $X_{t-}$ are replaced by the increments $\Delta^n_kX$ and the pre-increment values $X_{(k-1)\Delta}$, respectively. In the denominator, the sojourn time $\int_0^t g_1^{\eta,x}(X_s)\ds$ is replaced by its Riemann sum approximation $\Delta\sum_{k=1}^{n}g_1^{\eta,x}(X_{(k-1)\Delta})$. In analogy to \cref{D def: BC}, we also introduce a bias correction for our estimator: 

%%:- Derivative estimation
%In the classical context, it is well-known that the (partial) derivatives of a consistent kernel density estimator --~provided they exist~-- are consistent for the (partial) derivatives of the estimated density. 
%:   ! Definition: Bias correction
\begin{Definition}
For $\eta=(\eta_1,\eta_2)>0$, we call $\hat \gamma^{\eta}_t$ defined by 
\begin{gather*}
		\hat \gamma^{\eta}_t(x,y) \coloneqq {} 
		\begin{cases}
			\begin{aligned}
				\lefteqn{\eta_{1}^{\alpha_1}\sum_{\substack{|m_1+m_2|=\alpha_1\\|m_2|\neq0}}\frac{\kappa_{m_1+m_2}(g_1)}{m_1!m_2!} \frac{\int_0^t \frac{\partial^{m_1}}{\partial x^{m_1}} g_1^{ \eta,x}(X_s)\ds}{\int_0^t g_1^{\eta,x}(X_s)\ds}\frac{\partial^{m_2}}{\partial x^{m_2}}\hat f^{\eta}_t(x,y)} \\
				&\hspace{3em}+ \eta_{2}^{\alpha_2} \sum_{|m|=\alpha_2} \frac{\kappa_{m}(g_2)}{m!}\frac{\partial^m}{\partial y^m}\hat f^{\eta}_t(x,y), &\qquad& \text{if }
				\genfrac{}{}{0pt}{}{\int_0^t g_1^{\eta,x}(X_s)\ds > 0}{\alpha_1,\alpha_2\in\bbn^\ast}, \\
				&0, &\qquad& \text{otherwise},
				 \end{aligned}
		\end{cases}
\end{gather*}
the \emph{bias correction} for $\hat f^{\eta}_t$. %(The sums in the previous equation are over all multi-indices of {appropriate} length.)
\end{Definition}
%%:   ! Remark: Kernel and Bandwidth choice
%\begin{Remark}
%In general, it is favourable to choose a different set of kernels and a different bandwidth vector for the estimation of density derivatives.
%\end{Remark}

%:------------------------------------------------------------------------
\subsection{Consistency and central limit theorem}\label{C Results}
%:- Deterministic equivalents
Here, we present our results of this section. 
%\emph{Deterministic equivalents} of Markov processes play a crucial role in the limit theory for our estimator; that are non-decreasing functions $v:\bbrp\to\bbrp$ such that the families
%\[
%	\left\{\mathscr{L}(v(t)^{-1}H_t \mid \p^\pi) : t>0\right\} \mAnd \left\{\mathscr{L}(v(t)H^{-1}_t \mid \p^\pi) : t>0\right\}
%\] 
%are tight for every probability $\pi$ on $E$ and every non-decreasing additive functional $H$ of $X$ with $0<\E^\mu\!H_1<\infty$.
%We emphasise the following consequence of Théorème 3 of \citet{Touati1987}:
%Under Darling--Kac's condition (\cref{C a: DK}), the function $v$ in \cref{D eq: Darling Kac} is a deterministic equivalent of $X$. For every $H$ as before, furthermore, we have that $(v(t)^{-1}H_{st})_{s\ge0}$ converges in law to a non-trivial process as $\tto$. For Markov processes violating Darling--Kac's condition, the latter convergence may not hold. Nevertheless, \citet{Loecherbach2008} showed that some {deterministic equivalent} already exists when $X$ is Harris recurrent.
%
%:- Typographical conventions
{We continue to use the notation and conventions from \cref{D Results}.}
% Throughout this section, \cref{C a:AC} is enforced. %Momentarily, $v: \bbrp\to\bbrp$ denotes a generic, non-decreasing function. 
%We agree to the following conventions: Under \cref{C a: HR}, $v$ denotes an arbitrary deterministic equivalent of $X$. Under the stronger \cref{C a: DK}, $v$ denotes the regularly varying function in \cref{D eq: Darling Kac}. For typographical reasons, we may write $v_t$ for $v(t)$ or $X(t)$ for $X_t$ etc.\ as convenient.

%:- Conditions
We utilise the following conditions as $\tto$, where $0 \le \zeta_1,\zeta_2 < \infty$:
\begin{align}
	v_t\eta_{1,t}^d\eta_{2,t}^d \to \infty, &\quad\text{and}\quad \eta_{1,t} \to 0, \eta_{2,t} \to 0; \label{C eq:LLN}\\
	%v_t\eta_t^{d}\vartheta_t^{d+2\alpha} \to \xi_1, &\quad\text{and}\quad v_t\eta_t^{d+2\gamma}\vartheta_t^d \to \xi_2,
	v_t\eta_{1,t}^{d+2\alpha_1}\eta_{2,t}^d \to \zeta^2_1, & \mAnd v_t\eta_{1,t}^{d}\eta_{2,t}^{d+2\alpha_2} \to \zeta^2_2. \label{C eq:CLT}
\end{align}
\vspace{-\baselineskip}
%:   ! Theorem: Consistency
\begin{Theorem}\label{C theo:LLN}
	Grant \cref{C a:AC,C a: C-weak,C a: HR}. Let $\eta_t=(\eta_{1,t},\eta_{2,t})$ be such that \cref{C eq:LLN} holds. Moreover, let $(x,y) \in E \times E^\ast$ be such that $\mu'(x)>0$ and $F(x,E)>0$. Then, under any law $\p^\pi$, we have the following convergence in probability:
	\begin{gather*}
		\hat f_t^{\eta_t}(x,y) \limpPiT f(x,y).
	\end{gather*}
\end{Theorem}

%:- CLT notation
%For the next theorem, we establish additional notation. For $0<\alpha<1$, let $K$ denote the $\alpha$-stable L\'evy subordinator with Laplace transform $\E\e^{-\xi K_t}=\e^{-t\xi^\alpha}$ for $\xi,t\ge0$. Its right inverse $L_{t}\coloneqq\inf\{s>0: K_{s}> t\}$ is called the \emph{Mittag-Leffler process of order $\alpha$}. By abuse of notation, we call $L_t=t$ the \emph{Mittag-Leffler process of order $1$}. On an extension
%\begin{gather}\label{C eq:ExtendedPSpace}
%	(\tilde\Omega,\tilde{\mathscr{F}}, \tilde\p) \coloneqq (\Omega\times\Omega',\mathscr{F}\otimes\mathscr{F}',\p^\pi\otimes\p')
%\end{gather}
%of the probability space, let $V=(V(x,y))_{x\in E,y\in E^\ast}$ be a standard Gaussian white noise random field (that is, the finite dimensional marginals of $V$ are i.\,i.\,d.\ standard normal) and let $L=(L_t)_{t\ge0}$ be the Mittag-Leffler process of order $\delta$ (from \cref{C a: DK}) such that $V$, $L$ and $\mathscr{F}$ are independent. In the theorem below, convergence holds \emph{stably in law}; that is,  pre-limiting and limiting random variables are defined on the extended space \cref{C eq:ExtendedPSpace} and we have joint convergence in law of our pre-limiting random variables with any bounded, $\mathscr{F}$-measurable random variable. This notion, labelled $\mathscr{L}\mathrm{-st}$, is due to \citet{Renyi1963}.
%:   ! Theorem: CLT
\begin{Theorem}\label{C theo:CLT}
Grant \cref{C a:AC,C a: C-weak,C a: HR,C a: DK}. Let $\eta_t=(\eta_{1,t},\eta_{2,t})$ be such that \cref{C eq:LLN} holds. Moreover, let $(x_i,y_i)_{i\in I}$ be a finite family of pairwise distinct points in $E\times E^\ast$ such that $\mu'(x_i)>0$ and $F(x_i,E)>0$ for each $i\in I$. Then, under any law $\p^\pi$, we have the following stable convergence in law:
\begin{gather*}
	\left(
		\sqrt{v_{t}\eta_{1,t}^d\eta_{2,t}^d} \left(\hat f^{\eta_{t}}_{t}(x_{i},y_{i}) - \frac{\mu(g_1^{\eta_t,x_i}Fg_2^{\eta_t,y_i})}{\mu(g_1^{\eta_t,x_i})}\right)
	\right)_{i\in I}
	\limdstT
	\left(
		\frac{\sigma(x_{i},y_{i})}{\sqrt{L_{1}}}V(x_{i},y_{i})
	\right)_{i\in I},
\end{gather*}
where the asymptotic variance is given by
\begin{gather}\label{C eq: Variance}
% 	\gamma(x,y) &\coloneqq \zeta_2\frac{1}{2}\sum_{j=1}^d \partial_{y_j}^{} ,\\
	\sigma(x,y)^2 \coloneqq \frac{f(x,y)}{\mu'(x)}\int g_1(w)^2\dw\int g_2(z)^2\dz.
\end{gather}

In addition, grant \cref{C a: C-strong} and let $\eta_t$ be such that \cref{C eq:CLT} holds as well. 
Suppose either that $\alpha_1,\alpha_2\in\bbn^\ast$ or that $\zeta_1=\zeta_2=0$ in \cref{C eq:CLT}. Then, under any law $\p^\pi$, we have the following stable convergence in law:
\begin{gather}\label{C eq: AsymptoticMixedNormal}
	\left(
		\sqrt{v_{t}\eta_{1,t}^d\eta_{2,t}^d} \big(\hat f^{\eta_{t}}_{t}(x_{i},y_{i}) - f(x_{i},y_{i})\big)
	\right)_{i\in I}
	\limdstT
	\left(
		\gamma(x_i,y_i) + \frac{\sigma(x_{i},y_{i})}{\sqrt{L_{1}}}V(x_{i},y_{i})
	\right)_{i\in I},
\end{gather}
where -- in the former case -- the asymptotic bias $\gamma(x,y)$ is given by
\begin{gather}\label{C eq: Bias}
	\begin{aligned}
		\gamma(x,y) = {} & \frac{\zeta_1}{\mu'(x)}\sum_{\substack{|m_1+m_2|=\alpha_1\\|m_2|\neq0}}\frac{\kappa_{m_1+m_2}(g_1)}{m_1!m_2!}\frac{\partial^{m_1}}{\partial x^{m_1}}\mu'(x)\frac{\partial^{m_2}}{\partial x^{m_2}}f(x,y) \\
		& \hspace{15em} + {\zeta_2}\sum_{|m|=\alpha_2} \frac{\kappa_{m}(g_2)}{m!}\frac{\partial^m}{\partial y^m}f(x,y),
	\end{aligned}
\end{gather}
and -- in the latter case -- $\gamma(x,y)=0$.
\end{Theorem}

%: - Comparison of results
{We compare \cref{C theo:CLT,D t: CLT}. First, we remark that the asymptotic bias and variance of $\hat f^{\Delta,\eta}_n$ are equal to those of our benchmark estimator $\hat f^\eta_t$. Second, if we choose $\eta_{i,t} = v_t^{-\xi_i}$ with $\xi_1 = \alpha_2/[d(\alpha_1+\alpha_2) + 2\alpha_1\alpha_2]$ and $\xi_2 = \alpha_1/[d(\alpha_1+\alpha_2) + 2\alpha_1\alpha_2]$ again, then \cref{C eq:LLN} and \cref{C eq:CLT} hold with $\zeta_1=\zeta_2=1$. The rate of convergence in \cref{C theo:CLT} is
\begin{gather}\label{C eq: rate}
	v_t^{\alpha_1\alpha_2/[d(\alpha_1+\alpha_2) + 2\alpha_1\alpha_2]};
\end{gather}
the rates \cref{D eq: rate} and \cref{C eq: rate} are equivalent. Third, we observe that our remark on the issue of bandwidth selection holds analogously. Last, we note that \cref{C theo:CLT} does not allow for a direct construction of confidence intervals just as \cref{D t: CLT}. In analogy to \cref{D c: CLT}, we also obtain the following standardised version. }
%:   ! Corollary: Feasible CLT
\begin{Corollary}\label{C cor:CLT}
Grant \cref{C a:AC,C a: HR,C a: C-weak,C a: DK,C a: C-strong}. Let $\eta_t=(\eta_{1,t},\eta_{2,t})$ be such that \cref{C eq:LLN} and \cref{C eq:CLT} hold. Suppose either that $\alpha_1,\alpha_2\in\bbn^\ast$ or that $\zeta_1=\zeta_2=0$ in \cref{C eq:CLT}.
Then under any law $\p^\pi$, we have the following stable convergence in law:
\begin{gather*}
	\left(
		\sqrt{\frac{\eta_{1,t}^d\eta_{2,t}^d\int_{0}^tg_1^{\eta_t,x_i}(X_{s})\ds}{\xi_g^2 \hat f^{\eta_t}_{t}(x_{i},y_{i})}}
		\big([\hat f^{\eta_t}_{t} - \hat \gamma^{\eta_t}_{t}](x_{i},y_{i}) - f(x_{i},y_{i})\big)
	\right)_{i\in I}
	\limdstT
	\big( V(x_i,y_i) \big)_{i\in I},
\end{gather*}
where $\xi_g^2=\int g_1(w)^2\dw\int g_2(z)^2\dz$.
\end{Corollary}
%%:   ! Remark: On Assumptions
%\begin{Remark}
%Let us briefly comment on our assumptions. First, suppose $\mu$ were not absolutely continuous but were to have an atom at $x\in E$. For instance, $x$ were a holding point. Then we deduce from our proofs (see \cref{C lemma:JnSn}) that the rate of convergence in the central limit theorem would be ${(v_t\eta^{d}_{2,t})^{1/2}}$. Second, suppose that \cref{C a:AC} failed but the additive functional $H$ were a local time at $x$. By a time-change argument, we observe that $\int_0^t g_1^{\eta,x}(X_s)\mathrm{d}H_s=H_t/\eta_1^d$ would be the correct denominator for our estimator. From our proofs (see \cref{C lemma:JnSn,C prop:ML-Convergence}) we observe that $H_t$ is asymptotically of the same order as $\int_0^tg_1^{\eta,x}(X_s)\ds$. Therefore, our estimator $\hat f^{\eta_t}_t(x,y)$ would diverge to infinity.
%\end{Remark}

%:====================================================
\section{Proofs for results of \texorpdfstring{\cref{C Estimation}}{Section 3}}\label{C Proofs}
{The notion of a deterministic equivalent of a Markov process plays a crucial role in the limit theory for our estimator.
\begin{Definition}\label{J def: DE}
A non-decreasing function $v:\bbrp \to \bbrp$ is called a \emph{deterministic equivalent} of the Markov process $X$ if the families
\[
	\left\{\mathscr{L}(v(t)^{-1}H_t \mid \p^\pi) : t>0\right\} \mAnd \left\{\mathscr{L}(v(t)H^{-1}_t \mid \p^\pi) : t>0\right\}
\] 
are tight for every probability $\pi$ on $E$ and every non-decreasing additive functional $H$ of $X$ with $0<\E^\mu\!H_1<\infty$.
\end{Definition}
We emphasise the following consequence of Théorème 3 of \citet{Touati1987}: Under Darling--Kac's condition, the function $v$ in \cref{D eq: Darling Kac} is a deterministic equivalent of $X$. 
For every $H$ as in \cref{J def: DE}, furthermore, we have that $(v(t)^{-1}H_{st})_{s\ge0}$ converges in law to a non-trivial process as $\tto$. For Markov processes violating Darling--Kac's condition, the latter convergence may not hold. Nevertheless, \citet{Loecherbach2008} showed that some {deterministic equivalent} already exists when $X$ is Harris recurrent.}

Throughout the proofs, we denote convergence of processes by double arrow (``$\Rightarrow$'') and understand it as convergence on the relevant Skorokhod space. For instance, we denote by $\mathcal{D}(\bbr^d)\coloneqq\mathcal{D}(\bbrp;\bbr^d)$ the space of all \cadlag\ functions from $\bbrp$ to $\bbr^d$ equipped with Skorokhod's topology. For a kernel $F$, a measurable function $g$, and a $\sigma$-finite measure $\nu$, the function $Fg$, the measure $\nu F$, and the number $\nu(g)$ are given by
\begin{gather*}
	F g(x) \coloneqq \int F(x,\dy)g(y), \quad \nu F(A) \coloneqq \int\nu(\dx)F(x,A), \quad 	\nu(g) \coloneqq \int \nu(\dx)g(x).
\end{gather*}
A kernel $F$ is called \emph{strong Feller} if $F g$ is in the class of continuous functions for every bounded $g$. %, then we call $F$ a \emph{strong Feller} kernel.

This section is organised as follows: First, in \cref{C Birkhoff} we prove a triangular array extension of Birkhoff's theorem for additive functionals. Second, in \cref{C AuxYZ} we introduce auxiliary Markov chains $Z$ and $Z^\prime$ derived from our Markov process $X$. We show that our result from \cref{C Birkhoff} applies to these chains. Some technicalities are put off to \cref{C OnZ}. Third, in \cref{C ProofLLN} we demonstrate a preliminary version of \cref{C theo:LLN} which depends only on $Z$ and $Z^\prime$; we conclude with the final steps in the proof of consistency.
Last, in \cref{C ProofCLT} we demonstrate a preliminary central limit theorem which depends only on $Z$ and $Z^\prime$; we conclude with the final steps in the proof of \cref{C theo:CLT,C cor:CLT}.

%:------------------------------------------------------------------------
\subsection{An extension of Birkhoff's theorem}\label{C Birkhoff}
The theorem presented in this subsection is the underlying key result for our proofs. It is a triangular array extension of Birkhoff's theorem for additive functionals
\citep[cf. Théo\-rème~II.2 of][]{Azema1967}. We prove a rather general version.
%:! Lemma: Additive Functional LLN
\begin{Theorem}\label{C lemma:AdditiveFunctional-LLN}
Let $Z = (Z_k)_{k\in\bbn^\ast}$ be a Markov chain with values in some state space $D$, with invariant probability $\psi$, and with transition kernel $\Psi$.
Assume that the state space is \emph{petite}, that is, there exist a probability $\rho$ on $\bbn^\ast$ and a non-trivial measure $\nu_\rho$ on $D$ such that, for every Borel set $A \subseteq D$,
\[
	\inf_{x\in D} \sum_{k=1}^\infty \rho(k)\Psi^k(x,A) \ge \nu_\rho(A).
\]
Let $(h_{n})_{n\in\bbn^\ast}$ be a sequence of functions such that $(\Psi h_{n})_{n\in\bbn^\ast}$ is uniformly bounded. Let $\xi_n>0$ be such that
\[
	n\xi_{n}\to\infty, \quad \xi_{n}^{-1}\psi(h_{n}) \to c <\infty, \quad (n\xi_{n}^2)^{-1}\psi(|h_n|)\to0 \quad \text{and}\quad (n\xi_{n}^2)^{-1}\psi(h^2_n)\to0
\]
as $\nto$. Then, under every law $\p^\pi$ for some probability $\pi$ on $D$, the following convergence holds uniformly on compacts in probability:
\begin{gather}\label{C eq:AF-LLN}
	G^n_s \LimucpN cs, \qquad \text{where} \quad G^n_{s} \coloneqq \frac{1}{n\xi_{n}} \sum_{k=1}^{\lfloor sn\rfloor} h_{n}(Z_{k}).
\end{gather}
\end{Theorem}
\begin{Remark}
If $(h_n)_{n\in\bbn^\ast}$ is non-negative (resp., uniformly bounded), then $n\xi_n\to\infty$ and $\xi_n^{-1}\psi(h_n)\to c<\infty$ already imply $(n\xi_n^2)^{-1}\psi(|h_n|)\to0$ (resp., $(n\xi_n^2)^{-1}\psi(h_n^2)\to0$).
\end{Remark}
\begin{proof}[of \cref{C lemma:AdditiveFunctional-LLN}]
Convergence in probability is equivalent to the property that --~given any subsequence~-- there exists a further subsequence which converges almost surely. By Proposition 17.1.6 of \citet{MeynTweedie}, therefore, 
it is sufficient to prove this theorem under the law $\p^{\psi}$ only.

For each $s\ge0$ and $n\in\bbn^\ast$, we observe $G^n_{s} = H^n_{s} + H^{\prime n}_{s}$, where
\[
	H^n_{s} = \frac{\lfloor sn\rfloor\psi(h_{n})}{n\xi_{n}}
	\mAnd
	H^{\prime n}_{s} = \frac{1}{n\xi_{n}}\sum_{k=1}^{\lfloor sn\rfloor} \big(h_{n}(Z_{k}) - \psi(h_{n})\big).
\]
By assumption, we have $H^n_{s} \to sc$ uniformly in $s$ as $\nto$. It remains to show that $H^{\prime n}_s$ converges to zero uniformly on compacts in probability.

We note $\E^{\psi}[h_{n}(Z_{k})]=\psi(h_{n})$ for every $k,n\in\bbn^\ast$; thus, $\E^{\psi}[H^{\prime n}_{s}]=0$ {for all $s\ge0$.}
Moreover, its second moment satisfies \hbox{$\E^\psi[(H^{\prime n}_{s})^2] = K^{n}_{s} + K^{\prime n}_{s}$,} where
\begin{align*}
	K^{n}_{s}  &= \frac{1}{n^2\xi_{n}^2} \sum_{k=1}^{\lfloor sn\rfloor} \big(\psi(h^2_{n}) - \psi(h_{n})^2\big) \\
\shortintertext{and}
	K^{\prime n}_{s}  &= \frac{2}{n^2\xi_{n}^2} \sum_{k=1}^{\lfloor sn\rfloor-1} \int \psi(\dz)h_{n}(z) \sum_{l=k+1}^{\lfloor sn\rfloor}\big(\Psi^{l-k}h_{n}(z)-\psi(h_{n})\big).
\end{align*}
First, we note 
\begin{gather}\label{C eq:V1}
	|K^{n}_{s}| \le \frac{\lfloor sn\rfloor}{n} \left| \frac{\psi(h^2_{n})}{n\xi^{2}_{n}} - \frac{\psi(h_{n})^2}{n\xi^{2}_{n}} \right| \limN 0.
\end{gather}
Second, let $m\in\bbn^\ast$ denote the period of $Z$. {By Theorems\,5.4.4 and 10.4.5 and Proposition\,5.4.6 of \citet{MeynTweedie}, there exists a partition $D_{0},\dotsc,D_{m-1}$ of the state space such that the restriction of the sampled chain with transition kernel $\Psi^m$ to each set $D_i$ is aperiodic and Harris recurrent with invariant probability $m\psi({}\cdot{}\cap D_{i})$.} For every $i\in\{1,\dotsc,m\}$ and  $z\in D_{i}$, we denote~$j(l,z) \coloneqq (i+l)\bmod m$, where `$\mathrm{mod}$' stands for the modulo operator.
For every $n_0 \in \bbn^\ast$, we observe
\begin{gather}\label{C eq:Summation-ReOrder}
\begin{aligned}
	\sum_{l=1}^{n_0}\big(\Psi^{l}h_{n}(z)-\psi(h_{n})\big) = {} &
	\sum_{k=0}^{\left\lfloor \frac{n_0}{m}\right\rfloor} \sum_{l=1}^m \big( \Psi^{km+l}{h_{n}}_{\vert D_{j(l,z)}}(z) - m\psi({h_{n}}_{\vert D_{j(l,z)}}) \big) \\
	& \hspace{3em}{}+ \sum_{l=1}^{n_0 \bmod m} \big(\Psi^{\left\lfloor \frac{n_0}{m}\right\rfloor m+l}{h_{n}}_{\vert D_{j(l,z)}}(z)-\psi(h_{n})\big).
\end{aligned} 
\end{gather}
Hence,
\begin{align*}
{\left|\sum_{l=1}^{n_0}\big(\Psi^{l}h_{n}(z)-\psi(h_{n})\big)\right|}
		&   \le {}  \sum_{k=0}^\infty \sum_{l=1}^m \left|\Psi^{km+l}{h_{n}}_{\vert D_{j(l,z)}}(z) - m\psi({h_{n}}_{\vert D_{j(l,z)}})\right| + m|\psi(h_n)|.
\end{align*}
{As the state space is petite \wrt\ $\Psi$, so is each $D_i$ \wrt\ $\Psi^m$. By Theorems\,16.2.1 and 16.2.2} of \citet{MeynTweedie}, there exists a $\zeta<1$ such that, for every $l=1,\dotsc,m$ and each $k\in\bbn$,
\begin{gather}\label{C eq:Local-Rho-Existence}
	\sup_{z\in D}\left|\Psi^{km+l}{h_{n}}_{\vert D_{j(l,z)}}(z) - m\psi({h_{n}}_{\vert D_{j(l,z)}})\right| \le \zeta^k.
\end{gather}
Consequently,
\begin{gather}
	|K^{\prime n}_{s}| \le \frac{2\lfloor sn\rfloor m}{n} \left(\frac{\zeta\psi(|h_n|)}{(1-\zeta)n\xi_n^2} + \frac{\psi(|h_n|)|\psi(h_n)|}{n\xi_n^2} \right) \limN 0. \label{C eq:V2}
\end{gather}
By \cref{C eq:V1} and \cref{C eq:V2}, $\E^\psi[(H^{\prime n}_{s})^2] \to 0$, hence $H^{\prime n}_{s} \to 0$ in probability as $\nto$. It remains to show the local uniformity in $s$ of this convergence.

By \cref{C eq:Summation-ReOrder} and \cref{C eq:Local-Rho-Existence}, we have that $h_{n}-\psi(h_{n})$ is in the range of $(\Iota-\Psi)$. Let $\hat{h}_{n}$ denote its pre-image under $(\Iota-\Psi)$ (that is, its \emph{potential}), and define the process $M^n$ by
\[
	M^n_{s} \coloneqq \frac{1}{n\xi_{n}} \sum_{k=1}^{\lfloor sn\rfloor} \big(\hat{h}_{n}(Z_{k}) - \Psi\hat{h}_{n}(Z_{k-1})\big).
\]
We note that $M^n$ is a $\mathscr{G}^n_{s}$-martingale where $\mathscr{G}^n_{s}\coloneqq \sigma(Z_k: k\le\lfloor sn\rfloor)$. 
Since $(\Psi h_{n})_{n\in\bbn^\ast}$ is uniformly bounded by assumption, so is $(\Psi\hat h_n)_{n\in\bbn^\ast}$. As $\nto$, therefore, we have 
$|H^{\prime n}_s-M^n_s| = ({n\xi_{n}})^{-1}{|\Psi\hat{h}_{n}(Z_{0})-\Psi\hat{h}_{n}(Z_{\lfloor sn\rfloor})|}\to0$. Likewise, $\E^\psi[(M^n_{s})^2] \le 2\E^\psi(H^{\prime n}_{s})^2 + 2\E^\psi|H^{\prime n}_s-M^n_s|^2 \to 0$.
By Doob's inequality, therefore, $M^n \Rightarrow 0$ in ucp. Hence, also $H^{\prime n} \Rightarrow 0$ uniformly on compacts in probability as $\nto$.
\end{proof}

%:------------------------------------------------------------------------
\subsection{The auxiliary Markov chains}\label{C AuxYZ}
In this subsection, we construct auxiliary Markov chains $Z$ and $Z^\prime$ to which \cref{C lemma:AdditiveFunctional-LLN} applies.
Once and for all, we fix our points of interest, \ie, $\{(x_i,y_i):i\in I\}$ of \cref{C theo:CLT} such that $\mu'(x_i)>0$ and $F(x_i,E)>0$ for each $i$. Moreover, we choose a compact set $C \supset \{x_i:i\in I\}$ and constants $0<\varepsilon,\varepsilon'<\infty$ such that $\varepsilon < \lVert y_i \rVert < \varepsilon' $ for all $i\in I$ and such that
\begin{gather}
% 	 \nonumber \\
% 	\quad\text{and}\quad
	\inf_{x\in C} F\big(x,\{ y: \varepsilon < \lVert y \rVert < \varepsilon' \}\big) >0. \label{C eq:Choice}
\end{gather}
\vspace{-\baselineskip}
\begin{Remark}
	Under \cref{C a: C-weak,C a: HR}, such a set $C$ always exists by the choice of the points $x_i$ and the continuity of $f$ on $E\times E^\ast$.
\end{Remark}

Let $T_1,T_2,\dotsc$ denote the successive times of jumps of size between $\varepsilon$ and $\varepsilon'$ starting from $C$; that is,
\[
	T_1 \coloneqq \inf\big\{t>0: \varepsilon < \lVert\Delta X_t\rVert < \varepsilon', X_{t-} \in C\big\}
	\quad\text{and}\quad
	T_{n+1} \coloneqq T_1 \circ \theta_{T_n} + T_n.
\]
The conditional expectation \wrt\ the strict past of the stopping times $T_n$ plays a key role. We set
\begin{align*}
	q(x)		& %\coloneqq q_{\varepsilon,\varepsilon'}(x)		&&
		\coloneqq F\big(x,\{ y: \varepsilon < \lVert y \rVert < \varepsilon' \}\big)\mathbbm{1}_C(x), \\ % \quad\text{and} \\
	p(x,y)	& %\coloneqq p_{\varepsilon,\varepsilon'}(x,y)	&&
		\coloneqq
	\begin{cases}
		q^{-1}(x)f(x,y), 	&	\text{if } x\in C \text{ and } \varepsilon<\lVert y\rVert<\varepsilon', \\
		0, 						&	\text{else.}
	\end{cases}
\end{align*}
It is well-known that $T_1<\infty$ a.\,s.\ if, and only if, $\mu(q)>0$. In our case, this holds by \cref{C eq:Choice}. Therefore, $T_n < \infty$ a.\,s.\ for all $n$ as well.
For convenience, we abbreviate the kernel with density $p$ by $\Pi$; its shifted version with density $(x,y)\mapsto p(x,y-x)$ we denote by~$\bar\Pi$. By \citet{Weil1971}, $\Pi$ (resp., $\bar\Pi$) is the conditional transition probability kernel of the jumps at the time(s) $T_n$ in the following sense: On the set $\{T_n < \infty\}$, for every random variable $Y$, measurable function $g$, and all $x$, we have 
\begin{align}
	\E^x[ g(\Delta X_{T_n}) \mid \mathscr{F}_{T_n-}] & =
		\Pi g(X_{T_n-}), \label{C eq:Pi-CondExp} \\ % = \int p(X_{T_n-},y)g(y)\dy,
	\E^x[ Y\circ \theta_{T_n}\mid \mathscr{F}_{T_n-}] & =
		\bar\Pi\E^\cdot[Y](X_{T_n-}).  \label{C eq:Pibar-CondExp}
		% = \int p(X_{T_n-},y)\E^{X_{T_n-}{}+{}y}[Y]\dy, \label{C eq:Pibar-CondExp}% \quad \forall x\in E.
\end{align}
We note $\bar\Pi\E^\cdot[Y](x) = \int p(x,y)\E^{x+y}[Y]\dy$.

Let $\mathbf{D}  \coloneqq \mathcal{D}([0,1[{}; E)\times\bbrp{}\times C$. For every $k\in\bbn^\ast$, we define the $\mathbf{D}$-valued and $C$-valued random variables
\[
	Z_{k} \coloneqq \left(s\mapsto X_{(1-s)T_{k-1}+sT_{k}}, T_{k}-T_{k-1}, X_{T_{k}-}\right)
	\quad\text{and}\quad
	Z^\prime_k \coloneqq X_{T_{k}-}.
\]
The corresponding filtration $(\mathscr{G}_k)_{k\in \bbn^\ast}$ is given by $\mathscr{G}_k \coloneqq \mathscr{F}_{T_k-}$. We emphasise that we exclude time $k=0$. From \cref{C eq:Pibar-CondExp} and $T_1<\infty$ a.\,s., we deduce that $Z=(Z_k)_{k\in\bbn^\ast}$ and $Z^\prime=(Z^\prime_k)_{k\in\bbn^\ast}$ are $\mathscr{G}_k$-Markov chains. We denote their transition probabilities by $\Psi$ and $\Phi$, respectively. We refer to \cref{C OnZ} for technical results on these auxiliary Markov chains.

\begin{Lemma}\label{C lemma:OnYZ}
Let $(g,t,x)\in \mathbf{D}$, let $A\subseteq C$ and $\mathbf{A}\subseteq \mathbf{D}$ be measurable, and let $k\in\bbn^\ast$. Then
	\begin{gather}
		\Phi(x,A) = \bar\Pi \p^\cdot(Z^\prime_1\in A)(x),	\label{C eq:Upsilon}\\
		\Psi^{k+1}((g,t,x),\mathbf{A}) = \Phi^k\Psi(x,\mathbf{A}).	\label{C eq:PsiK}
	\end{gather}
\end{Lemma}
\begin{proof}
We deduce \cref{C eq:Upsilon} and \cref{C eq:PsiK} directly from \cref{C eq:Pibar-CondExp} and the Markov property of $X$, respectively.
\end{proof}
By \cref{C lemma:OnYZ}, \cref{C lemma:AdditiveFunctional-LLN} applies to $Z^\prime$ and, also, to $Z$.
%:! Lemma: Y petite state space
\begin{Lemma}\label{C lemma: C-petite}
Grant \cref{C a: C-weak,C a: HR}. Then the Markov chain $Z^\prime$ is strong Feller. Its state space $C$ is \emph{petite} with respect to $\Phi$.
\end{Lemma}
\begin{proof}
Let $f$ be a bounded Borel function and $x_{0}\in C$. Under \cref{C a: C-weak}, we deduce from Lebesgue's dominated convergence theorem that $q$ is continuous. By \cref{C eq:Choice}, we have that $x\mapsto p(x,y)$ is also continuous for every $y$ and $\sup\{p(x,y): x\in C, y\in E\}<\infty$. Again by Lebesgue's dominated convergence theorem, we conclude that
\[
	\lim_{x\to x_{0}} \bar\Pi g(x) = \lim_{x\to x_{0}} \int p(x,y) g(x+y) \dy = \int p(x_{0},y)g(x+y)\dy = \bar\Pi g(x_{0}).
\]
By \cref{C eq:Upsilon}, consequently, $\Phi = \bar\Pi\p^\cdot(Z^\prime_1\in\cdot{})$ is strong Feller on $C$.

By the same argument as for the equivalence of $T_1<\infty$ a.\,s.\ and $\mu(q)>0$, we have that the measure with $\mu$-density $q$ is an irreducibility measure of $Z^\prime$. Under \cref{C a: HR}, it is absolutely continuous. Thus, its support has non-empty interior. By Theorem\,6.2.5~(ii) of \citet{MeynTweedie}, therefore, every compact set -- hence the state space $C$ of $Z^\prime$ -- is petite with respect to $\Phi$.
\end{proof}
% By \cref{C eq:PsiK}, we obtain that \cref{C lemma:AdditiveFunctional-LLN} also applies to $Z$.
%:! Corollary: Z petite state space
\begin{Corollary}
Grant \cref{C a: C-weak,C a: HR}. Then the state space $\mathbf{D}$ of $Z$ is petite \wrt\ $\Psi$.
\end{Corollary}
\begin{proof}
By \cref{C lemma: C-petite}, there exists a probability $\rho$ on $\bbn^\ast$ and a non-trivial measure $\nu_\rho$ on $C$ such that, for every Borel set $A \subseteq C$,
\[
	\inf_{x\in C} \sum_{k=1}^\infty \rho(k)\Phi^k(x,A) \ge \nu_b(A).
\]
Let $(g,t,x) \in \mathbf{D}$, $\mathbf{A}\subseteq \mathbf{D}$ be measurable, and $\tilde\rho$ be the probability on $\bbn^\ast$ given by $\tilde\rho(1)=0$ and $\tilde\rho(k)=\rho(k-1)$ for $k>1$. By \cref{C eq:PsiK}, then
\begin{align*}
	\sum_{k=1}^\infty \tilde\rho(k)\Psi^k((g,t,x),A) &= \sum_{k=1}^\infty \rho(k) \Phi^k\Psi(x, A) \ge \nu_{\rho}\Psi(A) \eqqcolon \tilde\nu_{\tilde\rho}(A).
\end{align*}
Since $\nu_{\rho}$ is non-trival, so is $\tilde\nu_{\tilde\rho}$.
\end{proof}

%:------------------------------------------------------------------------
\subsection{Proof of \texorpdfstring{\cref{C theo:LLN}}{Theorem 3.6}}\label{C ProofLLN}
%:- Auxiliary processes: Gn, Jn, Sn
Throughout the remainder of \cref{C Proofs}, we work under the law $\p^\pi$ for some initial probability $\pi$ on $E$ and, for presentational purposes, we suppose w.\,l.\,o.\,g. that $\mu(q)=1$. 

We consider the processes $G^{n,\eta}$, $J^{n,\eta}$ and $S^{n,\eta}$ given by
\begin{gather}
	G^{n,\eta}_{s}(x,y) \coloneqq \frac{1}{n}\sum_{k=1}^{\lfloor sn\rfloor} g^{\eta,x}_{1}(X_{T_{k}-})g^{\eta,y}_{2}(\Delta X_{T_{k}}), \label{C eq: Gn def} \\
	J^{n,\eta}_{s}(x) \coloneqq \frac{1}{n}\sum_{k=1}^{\lfloor sn\rfloor} g^{\eta,x}_{1}(X_{T_{k}-}) \mAnd
	S^{n,\eta}_{s}(x)  \coloneqq \frac{1}{n}\int_{0}^{T_{\lfloor sn\rfloor}}g^{\eta,x}_{1}(X_{r})\dr. \label{C eq: JnSn def}
\end{gather}
We emphasise that these processes are of the form $\sum_{k=1}^{\lfloor sn\rfloor}h_n(Z_k)$ where $Z$ is the auxiliary Markov chain defined in \cref{C AuxYZ}.
%:= Auxiliary condition: 1st part
%For presentational purposes and w.\,l.\,o.\,g., we suppose $\mu(q)=1$. %, where $q$ is the function defined in the previous section.
We utilise the following preliminary condition as $\nto$ (cf., \cref{C eq:LLN}):
\begin{gather}
	n\eta_{1,n}^d\eta_{2,n}^d \to \infty, \mAnd \eta_{1,n}\to 0,\eta_{2,n} \to 0. \label{C eq:pLLN}
\end{gather}
\vspace{-\baselineskip}
%:! Lemma: Jn Sn convergence
\begin{Lemma}\label{C lemma:JnSn}
Grant \cref{C a:AC,C a: C-weak,C a: HR}. Let $\eta_n=\eta_{1,n}$ be such that \cref{C eq:pLLN} holds. %Moreover, let $x\in E$ such that $\mu'(x)>0$ and $F(x,E)>0$.
Then the following convergences hold uniformly on compacts in probability:
\[
	J_s^{n,\eta_{n}}(x) \LimucpN s q(x)\mu'(x)%\frac{q(x)\mu'(x)}{\mu(q)}
		\quad\text{and}\quad
	S_s^{n,\eta_{n}}(x) \LimucpN s \mu'(x). %\frac{\mu'(x)}{\mu(q)}.
\]
\end{Lemma}
\begin{proof}
Let $\psi$ and $\varphi$ denote the invariant probabilities of $Z$ and $Z^\prime$, respectively. We apply \cref{C lemma:AdditiveFunctional-LLN}:

\emph{(i)} We note that $J^{n,\eta_{n}}(x)$ is of the form \cref{C eq:AF-LLN} with $\xi_n=\eta_{n}^d$ and $h_n: C\to\bbr$ given by $h_n(z)=g_1((z-x)/\eta_n)$; $(h_n)_{n\in\bbn^\ast}$ is uniformly bounded. 
 By \cref{C cor:upsilon-psi-invariant} where $\mu(q)=1$, $q$ is the $\mu$-density of $\varphi$.
Also $q$ and $\mu'$ are continuous. %$|g_2|\le1$ and $\int g_2(z)\dz=1$,
By Lebesgue's differentiation theorem, thus,
\[
	\eta_{n}^{-d}\varphi(h_n) = \eta_{n}^{-d}\int \mu(\dz)q(z)g_1((z-x)/\eta_n) \limN q(x)\mu'(x).
% 	\frac{\varphi(h_n)}{\eta_n^d} = \frac{\int \mu(\dz)q(z)g_2((z-x)/\eta_n)}{\mu(q)\eta_n^d} \limN  \frac{q(x)\mu'(x)}{\mu(q)}.
\]
Since $n\eta_{n}^d\to\infty$, likewise, $(n\eta_n^{2d})^{-1}\varphi(|h_n|) \to 0$ as $\nto$.
\Halmos

\emph{(ii)} We note that $S^{n,\eta_n}(x)$  is of form \cref{C eq:AF-LLN} with $\xi_n=\eta_n^d$ and $h_n: \mathbf{D}\to\bbr$ given by $h_n(g,t,z) = t\int_0^{1} g_1((g(s)-x)/\eta_n)\ds$. By \cref{C cor:upsilon-psi-invariant}, $\psi=\varphi\Psi$. By \cref{C p:Representation,C lemma:Mu-Invariant}, thus,
\[
	\eta_{n}^{-d}\varphi(h_n) = \eta_{n}^{-d}\int\mu(\dz)g_1((z-x)/\eta_n) \limN \mu'(x).
	%\frac{\psi(h_n)}{\eta_n^d} = \frac{\int\mu(\dz)g_2((z-x)/\eta_n)}{\mu(q)\eta_n^d} \limN \frac{\mu'(x)}{\mu(q)}.
\]
Likewise, $(n\eta_n^{2d})^{-1}\varphi(|h_n|) \le (n\eta_n^{2d})^{-1}\int\mu(\dz)|g_1((z-x)/\eta_n)|\to 0$. By \cref{C cor:kthMoment-Inequality}, in addition, we observe
\[
	\frac{\psi(h^2_n)}{n\eta_n^{2d}} \le \frac{2\lVert g_1\rVert_\infty}{\inf_{z\in C} q(z)} \frac{\int\mu(\dz)|g_1((z-x)/\eta_n)|}{n\eta_n^{2d}} \limN 0.
\]
\end{proof}

%:! Lemma: Gn convergence
\begin{Lemma}\label{C lemma:Gn-LLN}
Grant \cref{C a:AC,C a: C-weak,C a: HR}. Let $\eta_n=(\eta_{1,n},\eta_{2,n})$ be such that \cref{C eq:pLLN} holds.
Then the following convergence holds uniformly on compacts in probability:\begin{equation*}
	G_s^{n,\eta_{n}}(x,y) \LimucpN sf(x,y)\mu'(x).  %\frac{sf(x,y)\mu'(x)}{\mu( q)}.
\end{equation*}
\end{Lemma}
\vspace{-\baselineskip}
\begin{proof}
Let $(\mathscr{H}^n_{s})_{s\ge0}$ be the filtration given by $\mathscr{H}^n_{s}\coloneqq \mathscr{F}_{T_{\lfloor sn\rfloor+1}-}$.
By \cref{C eq:Pi-CondExp}, we have $\E[\Delta G^{n,\eta_n}_{s} \mid \mathscr{H}^n_{s-}] = g^{\eta_n,x}_1(Z^\prime_{k})\Pi g_{2}^{\eta_{n},y}(Z^\prime_k)$ for $s=k/n$. Thus, the compensator of $G^{n,\eta_{n}}$ \wrt\ $(\mathscr{H}^n_{s})_{s\ge0}$ is given by $H^{n,\eta_n}_{s} \coloneqq {n}^{-1} \sum_{k=1}^{\lfloor sn\rfloor} g^{\eta_n,x}_1(Z^\prime_{k})\Pi g_{2}^{\eta_{n},y}(Z^\prime_k)$.

Fix $s\ge0$. In analogy to the proof of \cref{C lemma: C-petite}, $\Pi g_2^{\eta_{n},y}$ is continuous under \cref{C a: C-weak}. In analogy to \cref{C lemma:JnSn}, $n^{-1}\sum_{k=1}^{\lfloor sn\rfloor}|g^{\eta_n,x}_1(Z^\prime_{k})|$ converges in ucp to a non-trivial process as $\nto$. Therefore,
\begin{align*}
	\left|H^{n,\eta_n}_{s} - \Pi g_2^{\eta_{n},y}(x)J^{n,\eta_{n}}_{s}(x)\right| & \le \sup_{z\in B_{\eta_n}(x)} \left| \Pi  g_2^{\eta_{n},y}(z) -  \Pi g_2^{\eta_{n},y}(x) \right| \cdot \frac{1}{n}\sum_{k=1}^{\lfloor sn\rfloor}|g^{\eta_n,x}_1(Z^\prime_{k})|
	 \limN 0.
\end{align*}
Since $p$ is continuous under \cref{C a: C-weak}, $\lim_{\nto} \Pi g_2^{\vartheta_{n},y}(x) = p(x,y)$ by Lebesgue's differentiation theorem. We recall $f(x,y)=q(x)p(x,y)$. By \cref{C lemma:JnSn}, hence,
\[
	H^{n}_{s} \LimucpN sf(x,y)\mu'(x). %\frac{sf(x,y)\mu'(x)}{\mu(q)}.
\]

It remains to prove $M^{n}_s \coloneqq G^n_s - H^n_s \Rightarrow 0$ uniformly on compacts in probability. By \cref{C eq:pLLN}, we have $\sup_s \lVert\Delta M^n_s\rVert_\infty \le (n\eta_n^d\vartheta_n^d)^{-1}\lVert g_1\rVert_\infty\lVert g_2\rVert_\infty \to 0$. By Theorem\,VIII.3.33 of \citet{jacodshir}, thus, it is sufficient to show that the predictable quadratic variation $\langle M^{n},M^{n}\rangle_s$ of $M^n$ converges in probability to zero for all $s$. We observe
\begin{align*}
	\big\langle M^{n},M^{n}\big\rangle_{s}
	  &= \frac{1}{n^2} \sum_{k=1}^{\lfloor sn\rfloor} \E^\pi\big[g_1^{\eta_n,x}(Z^\prime_k)^2\big(g_2^{\eta_n,y}(\Delta X_{T_k})-\Pi g_2^{\eta_n,y}(Z^\prime_k)\big)^2 \,\big\vert\, \mathscr{H}^n_{k/n}\big] \\
	  &\le \frac{1}{n\eta_{1,n}^d\eta_{2,n}^d}\cdot\frac{1}{n} \sum_{k=1}^{\lfloor sn\rfloor} \eta_{1,n}^dg_1^{\eta_n,x}(Z^\prime_k)^2 \int_{B_1(0)} p(Z^\prime_k, y + \eta_{2,n}z)g_2(z)^2\dz.
\end{align*}
In analogy to \cref{C lemma:JnSn} again, $n^{-1}\sum_{k=1}^{\lfloor sn\rfloor}\eta_{1,n}^dg^{\eta_n,x}_1(Z^\prime_{k})^2$ converges in ucp to a non-trivial process as $\nto$. As in the proof of \cref{C lemma: C-petite}, moreover, $p$ is bounded on $C\times E$. Consequently, $\langle M^{n},M^{n}\rangle_{s} \to 0$ in probability as $\nto$.
\end{proof}

Next, we carry \cref{C lemma:JnSn,C lemma:Gn-LLN} over to the time-scale of $X$.
%:- Time-Scale J
Let $J$ be the process given by
\begin{gather}\label{C eq:J}
	J_t \coloneqq \sum_{k=1}^\infty \mathbbm{1}_{[0,t]}(T_k).
\end{gather}
We note that $J$ is a non-decreasing additive functional of $X$. It is the random clock of $Z$ (and $Z^\prime$) in terms of $X$. By \cref{C eq:Levy-system} -- where $H_t=t$ --, and by $\mu(q)=1$, we have $\E^\mu\!J_t=t$ for all $t>0$.
%:! Lemma: (Gvt,Svt) Tight Family
\begin{Lemma}\label{C lemma: TightFam1}
 Grant \cref{C a:AC,C a: C-weak,C a: HR}. Let $v: \bbrp\to\bbrp$ denote a deterministic equivalent of $X$, and let $\eta_t$ and $(x,y)\in E\times E^\ast$ be as in \cref{C theo:LLN}. Then %, within the limits of \cref{C lemma:Gn-LLN},
\begin{gather}\label{C eq: TightFam1}
	\text{the family } \left\{ \mathscr{L}\left( G^{v_{t},\eta_t}_{J_t/v_t}(x,y), S^{v_t,\eta_t}_{J_t/v_t}(x) \mid \p^\pi \right) : t>0 \right\} \text{ is tight.}
\end{gather}
Moreover, each limit point of the family in \cref{C eq: TightFam1} is the law $\mathscr{L}(f(x,y)\mu'(x)\tilde L, \mu'(x)\tilde L)$ for some positive random variable $\tilde L$.
\end{Lemma}
\begin{proof}
As $J$ is a non-decreasing additive functional of $X$, by \citet{Loecherbach2008}, the families $\{\mathscr{L}(J_t/v_t\mid\p^\pi):{t>0}\}$ and $\{\mathscr{L}(v_t/J_t\mid\p^\pi):{t>0}\}$ are tight. By  Corollary VI.3.33 of \citet{jacodshir} and \cref{C lemma:Gn-LLN}, thus,
\begin{gather}\label{C eq: FamTight2}
	\text{the family } \left\{\mathscr{L}(G^{v_{t},\eta_{t}}(x,y), S^{v_{t},\eta_{t}}(x), J_t/v_t, v_t/J_t \mid \p^\pi) : t>0\right\}  \text{ is tight}.
\end{gather}
Let $\mathbbm{Q}$ denote a limit point of the family in \cref{C eq: FamTight2}, and let $(t_n)_{n\in\bbn}$ a sequence such that
\[
	\mathscr{L}(G^{v_{t_n},\eta_{t_n}}(x,y), S^{v_{t_n},\eta_{t_n}}(x), J_{t_n}/v_{t_n},  v_{t_n}/J_{t_n} \mid \p^\pi) \limwN \mathbbm{Q}.
\]
On some extension of the probability space, w.\,l.\,o.\,g., there exists a random variable $\tilde L>0$ such that $\mathbbm{Q}=\mathscr{L}( s\mapsto {sf(x,y)\mu'(x)}, s\mapsto {s\mu'(x)}, \tilde L, 1/\tilde L )$. Since its first and second marginal are the laws of continuous processes, we have
\begin{gather*}
	\mathscr{L}\big(G^{v_{t_n},\eta_{t_n}}_{J_{t_n}/v_{t_n}}(x,y), S^{v_{t_n},\eta_{t_n}}_{J_{t_n}/v_{t_n}}(x) \mid \p^\pi\big) \limwN \mathscr{L}\left(f(x,y)\mu'(x)\tilde L, \mu'(x)\tilde L \right).
\end{gather*}
\end{proof}

%:! Proof of Theorem 2.7
\begin{proof}[of \cref{C theo:LLN}]
For every $t\ge0$ and each $x$ and $y$, we have
\begin{gather*}
	\hat f^{\eta_{t}}_{t}(x,y) = \frac{G^{v_{t},\eta_{t}}_{J_{t}/v_{t}}(x,y)}{S^{v_{t},\eta_{t}}_{J_{t}/v_{t}}(x) + v_t^{-1}\int_{T_{J_t}}^t g_1^{\eta_{t},x}(X_s)\ds}.
\end{gather*}
Let $h_n:\mathbf{D}\to\bbr$ be given by $h_n(g,t,z)\coloneqq t\int_0^1 |g_1^{\eta_n,x}(g(s))|\ds$. By \cref{C p:Representation,C cor:kthMoment-Inequality,C cor:upsilon-psi-invariant}, we have $\psi(h_n^2) \le {2\lVert g_1\rVert_\infty}{\eta_{1,n}^{-d}(\inf_{z\in C}q(z))^{-1}}\mu(|g_1^{\eta_n,x}|)$.
By Markov's inequality, since $v_t^2\eta_{1,t}^d\to\infty$, therefore,
\begin{gather}\label{C eq: RemainderConvergence}
	v_t^{-1}\int_{T_{J_t}}^t g_1^{\eta_{t},x}(X_s)\ds \le v_t^{-1}h_{v_t}(Z_{J_t+1}) \xrightarrow[\tto]{\p^\psi} 0.
\end{gather}
By Proposition 17.1.6 of \citet{MeynTweedie}, in analogy to the proof of \cref{C lemma:AdditiveFunctional-LLN}, this convergence in probability holds under every law $\p^\pi$.

We recall the results from \cref{C lemma: TightFam1}. Let $\tilde L>0$ be a random variable such that the law $\mathscr{L}(f(x,y)\mu'(x)\tilde L, \mu'(x)\tilde L)$ is a limit point of the family in \cref{C eq: TightFam1}. Moreover, let $(t_n)_{n\in\bbn^\ast}$ be a sequence such that
\[
	\left(G^{v_{t_n},\eta_{t_n}}_{J_{t_n}/v_{t_n}}(x,y), S^{v_{t_n},\eta_{t_n}}_{J_{t_n}/v_{t_n}}(x)\right) \limdN \big(f(x,y)\mu'(x)\tilde L, \mu'(x)\tilde L\big).
\]
We recall $\mu'(x)>0$. Consequently, $\hat f^{\eta_{t_n}}_{t_n}(x,y) \to f(x,y)$ in law as $\nto$ by the continuous mapping theorem. As this limit is unique and independent of the particular limit point of the family in \cref{C eq: TightFam1}, we have that $\hat f^{\eta_t}_{t}(x,y)$ converges to $f(x,y)$ in law, hence, in probability.
\end{proof}

%:------------------------------------------------------------------------
\subsection{Proofs of \texorpdfstring{\cref{C theo:CLT,C cor:CLT}}{Theorem 3.7 and Corollary 3.8}}\label{C ProofCLT}
%:- CLT notation
%We continue to work under the law $\p^\pi$ for some initial probability $\pi$ on $E$, and to suppose $\mu(q)=1$. 
In this subsection, we work on the extended space \cref{C eq:ExtendedPSpace}, $L$ denotes the Mittag-Leffler process of order $0<\delta\le1$, and $W=(W^i)_{i\in I}$ denotes an $I$-dimensional standard Wiener process such that $L$, $W$ and $\mathscr{F}$ are independent.

%:- Auxiliary process: Un
In addition to the processes $G^{n,\eta}$, $J^{n,\eta}$ and $S^{n,\eta}$ given in \cref{C eq: Gn def} and \cref{C eq: JnSn def}, we consider the process $U^{n,\eta}$ given by
\begin{gather}\label{C eq: Un def}
% 	U^{n,\eta}_s(x,y) \coloneqq \sqrt{n \eta_1^d \eta_2^d} \big( G^{n,\eta}_{s}(x,y) - f(x,y)S^{n,\eta}_s(x) \big), \\
	U^{n,\eta}_s(x,y) \coloneqq \sqrt{n \eta_1^d \eta_2^d} \left( G^{n,\eta}_{s}(x,y) - \frac{\mu(g_1^{\eta,x}Fg_2^{\eta,y})}{\mu(g_1^{\eta,x})}S^{n,\eta}_s(x) \right).
\end{gather}
We emphasise again that these processes are of the form $\sum_{k=1}^{\lfloor sn\rfloor}h_n(Z_k)$ where $Z$ is the auxiliary Markov chain defined in \cref{C AuxYZ}. 
%%:= Auxiliary condition: 2nd part
%Also, in addition to \cref{C eq:pLLN}, we utilise the following preliminary condition as $\nto$ (cf., \cref{C eq:CLT}):
%\begin{gather}
%	n\eta_{1,n}^{d+2\alpha_1}\eta_{2,n}^d \to \zeta_1^2, \mAnd n\eta_{1,n}^{d}\eta_{2,n}^{d+2\alpha_2} \to \zeta_2^2. \label{C eq:pCLT}
%\end{gather}
%\vspace{-\baselineskip}
%:! Lemma: Un asymptotic normality
\begin{Lemma}\label{C lemma:UnCLT}
Grant \cref{C a:AC,C a: C-weak,C a: HR,C a: DK}. Let $\eta_n=(\eta_{1,n},\eta_{2,n})$ be such that \cref{C eq:pLLN} holds. Then we have the following convergence in law in $\mathcal{D}(\bbr^{I})$:
\[
	\big( U_s^{n,\eta_{n}}(x_{i},y_{i}) \big)_{i\in I}
		\LimdN
	\big(\mu'(x_i) \sigma(x_{i},y_{i}) W^{i}_s \big)_{i\in I},
\]
where $\sigma(x,y)^2$ is given by \cref{C eq: Variance}.
% 
% In addition, suppose either that $\alpha_1,\alpha_2\in\bbn^\ast$ or that $\zeta_1=\zeta_2=0$ in \cref{C eq:pCLT}. Then we have the following convergence in law in $\mathcal{D}(\bbr^{I})$:
% \[
% 	\big( U_s^{n,\eta_{n}}(x_{i},y_{i}) \big)_{i\in I}
% 		\LimdN
% 	\big(s\gamma(x_i,y_i)\mu'(x_i) + \sigma(x_{i},y_{i})\mu'(x_i)  W^{i}_s \big)_{i\in I},
% \]
% where $\gamma(x,y)$ is given by \cref{C eq: Bias} in the former case, and $\gamma(x,y)=0$ in the latter case.
\end{Lemma}
\begin{proof}
For $n\in\bbn^\ast$, let $M^{n,\eta}$ be the process given by
\begin{align*}
	{M^{n,\eta}_{s}(x,y)}
	&  \coloneqq
		\frac{\sqrt{\eta_1^d\eta_2^d}}{\sqrt{n}} \sum_{k=1}^{\lfloor sn\rfloor}
			\left(g_1^{\eta,x}(Z^\prime_k)g_2^{\eta,y}(\Delta X_{T_{k}}) - \int_{T_{k-1}}^{T_{k}} g_1^{\eta,x}(X_s)Fg_2^{\eta,y}(X_s)\ds \right),
\end{align*}
and let $(\mathscr{H}^n_{s})_{s\ge0}$ be given by $\mathscr{H}^n_{s}\coloneqq \mathscr{F}_{T_{\lfloor sn\rfloor}}$. By Theorem VIII.3.33 of \citet{jacodshir}, it is sufficient to prove (i)--(iv) as follows:
\begin{enumerate}
	\item We have $U^{n,\eta_{n}}_s(x,y)-M_s^{n,\eta_{n}}(x,y) \Rightarrow 0$ in ucp as $\nto$.
% 	\item we have $(n\eta_{1,n}^d\eta_{2,n}^d)^{1/2}[{\mu(g_1^{\eta,x}Fg_2^{\eta,y})}/{\mu(g_1^{\eta,x})} - f(x,y)] \to \gamma(x,y)$;
	\item The process $M^{n,\eta}$ 	is an $\mathscr{H}^n_s$-martingale for each $n$.
	\item For all $i,j\in I$, we have
	\begin{gather*}
		\big\langle M^{n,\eta_{n}}(x_{i},y_{i}), M^{n,\eta_{n}}(x_{j},y_{j})\big\rangle_{s} \limpPiN s[\sigma(x_{i},y_{i})\mu'(x)]^2\delta_{ij}.
		\vspace{-\baselineskip}
	\end{gather*}
	\item We have the ``conditional Lyapunov condition''
	\begin{equation*}
		K^{n,\eta_n}_s(x,y) \coloneqq \sum_{k=1}^{\lfloor sn\rfloor} \E^\pi\left[\big(\Delta M_{k/n}^{n,\eta_{n}}(x,y)\big)^4 \,\middle|\, \mathscr{H}^n_{k/n-}\right] \limpPiN 0.
	\end{equation*}
\end{enumerate}

\emph{(i)} We note that $U^{n,\eta}(x,y)-M^{n,\eta}(x,y)$ is of form \cref{C eq:AF-LLN} with $h_n: \mathbf{D}\to\bbr$ given by
\[
	h_n(g,t,z) = t\int_0^1 g_1\left(\frac{g(s)-x}{\eta_{1,n}}\right)\left(Fg_2^{\eta_n,y}(g(s)) - \frac{\mu(g_1^{\eta_n,x}Fg_2^{\eta_n,y})}{\mu(g_1^{\eta_n,x})} \right)\ds,
\]
and $\xi_n=\eta_{1,n}^{d/2}\eta_{2,n}^{-d/2}n^{-1/2}$. By \cref{C p:Representation,C lemma:Mu-Invariant,C cor:upsilon-psi-invariant}, we have
\[
	\xi_n^{-1}\psi(h_n) = \sqrt{n\eta_{1,n}^d\eta_{2,n}^d}\int \mu(\dz) g_1^{\eta,x}(z)\left(Fg_2^{\eta_n,y}(z) - \frac{\mu(g_1^{\eta_n,x}Fg_2^{\eta_n,y})}{\mu(g_1^{\eta_n,x})} \right) \equiv 0.
\]
Since $\eta_{2,n}\to0$, we also observe
\[
	\frac{\psi(|h_n|)}{n\xi_n^2} \le \eta_{2,n}^d \left(\mu(|g_1^{\eta_n,x}Fg_2^{\eta_n,y}|) + \mu(|g_1^{\eta_n,x}|)\cdot\left\lvert\frac{\mu(g_1^{\eta_n,x}Fg_2^{\eta_n,y})}{\mu(g_1^{\eta_n,x})}\right\rvert \right) \limN 0.
\]
By \cref{C cor:kthMoment-Inequality}, likewise,
\[
	\frac{\psi(h_n^2)}{n\xi_n^2} \le \frac{2\eta_{2,n}^d\lVert g_1\rVert_\infty\lVert Fg_2^{\eta_n,y}\rVert_\infty}{\inf_{z\in C}q(z)} \left(\mu(|g_1^{\eta_n,x}Fg_2^{\eta_n,y}|) + \mu(|g_1^{\eta_n,x}|)\left\lvert\frac{\mu(g_1^{\eta_n,x}Fg_2^{\eta_n,y})}{\mu(g_1^{\eta_n,x})}\right\rvert \right) \limN 0.
\]
Since $n\xi_n\to\infty$, we deduce from \cref{C lemma:AdditiveFunctional-LLN} that (i) holds. \Halmos

\emph{(ii)}  By construction,  $M^{n,\eta}$ is integrable and adapted to $(\mathscr{H}^n_s)_{s\ge0}$. For $s = k/n$, we note  $\mathscr{H}^n_{s-} = \mathscr{F}_{T_{k-1}}$.
By \cref{C eq:Levy-system} -- where $H_t=t$ -- the compensator of our process's jump measure is given by $\dt \otimes F(X_t,\dy)$. By Doob's optional sampling theorem, thus, 
\[
	\E^\pi \left[ g_1^{\eta,x}(Z^\prime_k)g_2^{\eta,y}(\Delta X_{T_{k}}) - \int_{T_{k-1}}^{T_k} g_1^{\eta,x}(X_s)Fg_2^{\eta,y}(X_s)\ds \,\middle\vert\, \mathscr{F}_{T_{k-1}}\right] = 0
\]
for all $k\in\bbn^\ast$. Therefore, $M^{n,\eta}(x,y)$ is an $\mathscr{H}^n_s$-martingale. \Halmos

\emph{(iii)} Let $i,j\in I$. In analogy to step (ii), we deduce
\begin{align*}
	\lefteqn{\big\langle M^{n,\eta_n}(x_{i},y_{i}), M^{n,\eta_n}(x_{j},y_{j})\big\rangle_{s}} \\
	&\hspace{10em} = \frac{{\eta_{1,n}^d\eta_{2,n}^d}}{{n}} \sum_{k=1}^{\lfloor sn\rfloor} \E^\pi \left[g_1^{\eta_n,x_i}g_1^{\eta_n,x_j}(Z^\prime_k)g_2^{\eta_n,y_i}g_2^{\eta_n,y_j}(\Delta X_{T_k}) \,\middle|\, \mathscr{F}_{T_{k-1}} \right].
\end{align*}
For all $n$ large enough, we have $g_1^{\eta_n,x_i}g_1^{\eta_n,x_j}=0$ whenever $x_i\neq x_j$, and $g_2^{\eta_n,y_i}g_2^{\eta_n,y_j}=0$ whenever $y_i\neq y_j$. For all $\omega$, if $i\neq j$, thus, $\langle M^{n,\eta_n}(x_{i},y_{i}), M^{n,\eta_n}(x_{j},y_{j})\rangle_{s} \to 0$.

Moreover, let $J_{s}^{\prime n,\eta_{n}}(x) \coloneqq n^{-1}\eta_{1,n}^d\sum_{k=1}^{\lfloor sn\rfloor} \E^{X_{T_{k-1}}}[g_1^{\eta,x}(Z^\prime_1)^2]$. We note that ${J}^{\prime n,\eta_{n}}$ is of form \cref{C eq:AF-LLN} with $\xi_n = \eta_{1,n}^d$ and $h_n:\mathbf{D}\to\bbr$ given by $h_n(g,t,z) = \E^{g(0)}\left[g_1((Z^\prime_1-x)/\eta_{1,n})^2\right]$.
By \cref{C lemma:Mu-Invariant,C cor:upsilon-psi-invariant} and under \cref{C a: C-weak}, we observe 
\[
	\eta_{1,n}^{-d}\psi(h_n) = \int {\mu'(x+\eta_{1,n}z)}q(x+\eta_{1,n}z) g_1(z)^2\dz \limN {\mu'(x)q(x)}\int g_1(z)^2\dz.
\]
By \cref{C lemma:AdditiveFunctional-LLN}, since $h_n$ is non-negative and uniformly bounded, thus,
\begin{gather}\label{C eq:Jprime}
	{J_s'}^{n,\eta_{n}}(x) \LimucpN {sq(x)\mu'(x)\int g_1(z)^2\dz}.
\end{gather}
Hence, we observe 
\begin{align*}
	\lefteqn{\left|\big\langle M^{n,\eta_{n}}(x,y), M^{n,\eta_{n}}(x,y)\big\rangle_{s} - {J'_{s}}^{n,\eta_{n}}(x)p(x,y)\int g_2(w)^2\dw\right|} \\
	& \qquad \le {J'_{s}}^{n,\eta_{n}}(x)\int g_2(w)^2\dw\sup_{z,w\in B_1(0)}\big|p(x+\eta_{1,n}z,y+\eta_{2,n}w) - p(x,y)\big| \limpPiN 0.
\end{align*}
Since $f(x,y)=q(x)p(x,y)$, consequently,
\[
	\big\langle M^{n,\eta_{n}}(x,y), M^{n,\eta_{n}}(x,y)\big\rangle_{s} \limpPiN {sf(x,y)\mu'(x)}\int g_1(w)^2\dw\int g_2(z)^2\dz;
\]
that is, (iii) holds.
% where $\xi_g^2 = \int g_1(w)^2\dw\int g_2(z)^2\dz$.
\Halmos

\emph{(iv)} We observe  $| K^{n,\eta_n}_s(x,y) | \le K^{\prime n,\eta_n}_s + K^{\prime\prime n,\eta_n}_s$, where
\begin{align*}
	K^{\prime n,\eta_n}_s &\coloneqq \frac{4\eta^{2d}_{1,n}\eta^{2d}_{2,n}}{n^2}\sum_{k=1}^{\lfloor sn\rfloor} \E^{X_{T_{k-1}}}\left[\big(g_1^{\eta,x}(Z^\prime_1)g_2^{\eta,y}(\Delta X_{T_1})\big)^4\right], \\
\shortintertext{and}
	K^{\prime\prime n,\eta_n}_s &\coloneqq \frac{4\eta^{2d}_{1,n}\eta^{2d}_{2,n}}{n^2}\sum_{k=1}^{\lfloor sn\rfloor} \E^{X_{T_{k-1}}}\left[\left(\int_0^{T_1}g_1^{\eta,x}Fg_2^{\eta,y}(X_s)\ds \right)^4\right].
\end{align*}
We note that $K^{\prime n,\eta_n}$ and $K^{\prime\prime n,\eta_n}$ are of form \cref{C eq:AF-LLN} with $\xi_n =n\eta_{1,n}^{2d}\eta_{1,n}^{2d}/4$ and, respectively,
\begin{align*}
 	h_{n}(g,t,z) & = \E^{g(0)}\left[g_1((Z^\prime_1-x)/\eta_{1,n})^4g_2((\Delta X_{T_1}-y)/\eta_{2,n})^4 \right],  \\
\shortintertext{and}
	h_{n}(g,t,z) & = \E^{g(0)}\left[ \left(\int_0^{T_1}g_1\left(\frac{X_s-x}{\eta_{1,n}}\right)\int F(X_s,\mathrm{d}w)g_2\left(\frac{w-y}{\eta_{2,n}}\right)\right)^4 \right].
\end{align*}
By \cref{C lemma:Mu-Invariant,C cor:upsilon-psi-invariant}, for $K^{\prime n}$, we have 
\begin{align*}
	\frac{\psi(h_{n})}{\xi_n} = \frac{4}{n\eta_{1,n}^d\eta_{2,n}^d}\iint {\mu'(x+\eta_{1,n}z)} g_1(z)^4  f(x+\eta_{1,n}z,y+\eta_{2,n}w)g_2(w)^4\mathrm{d}w\dz \limN 0.
\end{align*}
% Analogously, we conclude that $(n\xi_n^2)^{-1}\psi(h_{n}^2) \to 0$ as $\nto$.
By \cref{C cor:kthMoment-Inequality} and \cref{C lemma:Mu-Invariant}, for $K^{\prime\prime n}$ moreover, there exists a $\zeta<\infty$ such that 
\begin{align*}
	\frac{\psi(h_{n})}{\xi_n} & \le \frac{4\zeta}{n\eta_{1,n}^d\eta_{2,n}^d} \iint {\mu'(x+\eta_{1,n}z)} |g_1(z)|f(x+\eta_{1,n}z,y+\eta_{2,n}w)|g_2(w)|\dw\dz \limN 0.
\end{align*}
Since, in both cases, $h_n$ is non-negative and uniformly bounded, we deduce from \cref{C lemma:AdditiveFunctional-LLN} that $| K^{n,\eta_n}_s(x,y) \big| \le K^{\prime n,\eta_n}_s + K^{\prime\prime n,\eta_n}_s \Rightarrow 0$ in ucp as $\nto$.
\end{proof}

Next, we carry \cref{C lemma:UnCLT} over to the time-scale of $X$. We recall that the additive functional $J$ of $X$, given in \cref{C eq:J}, is the random clock of $Z$ (and $Z^\prime$) in terms of $X$. In addition, let $L^t$ denote the process given by $L^t_{s} \coloneqq v^{-1}_{t}J_{st}$.

%:- Darling-Kac condition
Under Darling--Kac's condition, we have the important Th\'eor\`eme 3 of \citet{Touati1987} at hand; see also p.\,119 of \citet{Hoepfner1990} and Theorem 3.15 of \citet{Hoepfner2003}. For reference, we include it as the following proposition.
%:! Proposition: ML-Convergence
\begin{Proposition}\label{C prop:ML-Convergence}
Grant \cref{C a: HR,C a: DK}. Let $H = (H^1,\dotsc,H^l)$ be a $\mu$-integrable additive functional of $X$ with (component-wise) non-decreasing paths. Then, under every law $\p^\pi$, we have the following convergence in law in $\mathcal{D}(\bbr^l)$:
\begin{gather}\label{C eq:ML-Convergence}
	(v_{t}^{-1}H_{st})_{s\ge0} \LimdT \big(\!\E^\mu[H^1_{1}]L,\cdots,\E^\mu[H^l_{1}]L\big).
\end{gather}
\end{Proposition}
%:! Corollary:
Recalling \cref{C lemma:JnSn}, by eq.\ (3.4) of \citet{Hoepfner1990}, we obtain the following corollary to \cref{C prop:ML-Convergence}.
\begin{CorollaryOP}\label{C cor: LtSvt}
Grant \cref{C a:AC,C a: C-weak,C a: HR,C a: DK}. Let $\eta_t=\eta_{1,t}$ be such that \cref{C eq:LLN} holds. Then we have the following convergence in law in $\mathcal{D}(\bbr^{1+I})$:
\[
	\left(L^t, \big( S^{v_t,\eta_t}_{L^t}(x_i) \big)_{i\in I}\right) \LimdT \left( L, \big(\mu'(x_i)L\big)_{i\in I} \right).
\]
\end{CorollaryOP}

%:! Lemma: (Lt,Uvt) joint convergence
\begin{Lemma}\label{C lemma:JointCLT}
Grant \cref{C a:AC,C a: C-weak,C a: HR,C a: DK}. Let $\eta_t=(\eta_{1,t},\eta_{2,t})$ be such that \cref{C eq:LLN} holds. Then we have the following convergence in law in $\mathcal{D}(\bbr^{1+I})$:
\[
	\left(L^t, \left(	U^{v_{t},\eta_t}(x_{i},y_{i}) \right)_{i\in I}\right)
		\LimdT
	\left(L, \left( 	\mu'(x_i)\sigma(x_{i},y_{i})  W^{i} \right)_{i\in I}\right),
\]
where $\sigma(x,y)^2$ is given by \cref{C eq: Variance}.
\end{Lemma}
\begin{proof}
From \cref{C cor: LtSvt,C lemma:UnCLT}, we infer
\begin{gather}\label{C eq:Margins}
	L^t \LimdT L
	\quad\text{and}\quad
	\left(	U^{v_{t},\eta_t}(x_{i},y_{i})
		\right)_{i\in I}
			\LimdT
	\left( 	\mu'(x_i)\sigma(x_{i},y_{i})  W^{i}
		\right)_{i\in I}.
\end{gather}
% where %$(W^{i})_{i\in I}$ is an $I$-dimensional standard Wiener process, $L$ is the Mittag-Leffler process of order $\delta\in{}]0,1]$, and
Thus, the families
\[
	\left\{ \mathscr{L}(L^t\mid\p^\pi) : t\ge0 \right\}
	\quad\text{and}\quad
	\left\{  \mathscr{L}\big((	U^{v_{t},\eta_t}(x_{i},y_{i}) )_{i\in I} \mid\p^\pi\big) : t\ge0 \right\}
\]
are C-tight. By Corollary\,VI.3.33 of \citet{jacodshir}, we conclude that
\begin{gather}\label{C eq:JointCtight}
	\text{the family}\quad \left\{ \mathscr{L}\big(L^t, (	U^{v_{t},\eta_t}(x_{i},y_{i})
	)_{i\in I} \mid\p^\pi\big)  : t\ge0 \right\}
	\quad\text{ is C-tight.}
\end{gather}
In the remainder of this proof, we abbreviate ${\bm U}^{v_{t}}\coloneqq(U^{v_{t},\eta_t}(x_{i},y_{i}))_{i\in I}$.

Let $(\bar\Omega,\bar{\mathscr{F}}) \coloneqq (\mathcal{D}(\bbr\times\bbr^{I}),\mathscr{D}(\bbr\times\bbr^{I}))$ denote the canonical space, and let $({ L},{\bm W})$ be the canonical process. Moreover, let $\bar\p$ be an arbitrary limit point of the family in \cref{C eq:JointCtight}. We deduce from \cref{C eq:Margins} that its marginals are given by the Mittag-Leffler law of order~$\delta$ and the $I$-dimensional (scaled) Wiener law, respectively. For convenience, we abbreviate $\mathbbm{Q}_{1}\coloneqq\mathscr{L}({ L}\mid\bar\p)$ and $\mathbbm{Q}_{2}\coloneqq\mathscr{L}({\bm W}\mid\bar\p)$. Suppose that $L$ and ${\bm W}$ are independent processes under $\bar\p$. Then $\bar \p = \mathbbm{Q}_{1} \otimes \mathbbm{Q}_{2}$ holds. As $\bar\p$ is an arbitrary limit point of the family in \cref{C eq:JointCtight}, then it has to be unique.
Hence, $(\mathscr{L}((L^{t}, {\bm U}^{v_{t}})\mid \p^\pi) \to \mathbbm{Q}_{1} \otimes \mathbbm{Q}_{2}$ weakly as $\tto$.
\Halmos

Let $K$ denote the right-inverse of ${ L}$, \ie, $K_{t}\coloneqq\inf\{s:{ L}_{s}> t\}$, and let~$(\mathscr{H}_t)_{t\ge0}$ be the filtration on $\bar\Omega$ which is generated by the process $(K,{\bm W})$. Suppose that --~under $\bar\p$~-- $K$ and ${\bm W}$ are processes with independent increments relative to $(\mathscr{H}_t)_{t\ge0}$. (That is, $K_{t+s}-K_t$ and $\mathscr{H}_t$ are independent for all $s,t>0$, and ${\bm W}_{t+s}-{\bm W}_t$ and $\mathscr{H}_t$ are independent for all $s,t>0$.)
Then, in analogy to Step\,6 on p.\,122 of \citet{Hoepfner1990}, we deduce that --~under $\bar\p$~-- the pair~$(K,{\bm W})$ itself is a process with independent increments relative to $(\mathscr{H}_t)_{t\ge0}$. We recall that~$K$ is a $\delta$-stable subordinator, thus, purely discontinuous (resp., deterministic if $\delta=1$). Since ${\bm W}$ is continuous, hence, $K$ and ${\bm W}$ are independent processes --~under~$\bar \p$. Consequently,~$\bar\p=\mathbbm{Q}_{1} \otimes \mathbbm{Q}_{2}$. \Halmos

It remains to show that --~under $\bar\p$~-- $K$ and ${\bm W}$ are processes with independent increments relative to $(\mathscr{H}_t)_{t\ge0}$. This, however, follows in analogy to Step\,7 on pp.\,123f of \citet{Hoepfner1990} with obvious notation.
\end{proof}

%:-->Stable convergence
Next, we demonstrate that the convergence in \cref{C lemma:JointCLT} holds stably in law.
%:->Lemma: Stable
\begin{Lemma}\label{C lemma:Stable}
Grant \cref{C a:AC,C a: HR,C a: C-weak,C a: DK}. Let $\eta_t$ be as in \cref{C lemma:JointCLT}.
%\cref{C lemma:JointCLT}
%Grant \cref{C a:AC,C a: C-strong,C a: DK}. Let $\eta_t$ and $\vartheta_t$ be such that \cref{C eq:LLN} and \cref{C eq:CLT} hold. %Moreover, let $I$ be a finite index set, $x_i \in E$ and $y_i \neq 0$ such that $\mu'(x_i)>0$ and $F(x_i,E)>0$ for each $i\in I$, and $(x_i,y_i)\neq(x_j,y_j)$ whenever $i\neq j$.
%Then
Then, we have the following stable convergence in law in $\mathcal{D}(\bbr^{1+I})$:
\[
	\left(L^t, \left(	U^{v_{t},\eta_t}_{L^t}(x_{i},y_{i})%, \bar U^{v_{t},\eta_{t}}_{L^t}(x_{i})
		\right)_{i\in I}\right)
		\LimdstT
	\left(L, \left(\mu'(x_i)\sigma(x_{i},y_{i})  W^{i}_{L}%, \bar\varsigma(x_{i})  \bar W^{i}_{\mu( q)L}
		\right)_{i\in I}\right),
\]
where $\sigma(x,y)^2$ is given by \cref{C eq: Variance}.
\end{Lemma}
\begin{proof}
Let $h$ be a bounded, Lipschitz continuous function on $\mathcal{D}(\bbr^{1+I})$ and $Y$ be a bounded $\mathscr{F}$-measurable random variable. %, $(W^{i})_{i\in I}$ be an $I$-dimensional standard Wiener processes and $L$ be the Mittag-Leffler process of order $\delta$.
With $\sigma(x,y)^2$ given by \cref{C eq: Variance}, we abbreviate
\begin{align*}
	{\bm U}^{v_{t}} \coloneqq \big(U^{v_{t},\eta_t}(x_{i},y_{i})\big)_{i\in I}
		\quad\text{and}\quad
	{\bm W} \coloneqq \big(\mu'(x_i)\sigma(x_{i},y_{i})W^{i}\big)_{i\in I}.
\end{align*}
We have to demonstrate
\begin{gather}\label{C eq:Stable}
	\E^\pi\left[h(L^t, {\bm U}^{v_{t}}_{L^t})Y\right]
	\limT
	\tilde\E\left[h\left(L, {\bm W}_{L} \right)\right] \E^\pi Y.
\end{gather}
%Since the class of all bounded Lipschitz continuous functions is convergence determining, we suppose that $f$ is Lipschitz.

First, we suppose that $Y$ is $\mathscr{F}_{u}$-measurable for some $u\ge0$. Let $a^t$ be given by $a^t_{s}= (s-ut^{-1})^+$. Then $a^t$ converges to $a_{s}=s$ as $t\to\infty$. By \cref{C lemma:JointCLT}, since $a^t$ is non-random, $\mathscr{L}(a^t,L^t, {\bm U}^{v_{t}}\mid\p^\pi) \to \mathscr{L}(a, L, {\bm W} \mid\tilde\p)$ weakly as $\tto$.
% \[
% 	\mathscr{L}(a^t,L^t, {\bm U}^{v_{t}}%, \bar{\bm U}^{v_{t}}
% 		 \mid\p^\pi)
% 	\limwT
% 	\mathscr{L}(a, \mu( q)L, {\bm W}%, \bar{\bm W}
% 		\mid\tilde\p).
% \]
The paths of the limit process are a.\,s.\ continuous. By eq. (3.4) of {\citet{Hoepfner1990}}, therefore,
\begin{align*}
	\mathscr{L}(a^t, L^t_{a^t}, {\bm U}^{v_{t}}\circ{L^t_{a^t}} \mid \p^{\pi})
	 &\limwT \mathscr{L}(a, L, {\bm W}_{L} \mid \tilde\p).
% 	\mathscr{L}(a, \mu( q)L_{a}, {\bm W}_{\mu( q)L_{a}} \mid \tilde\p) \\
% 	&= \mathscr{L}(a, \mu( q)L, {\bm W}_{\mu( q)L} \mid \tilde\p).
\end{align*}
Since $\E^\pi[h(L^t_{a^t}\circ\theta_{u},({\bm U}^{v_{t}}\circ{L^t_{a^t}})\circ\theta_{u})Y] = {} \E^\pi[\E^{X_{u}}[h(L^t_{a^t},{\bm U}^{v_{t}}\circ{L^t_{a^t}})]Y]$ by the Markov property, and since $\E^\pi[\tilde\E[h(L,{\bm W}_{L})]Y]={\tilde{\E}}[h(L,{\bm W}_{L})]\E^\pi Y$, consequently,
\begin{align*}
	\E^\pi[h(L^t_{a^t}\circ\theta_{u},({\bm U}^{v_{t}}\circ{L^t_{a^t}})\circ\theta_{u})Y]
	 %&{}= {} \E^\pi[\E^{X_{u}}[h(L^t_{a^t},{\bm U}^{v_{t}}\circ{L^t_{a^t}})]Y] \\
	%& \limT {} \E^\pi[\tilde\E[h(\mu( q)L,{\bm W}_{\mu( q)L})]Y], %\\
	 &{}\limT {} {\tilde{\E}}[h(L,{\bm W}_{L})]\E^\pi Y.
\end{align*}

For every $r>0$, we note 
\begin{gather*}
	\sup_{s\le r}\big|L^t_{s} - L^t_{a^t_{s}}\circ\theta_{u}\big| = \sup_{s\le r}\big|v_{t}^{-1}J_{st\wedge u}\big| \le v_{t}^{-1}J_{u} \limasT 0, \\
% \]
% Likewise, we observe
\shortintertext{and}
% \[
	\sup_{s\le r}\big\lVert ({\bm U}^{v_{t}}\circ{L^t_{a^t_{s}}})\circ\theta_u - {\bm U}^{v_{t}}\circ{L^t_{s}}\big\rVert_{\infty}
	 \le \frac {\lVert g_1\rVert_\infty(\lVert g_2\rVert_\infty J_u + \eta_{2,t}^d\lVert Fg_2^{\eta,y}\rVert_\infty u)}{\sqrt{v_{t}\eta_{1,t}^d\eta_{2,t}^d}} \limasT 0.
\end{gather*}
Since $h$ is Lipschitz, therefore,
\[
	\big|h(L^t,{\bm U}^{v_{t}}\circ{L^t}%,\bar{\bm U}^{v_{t}}\circ{L^t}
		)
	- h(L^t_{a^t}\circ\theta_{u},({\bm U}^{v_{t}}\circ{L^t_{a^t}})\circ\theta_{u}%,(\bar{\bm U}^{v_{t}}\circ{L^t_{a^t}})\circ\theta_{u}
		)\big| \limasT 0.
\]
Since $h$ and $Y$ are bounded, we deduce from Lebesgue's dominated convergence theorem that \cref{C eq:Stable} holds for all bounded $\mathscr{F}_{u}$-measurable random variables $Y$.

Second, for arbitrary bounded $\mathscr{F}$-measurable $Y$, we have $\E^\pi[Y\vert\mathscr{F}_u] \to Y$ in $\mathcal{L}^1$ as $u\to\infty$. Consequently, again by Lebesgue's dominated convergence theorem,
\begin{align*}
	\lim_{u\to\infty} \sup_{t>0} \left\lvert \E^\pi\big[h(L^t,{\bm U}^{v_{t}}\circ{L^t},\bar{\bm U}^{v_{t}}\circ{L^t})(\E^\pi[Y\vert\mathscr{F}_u]-Y)] \right\rvert = 0.
\end{align*}
Thus, \cref{C eq:Stable} holds in general.
\end{proof}
By \cref{C cor: LtSvt} and by eq.\ (3.5) of \citet{Hoepfner1990}, we obtain the following corollary to \cref{C lemma:Stable}.
%From the definition of stable convergence in law and \cref{C lemma:AdditiveFunctional-LLN}, we obtain the following corollary to \cref{C lemma:Stable}.
%:! Corollary: Joint CLT
\begin{CorollaryOP}\label{C cor:Joint-CLT}
Grant \cref{C a:AC,C a: HR,C a: C-weak,C a: DK}. Let $\eta_t$ be as in \cref{C lemma:JointCLT}.
%Grant \cref{C a: CLT}. Let $\eta_t$ and $\vartheta_t$ be such that \cref{C eq:LLN} and \cref{C eq:CLT} hold. 
Then we have the following stable convergence in law in $\mathcal{D}(\bbr^{2I})$:
\begin{align*}
	\left( S^{v_{t},\eta_{t}}_{L^t}(x_i), U^{v_{t},\eta_t}_{L^t}(x_{i},y_{i})%, \bar U^{v_{t},\eta_{t}}_{L^t}(x_{i})
	\right)_{i\in I}
	& \LimdstT
	\left({\mu'(x_i)}L, \mu'(x_i)\sigma(x_{i},y_{i})  W^{i}_{L}%, \bar\varsigma(x_{i})  \bar W^{i}_{\mu(\tilde q)L}
	\right)_{i\in I},
\end{align*}
where $\sigma(x,y)^2$ is given by \cref{C eq: Variance}.
\end{CorollaryOP}

%:-->Proof of Thm. 3.3
\begin{proof}[of \cref{C theo:CLT}]
For every $t\ge0$ and each $x$ and $y$, we have
\begin{gather*}
	\sqrt{v_{t}\eta_{1,t}^d\eta_{2,t}^d} \left(\hat f^{\eta_t}_{t}(x,y) - \bar f^{\eta_t}(x,y) \right)
	= \frac{U^{v_{t},\eta_t}_{J_t/v_t}(x,y) - \bar f^{\eta_t}(x,y)\sqrt{\eta_{1,t}^d\eta_{2,t}^d/v_t}\int_{T_{J_t}}^t g_1^{\eta_{t},x}(X_s)\ds }{S^{v_t,\eta_t}_{J_t/v_t}(x) + v_t^{-1}\int_{T_{J_t}}^t g_1^{\eta_{t},x}(X_s)\ds},
\end{gather*}
where $\bar f^{\eta}(x,y)\coloneqq {\mu(g_1^{\eta,x}Fg_2^{\eta,y})}/{\mu(g_1^{\eta,x})}$. Let $h_n:\mathbf{D}\to\bbr$ be as in the proof of \cref{C theo:LLN}. We recall $\psi(h_n^2)\le\zeta\eta_{1,n}^{-d}$ for some $\zeta<\infty$. We also note $v_t\eta_{2,t}^{-d} \to \infty$. In analogy to \cref{C eq: RemainderConvergence}, thus,
\[
	\sqrt{\eta_{1,t}^d\eta_{2,t}^d/v_t}\int_{T_{J_t}}^t g_1^{\eta_{t},x}(X_s)\ds \le \sqrt{\eta_{1,t}^d\eta_{2,t}^d/v_t}h_{v_t}(Z_{J_t+1}) \limpPiT 0.
\]

Since $L$ and $W$ are independent, $V(x_i,y_i) \coloneqq L_1^{-1/2}W^i_{L_1}$ defines an $I$-dimensional standard Gaussian random vector such that $L$, $V$ and $\mathscr{F}$ are independent. By the continuous mapping theorem and \cref{C cor:Joint-CLT}, consequently,
\begin{equation*}
	\left(\sqrt{v_{t}\eta_{1,t}^d\eta_{2,t}^d} \big(\hat f^{\eta_t}_{t}(x_{i},y_{i}) - \bar f^{\eta_t}(x_{i},y_{i})\big)  \right)_{i\in I}
	\limdstT
	\left( \sigma(x_{i},y_{i})  V(x_{i},y_{i})L_{1}^{-1/2}  \right)_{i\in I},
\end{equation*}
where $\sigma(x,y)^2$ is given by \cref{C eq: Variance}. \Halmos

In addition, grant \cref{C a: C-strong} and let $\eta_t=(\eta_{1,t},\eta_{2,t})$ be such that \cref{C eq:CLT} holds as well. We abbreviate $\bar\gamma^{\eta}(x,y) = \bar f^{\eta}(x,y) -f(x,y)$ and note
\[
	\mu(g_1^{\eta,x})\bar\gamma^{\eta}(x,y) = \iint \mu'(x+\eta_{1}z)\big(f(x+\eta_{1}z,y+\eta_{2}w)-f(x,y) \big) g_1(z)g_2(w)\dw\dz.
\]
We apply Taylor's theorem to $\mu'$ and $f$: In $x$, we expand up to the order $\lceil \alpha_1\rceil-1$ and, in $y$, we expand up to the order $\lceil \alpha_2\rceil-1$. We recall   from \cref{D eq:Kernel-g1} that $g_1$ and $g_2$ are, at least, of order $\alpha_1$ and $\alpha_2$, respectively. By a classical approximation argument, therefore, there exists a constant $\zeta<\infty$ such that $|\mu(g_1^{\eta,x})\bar\gamma^{\eta_t}(x,y)| \le \zeta(\eta_{1,t}^{\alpha_1} + \eta_{2,t}^{\alpha_2})$. If $\zeta_1=\zeta_2=0$ in \cref{C eq:CLT}, then it is immediate that $({v_{t}\eta_{1,t}^d\eta_{2,t}^d})^{1/2}\bar\gamma^{\eta_t}(x,y) \to 0$. If $\alpha_1,\alpha_2\in\bbn^\ast$, more explicitly, 
\begin{align*}
	\lefteqn{\mu(g_1^{\eta,x})\bar\gamma^{\eta}(x,y) = \eta_{1,t}^{\alpha_1} \sum_{\substack{|m_1+m_2|=\alpha_1\\|m_2|\neq0}}\frac{\kappa_{m_1+m_2}(g_1)}{m_1!m_2!}\frac{\partial^{m_1}}{\partial x^{m_1}}\mu'(x)\frac{\partial^{m_2}}{\partial x^{m_2}}f(x,y)} \\
		& \hspace{12em} + \eta_{2,t}^{\alpha_2}\sum_{|m|=\alpha_2} \frac{\kappa_{m}(g_2)}{m!}\mu'(x)\frac{\partial^m}{\partial y^m}f(x,y) + o(\eta_{1,t}^{\alpha_1} + \eta_{2,t}^{\alpha_2}).
\end{align*}
Since $\mu(g_1^{\eta,x})\to \mu'(x)$, we have $({v_{t}\eta_{1,t}^d\eta_{2,t}^d})^{1/2}\bar\gamma^{\eta_t}(x,y) \to \gamma(x,y)$ given by \cref{C eq: Bias}.
\end{proof}
\begin{proof}[of \cref{C cor:CLT}]
In analogy to the proof of \cref{C theo:CLT}, by \cref{C cor:Joint-CLT} it remains to show that $({v_{t}\eta_{1,t}^d\eta_{2,t}^d})^{1/2}\hat\gamma^{\eta_t}_t(x,y)$ is a consistent estimator for $\gamma(x,y)$. 

We recall that in classical (conditional) density estimation,  the (partial) derivatives of a consistent density estimator --~provided they exist~-- are consistent for the (partial) derivatives of the estimated density. 
In analogy to \cref{C lemma: TightFam1}, we observe that this is also true in our context. In particular, 
\begin{gather*}
	\frac{\partial^{m_1+m_2}}{\partial x^{m_1}\partial y^{m_2}}\hat f^{\eta_t}_t(x,y) \limpPiT \frac{\partial^{m_1+m_2}}{\partial x^{m_1}\partial y^{m_2}} f(x,y)
	\mAnd	
	\frac{\int_0^t \frac{\partial^m}{\partial x^m}g_1^{\eta_t,x}(X_s)\ds}{\int_0^t g_1^{\eta_t,x}(X_s)\ds} \limpPiT \frac{\frac{\partial^m}{\partial x^m}\mu'(x)}{\mu'(x)}.
\end{gather*}
If either $\alpha_1,\alpha_2\in\bbn^\ast$ or $\zeta_1=\zeta_2=0$ in \cref{C eq:CLT}, consequently, $({v_{t}\eta_{1,t}^d\eta_{2,t}^d})^{1/2}\hat\gamma^{\eta_t}_t(x,y) \to \gamma(x,y)$ in probability as $\tto$.
%the partial derivatives $\int_0^t \partial^mg_1^{\eta,x}(X_s)\ds$ of the sojourn time, normalised by the sojourn time itself, are consistent estimators for the ratios $\partial^m\mu'(x)/\mu'(x)$.
\end{proof}

\section{Proofs for results of \texorpdfstring{\cref{D Estimation}}{Section 2}}\label{D Proofs}
Throughout this section, $\zeta<\infty$ denotes some generic constant which may depend on the variables specified at the beginning of each proof. It may change from line to line. 

This section is organised as follows:
{First, in \cref{D Discrete} we study the influence of discretisation on our estimator. We prove results for the small-time asymptotic of Itô semi-martingales and for the sojourn time discretisation error.} Second, in \cref{D ProofMartingale} we prove an auxiliary, non-standard martingale limit theorem.
%Second, in \cref{D Proofs6.2} we present two important upper bounds for the small-time asymptotic of Itô semi-martingales. We conclude with the proofs of \cref{D p: Small Time Asymptotic,D p: Sojourn Time Approximation}. 
Third, in \cref{D ProofLLN} we prove the consistency of our estimator (\cref{D t: LLN}) utilising our results from \cref{D Discrete,C ProofLLN}. Last, in \cref{D ProofCLT} we apply \cref{D p: AuxMart} from \cref{D ProofMartingale} to our case and conclude with the final steps in the proof of the central limit theorem (\cref{D t: CLT,D c: CLT}) utilising our results from \cref{D Discrete,C ProofCLT}.

%:------------------------------------------------------------------------
\subsection{Small-time asymptotic and sojourn time discretisation error}\label{D Discrete}
In this subsection, we study the influence of discretisation. %on our estimator. 
%%:- Deterministic equivalents
%\emph{Deterministic equivalents} of Markov processes play a crucial role in the limit theory for our estimator. We refer to the first paragraph of Section\,2.3 and recall: Under Darling--Kac's condition (\cref{D a: DK}), on the one hand, the $v$ in \cref{D eq: Darling Kac} is a deterministic equivalent of $X$; for Markov processes violating Darling--Kac's condition, on the other hand, \citet{Loecherbach2008} showed that some {deterministic equivalent} already exists when $X$ is Harris recurrent.

%:- Small Time Asymptotic
We compare our estimators in \cref{D def:PE,C def:PE}: In the numerator of the former, the jumps $\Delta X_t$ and the pre-jump left-limits $X_{t-}$ are replaced by the increments $\Delta^n_kX$ and the pre-increment values $X_{(k-1)\Delta}$, respectively. Our Itô semi-martingale meets the following small-time asymptotic:
%:   ! Proposition: Small Time Asymptotic
\begin{Proposition}\label{D p: Small Time Asymptotic}
Let $A$ be a compact subset of $E \times E^\ast$, \hbox{$\eta_0 < \min\{\lVert y\rVert: (x,y)\in A\}$}, and let $g$ be a twice continuously differentiable kernel with compact support. Grant \cref{D a: BCn,D a: C-weak}. Then, for every $m\in\bbn^\ast$, there exists $\zeta<\infty$ such that
\begin{gather}\label{D eq: Small Time Asymptotic}
	\begin{aligned}
		\lefteqn{\left| \frac{1}{\Delta}\E^x \left[g^{\eta,y}(\Delta_1^nX)\right] - \int F(x,\dw) g^{\eta,y}(w) \right|} \\
		& \hspace{5em} \le \zeta\left[{\Delta}^{(\alpha\wedge1)/2} + \frac{\Delta}{\eta^{2\vee (\beta+d)}}\left(1 + \sum_{k=1}^m\frac{\Delta^{k}}{\eta^{2k}}\right) + \frac{\Delta^m}{\eta^{2(m+1)+d}}\right]
	\end{aligned}
\end{gather}
holds for every $(x,y) \in A$, $\eta<\eta_0$ and $\Delta\le1$, where $g^{\eta,y}(w) = \eta^{-d}g((w-y)/\eta)$.
\end{Proposition}
%:   ! Remark: Presentational comment
\begin{Remark}
For presentational purposes, we have left a small gap in the finite activity case. For instance, if $f$ is locally bounded on $E\times E$, then we can improve the bound in \cref{D eq: Small Time Asymptotic} replacing $\eta^{2\vee (\beta+d)}$ by $\eta^2$ independently of the dimension $d$.
\end{Remark}
%:- Sojourn Time
In the former estimator's denominator, the sojourn time $\int_0^t g_1^{\eta,x}(X_s)\ds$ is replaced by its Riemann sum approximation $\Delta\sum_{k=1}^{n}g_1^{\eta,x}(X_{(k-1)\Delta})$.
%:   ! Proposition: Sojourn Time Approximation
\begin{Proposition}\label{D p: Sojourn Time Approximation}
Let $x\in E$, $v:\bbrp\to\bbrp$ be a non-decreasing function, $\xi_n>0$, $\eta_n \to 0$, and $(h_n)_{n\in\bbn^\ast}$ be a uniformly bounded family of twice continuously differentiable functions supported on $\ball{x}{\eta_n}$ such that $(\eta_n^{|m|}\partial^m h_n)_{n\in\bbn^\ast}$ is uniformly bounded for every multi-index $m$ with $|m|\in\{1,2\}$. As $n\Delta\to\infty$ and $\Delta\to0$, we suppose $v(n\Delta)\eta_n^d\to \infty$ and $\xi_n\Delta\eta_n^{-2 - d[(1-2/(\beta+d))\vee0]}\to 0$. % and $\xi_n(v(n\Delta)\eta_n^d)^{-1}\to0$.
\begin{enumerate}
	\item Grant \cref{C a: HR,D a: BCn,D a: C-weak}. If $n\Delta^2\xi_n \to 0$ and $v(s)=\bar v(st)$ for some deterministic equivalent $\bar v$ of $X$ and some $t>0$, then, under any law $\p^\pi$, we have the following convergence in probability:
		\begin{gather}\label{D eq: Sojourn Time Convergence}
			\sup_{s\le t} \frac{\xi_n}{v(n\Delta)\eta_n^d} \left|{\Delta}\sum_{k=1}^{\lfloor sn\rfloor}h_n(X_{(k-1)\Delta})-\int_0^{\lfloor sn\rfloor\Delta} h_n(X_r)\dr\right| \limpPiN 0.
		\end{gather}
	\item Grant \cref{C a: HR,C a: DK,D a: BCn,D a: C-weak}. If $(n\Delta)^{1-\delta}\Delta\xi_n \to 0$ and $v$ is the regularly varying function from \cref{D eq: Darling Kac}, then, under any law $\p^\pi$, \cref{D eq: Sojourn Time Convergence} holds for all $t>0$.
\end{enumerate}
\end{Proposition}

{Before we turn to the proofs of \cref{D p: Small Time Asymptotic,D p: Sojourn Time Approximation}, we present two auxiliary upper bounds for the small-time asymptotic of Itô semi-martingales.} Below, we heavily utilise results and notation from the books \citet{jacodshir} (esp., Chapter II) and \citet{Jacod2012} (esp., Section 2.1).

%:- Grigelionis decomposition
We recall that our underlying process $X$ is an Itô semi-martingale with absolutely continuous characteristics $(B,C,\mathfrak{n})$ satisfying \cref{D eq: Ito semimartingale}, and that its jump measure $\mathfrak{m}$ is the random measure on $\bbrp\times E$ given by $\mathfrak{m}(\dt,\dx) \coloneqq \sum_{\{s:\Delta X_s\neq0\}} \epsilon_{(s,\Delta X_s)}(\dt,\dx).$ For a function $g$ on $\Omega\times\bbrp\times E$, we define the stochastic integrals
\[
	g \star \mathfrak{m}_t \coloneqq \int_{[0,t]\times E} g(\omega, s, w)\mathfrak{m}(\omega; \ds,\dw) \mAnd g \star \mathfrak{n}_t \coloneqq \int_{[0,t]\times E} g(\omega, s, w)\mathfrak{n}(\omega; \ds,\dw),
\]
and also the purely discontinuous martingale $g\star(\mathfrak{m}-\mathfrak{n})_t$, as soon as these integrals are well-defined. By L\'evy--Itô and Grigelionis decomposition, we can assume w.\,l.\,o.\,g.\ that there exists a $d$-dimensional Wiener process $W$, defined on $(\Omega,\mathscr{F},(\mathscr{F}_t)_{t\ge0},(\p^x)_{x\in E})$, and an $E\otimes E$-valued function $\sigma$ with $c=\sigma\sigma\transpose$ such that
\[
	X_t = X_0 + \int_0^t b(X_s)\dt + \int_0^t \sigma(X_s)\mathrm{d}W_s + (w\mathbbm{1}_{\lVert w\rVert\le1}) \star (\mathfrak{m}-\mathfrak{n})_t + (w\mathbbm{1}_{\lVert w\rVert>1})\star\mathfrak{m}_t.
\]

%:- Itô formula
Itô's formula plays a crucial role in the sequel. {By a version derived from (2.1.20) of \citet{Jacod2012}, if $g: E \to \bbr$ is twice continuously differentiable, then}
	\begin{gather}\label{D eq: Ito formula}
		\begin{aligned}
			g(X_t) = g(X_0) & +\int_0^t b(X_s)\transpose \nabla\!g(X_s)\ds + \frac{1}{2}\int_0^t  \tr\big(c(X_s)\nabla^2\!g(X_s)\big)\ds  \\
				&  + \left(g(X_{-}+ w) - g(X_{-}) - w\transpose \nabla\!g(X_{-}) \right)\mathbbm{1}_{\lVert w\rVert\le1}\star\mathfrak{n}_t \\
				&  + \int_0^t \sigma(X_s)\mathrm{d}W_s + \big(g(X_{-} + w) - g(X_{-})\big)\mathbbm{1}_{\lVert w\rVert\le1}\star(\mathfrak{m}-\mathfrak{n})_t\\
				&  + \big(g(X_{-} + w) - g(X_{-})\big)\mathbbm{1}_{\lVert w\rVert>1}\star\mathfrak{m}_t,
		\end{aligned}
	\end{gather}
	where $\tr(\cdot)$ denotes the trace operator on $E\otimes E$ and $\nabla^2\!g$ denotes the Hessian of $g$.

%:- Big-Jump truncation
For $\xi>0$, we denote by $T^\xi\coloneqq\inf\{t>0: \lVert\Delta X_t\rVert>\xi\}$ the first time of a jump greater than $\xi$ . Also, we introduce the following decomposition of our semi-martingale $X$:
\[
	X_t = X_0 + X^\xi_t + X^{\prime\xi}_t, \quad\text{where } X^{\prime\xi}_t \coloneqq (w\mathbbm{1}_{\lVert w\rVert>\xi})\star\mathfrak{m}_t = \sum_{s\le t} \Delta X_s\mathbbm{1}_{\lVert \Delta X_s\rVert > \xi}.
\]
We note that $X^\xi$ and $X^{\prime\xi}$ are again Itô semi-martingales; we denote their characteristics by $(B^\xi,C, \mathfrak{n}^\xi)$ and $(B^{\prime\xi},0,\mathfrak{n}^{\prime\xi})$, respectively. Furthermore, we decompose $X^\xi$ into drift $B^\xi$, continuous martingale part $M^\cpart$, and purely discontinuous martingale part $M^\xi$. These are given by
\[
	B^\xi_t = \int_0^t b^\xi(X_s)\ds, \quad M^\cpart_t = \int_0^t \sigma(X_s)\mathrm{d}W_s \mAnd  M^\xi_t = (w\mathrm{1}_{\lVert w\rVert\le\xi})\star(\mathfrak{m}-\mathfrak{n})_t,
\]
where $b^\xi(x) = b(x) - \int_{\xi < \lVert w\rVert \le 1} F(x,\dw) w$ if $\xi<1$, and $b^\xi(x) = b(x) + \int_{1< \lVert w\rVert \le \xi}F(x,\dw)w$ if $\xi\ge1$.
Under \cref{D a: BCn}, we derive the following two lemmata.
%:   ! Lemma: Small Time Moments
\begin{Lemma}\label{D l: Small Time Moments}
	Let $\xi_0>0$ and  $p\ge2$. Grant \cref{D a: BCn}. Then, there exists a constant $\zeta<\infty$ such that, for every $0<\xi\le\xi_0$, $x\in E$, and $t\le 1$, we have
	\[
		\E^x \sup_{s\le t} \lVert X^\xi_{s\wedge T^\xi}\rVert^p \le \zeta(1+\lVert x\rVert^p)t.
	\]
\end{Lemma}
\begin{proof}
	In this proof, $\zeta<\infty$ may depend on $\xi_0$ and $p$ but neither on $t$, $x$, $\xi$ nor $\zeta'$.
	
	First, let $1 \le \xi \le \xi_0$. We emphasise that, in this case,
	\begin{gather}\label{D eq: bEta bound}
		\lVert b^\xi(x)\rVert \le \lVert b(x)\rVert + \xi_0^{d+1}F(x, \{1 < \lVert w\rVert \le \xi_0\}).
	\end{gather}
	By \cref{D eq: Ito semimartingale}, we have $\mathfrak{n}^\xi(\dt,A)=\dt F^\xi(X_t,A)\coloneqq\dt F(X_t,A\cap B_\xi(0))$ for every Borel set $A$. By construction, $X^{\prime\xi}_t=0$ on $\{t < T^\xi\}$. By (2.1.43) of \citet{Jacod2012}, thus,
	\begin{align*}
		\E^x \sup_{s\le t}\lVert X^\xi_{s\wedge T^\xi}\rVert^p & \le \zeta\E^x\left[ t^{p-1}\int_0^{t} \lVert b^\xi(X_0 + X^\xi_{s\wedge T^\xi})\rVert^p\ds +  t^{p/2-1}\int_0^t \lVert c(X_0+X^\xi_{s\wedge T^\xi})\rVert^{p/2}\ds\right] \\
		& \quad + \zeta \E^x \int_0^t\ds\int F^{\xi_0}(X_0+X^\xi_s,\dw) \lVert w\rVert^p \\
		&	\quad + \zeta  \E^x t^{p/2-1}\int_0^t\ds\left(\int F^{\xi_0}(X_0+X^\xi_s,\dw)\lVert w\rVert^2\right)^{p/2}.
	\end{align*}
	Under \cref{D a: BCn}, for all $t\le 1$, we observe
	\[
		\E^x \sup_{s\le t} \lVert X_{s\wedge T^\xi}^\xi \rVert^p \le \zeta \int_0^t (1 + \E^x\lVert X_0 + X^\xi_{s\wedge T^\xi}\rVert^p)\ds.
	\]
	For $\zeta'>0$, let $S^{\zeta'}\coloneqq \inf\{s>0: \lVert X^\xi_s \rVert > \zeta'  \}$. Then
	\[
		\E^x \sup_{s\le t} \lVert X_{s\wedge T^\xi \wedge S^{\zeta'}}^\xi \rVert^p \le \zeta \int_0^t (1 + \E^x\lVert X_0 + X^\xi_{s\wedge T^\xi \wedge S^{\zeta'}}\rVert^p)\ds,
	\]
	where we note $\sup_{s\le t} \lVert X_{s\wedge T^\xi \wedge S^{\zeta'}}^\xi \rVert \le \zeta' + \xi$. By the Grönwall--Bellmann inequality, thus,
	\begin{align*}
		\E^x \sup_{s\le t} \lVert X_{s\wedge T^\xi  \wedge S^{\zeta'}}^\xi \rVert^p & \le \zeta(1+\lVert x\rVert^p)\left(t + \int_0^t \zeta\e^{\zeta(t-s)}\ds\right) =  \zeta(1+\lVert x\rVert^p)(\e^{\zeta t} - 1).
	\end{align*}
	Since $S^{\zeta'}\wedge T^\xi \to T^\xi$ as $\zeta'\to\infty$, consequently, $\E^x \sup_{s\le t} \lVert X_{s\wedge T^\xi}^\xi \rVert^p\le\zeta(1+\lVert x\rVert^p)t$.

	Second, let $0<\xi<1$. We note that $X^\xi_t\mathbbm{1}_{t<T^\xi} = (X_t - X_0)\mathbbm{1}_{t<T^\xi}$ holds, and that $X^\xi$ is continuous at $T^\xi$ outside the null set $\{\lVert\Delta X_{T^\xi}\rVert = \xi\}$. As $T^\xi \le T^1$ for all $\omega$, thus,
	\[
		\sup_{s\le t} \lVert X^\xi_{s\wedge T^\xi}\rVert = \sup_{s\le t} \lVert (X_s-X_0)\mathbbm{1}_{s<T^\xi} \rVert \le
		\sup_{s\le t} \lVert (X_s-X_0)\mathbbm{1}_{s<T^1} \rVert = \sup_{s\le t} \lVert X^1_{s\wedge T^1}\rVert
	\]
	almost surely. By case $\xi\ge 1$, consequently, $\E^x \sup_{s\le t}\lVert X^\xi_{s\wedge T^\xi}\rVert^p \le \zeta(1+\lVert x\rVert^p)t$.
\end{proof}
%:   ! Lemma: Small Time Ball Asymptotic
\begin{Lemma}\label{D l: Small Time Ball Asymptotic}
	Let $y\neq 0$ and $\eta_0 < \lVert y\rVert$. Grant \cref{D a: BCn}. Then, for every $m\in\bbn^\ast$, there exists a constant $\zeta < \infty$ --~non-increasing in $\lVert y\rVert$~-- such that, for every $x\in E$, $\eta < \eta_0$, and $t\le1$,
	\begin{gather}\label{D eq: Small Time Ball Asymptotic}
		\begin{aligned}
		\lefteqn{\p^x(X_t \in B_\eta(X_0+y))} \\
		& \quad \le \zeta\left(1+\lVert x\rVert^{2(m+1)}+\lVert y\rVert^{2(m+1)}\right)\left[t\eta^d\left(1 +\sum_{k=1}^{m} t^k\eta^{-2\vee(\beta+d)-2(k-1)}\right) + \frac{t^m}{\eta^{2m}}\right].
		\end{aligned}
	\end{gather}
\end{Lemma}
\begin{proof}
	Let $1 < \zeta' < (\lVert y\rVert / \eta_0)^{1/(m+1)}$, $\varepsilon\coloneqq(\zeta'^{m+1}\eta_0-\zeta'^m\eta_0)/6>0$ and $\xi < \varepsilon/2$. In addition, let $g$ be a $\mathcal{C}^2$-kernel such that $\mathbbm{1}_{B_1(0)} \le g \le \mathbbm{1}_{B_{(\zeta'+1)/2}(0)}$. We set $g_\eta(z) = g((z-x-y)/\eta)$ and abbreviate
	\[
		h(t,\eta) \coloneqq \p^x(X_t\in B_\eta(x+y)) \le \E^x g_\eta(X_t).
	\]
	In this proof, $\zeta <\infty$ may depend on $\eta_0$, $\zeta'$, $\beta$ and $m$, but neither on $x$, $t$ nor $\eta$.

	By Itô's formula \cref{D eq: Ito formula}, we have $h(t,\eta) \le |H_t^{\eta}| + |H_t^{\prime\eta}| + |H_t^{\prime\prime\eta}|$, where
	\begin{align*}
		H_t^{\eta} & \coloneqq \E^x\int_0^t b(X_s)\transpose \nabla\!g_\eta(X_s)ds + \frac{1}{2}\E^x\int_0^t \tr\big(c(X_s)\nabla^2\!g_\eta(X_s)\big)\ds, \\
		H_t^{\prime\eta} & \coloneqq \E^x\int_0^t \ds \mathbbm{1}_{B_{\zeta'\eta}(x+y)}(X_s)\int F(X_s,\dw)\{ g_\eta(X_s+w)-g_\eta(X_s) - w\transpose\nabla\!g_\eta(X_s)\mathbbm{1}_{\lVert w\rVert\le1} \}, \\
		H_t^{\prime\prime\eta} & \coloneqq \E^x\int_0^t  \ds \mathbbm{1}_{B_{\zeta'\eta}(x+y)^\mathrm{c}}(X_s)\int F(X_s,\dw)g_\eta(X_s+w).
	\end{align*}
Under \cref{D a: BCn}, $b(z)$ and $c(z)$ are bounded in norm by $\zeta(1+\lVert z\rVert^2)$. Moreover, the gradient and Hessian of $g_\eta$ vanish outside $\ball{x+y}{(\zeta'+1)\eta/2}$ and satisfy $\lVert \partial_ig_\eta\rVert\le \zeta\eta^{-1}$ and $\lVert \partial_{ij}g_\eta\rVert\le \zeta\eta^{-2}$.
Hence,
\[
	\left\lvert H_t^{\eta}\right\rvert \le {\zeta(1+\lVert x\rVert^2 + \lVert y\rVert^2)}{\eta^{-2}} \E^x\int_0^t \mathbbm{1}_{\ball{x+y}{(\zeta'+1)\eta/2}}(X_s)\ds.
\]
For $z \in B_{\zeta'\eta}(x+y)$, furthermore,
\[
	\int F(z,\dw)\{ g_\eta(z+w)-g_\eta(z) - w\transpose \nabla\!g_\eta(z)\mathbbm{1}_{\lVert w\rVert\le1} \} \le \frac{\zeta(1+\lVert z\rVert)}{\eta^2}\int \bar F(\dw)(1\wedge\lVert w\rVert^2).
\]
Therefore,
\begin{gather}\label{D eq: HHprime bound}
	\left\lvert H_t^{\eta} \right\rvert+ 	\left\lvert H_t^{\prime\eta} \right\rvert \le \frac{\zeta(1+\lVert x\rVert^2 +\lVert y\rVert^2)}{\eta^2}\int_0^t h(s,\zeta'\eta)\ds.
\end{gather}

Suppose that $|H_t^{\prime\prime\eta}| \le \zeta(1 + \lVert x\rVert^{3}+ \lVert y\rVert^3) (t\eta^d + t^2\eta^{-\beta})$ holds. % for every $\lambda > \bar{\beta}$.
Then,
\[
	h(t,\eta) \le \zeta(1 + \lVert x\rVert^{3}+ \lVert y\rVert^3) t\eta^d(1+t\eta^{-(\beta+d)}) + \frac{\zeta(1+\lVert x\rVert^2+\lVert y\rVert^2)}{\eta^2}\int_0^t h(s,\zeta'\eta)\ds.
\]
By iteration, we obtain \cref{D eq: Small Time Ball Asymptotic} after $m$ steps. \Halmos 

It remains to prove  $|H_t^{\prime\prime\eta}| \le \zeta(1 + \lVert x\rVert^{3}+ \lVert y\rVert^3) (t\eta^d + t^2\eta^{-\beta})$.
Under \cref{D a: BCn}~(iii), on the one hand, we have
\begin{align*}
	\int F(z,\dw)g_\eta(z+w) &\le \zeta(1 + \lVert z\rVert) \eta^d\int \bar f(y+x-z+\eta w)g(w)\dw \\
	& \le \begin{cases}
		\zeta(1+ \lVert x\rVert)\eta^d, & \text{if } z\in \ball{x}{3\varepsilon}, \\
		\zeta(1 + \lVert x + y\rVert)\eta^d & \text{if } z\in  \ball{x+y}{1 + \zeta'\eta}^\complement.
	\end{cases}
\end{align*}
For $z \in \ball{x+y}{1 + \zeta'\eta} \setminus \ball{x+y}{\zeta'\eta}$, on the other hand, we have
\[
	\int F(z,\dw)g_\eta(z+w) \le \frac{\zeta(1 + \lVert z\rVert)} {((\zeta'-1)\eta/2)^{\beta}} \int \dw g\left(\frac{w+z-x-y}{\eta}\right)\bar{f}(w)\lVert w\rVert^{\beta}.
\]
Since $\eta^d\le\eta^{-\beta}$ and $\int\bar F(\dw)(\lVert w\rVert^\beta\wedge1)<\infty$ by assumption, thus,
\begin{gather}\label{D eq: Fg bound}
	\begin{aligned}
	\int F(z,\dw)g_\eta(z+w)
		& \le \begin{cases}
			\zeta(1+\lVert x+y\rVert)\eta^{-\beta},	& \text{if } z\in \ball{x+y}{\zeta'\eta}^\complement, \\ 
			\zeta(1+\lVert x\rVert)\eta^{d},	& \text{if } z\in\ball{x}{3\varepsilon}.
		\end{cases}
	\end{aligned}
\end{gather}

Let $S^{\varepsilon,\xi}\coloneqq \inf\{t>0: \lVert X^\xi_t\rVert > 3\varepsilon \}$, and $\Omega_t^{\varepsilon,\xi}\coloneqq \{S^{\varepsilon,\xi} \le T^\xi\wedge t\}$. We split the set $\Omega\times[0,t]$ into $A_1\coloneqq\Omega\times[\![0,t\wedge T^\xi\wedge S^{\varepsilon,\xi}[\![$, $A_2\coloneqq(\Omega_t^{\varepsilon,\xi})^\complement\times[\![T^\xi\wedge t,t]\!]$ and $A_3\coloneqq \Omega_t^{\varepsilon,\xi}\times [\![S^{\eta,\xi},t]\!]$. Then we obtain the following:

First: Since $\sup_{s\le t}\lVert X^\xi_{s\wedge T^\xi\wedge S^{\varepsilon,\xi}} - X_0\rVert \le 3\varepsilon$, by \cref{D eq: Fg bound}, we obtain
\[
	\iint_{A_1} \mathrm{d}\!\p^x\ds \mathbbm{1}_{\ball{x+y}{\zeta'\eta}^\complement}(X_s)\int F(X_s,\dw)g_\eta(X_s+w) \le \zeta(1+\lVert x\rVert)t\eta^d.
\]
Second: Under \cref{D a: BCn}, we have
\[
	\p^x(T^\xi \le t\wedge S^{\varepsilon,\xi}) \le \E^x\int_0^t \ds \mathbbm{1}_{\ball{x}{3\varepsilon}}(X_s) F(X_s, \lVert w\rVert>\xi) \le \zeta(1+\lVert x\rVert)t.
\]
By the Markov property and \cref{D eq: Fg bound}, therefore,
\begin{gather}\label{D eq: L63-2}
	\begin{aligned}
		\lefteqn{\iint_{A_2} \mathrm{d}\!\p^x\ds \mathbbm{1}_{\ball{x+y}{\zeta'\eta}^\complement}(X_s)\int F(X_s,\dw)g_\eta(X_s+w)} \\
		& \hspace{4em} \le \E^x \mathbbm{1}_{\{T^\xi \le t\wedge S^{\varepsilon,\xi}\}}\E^{X_{T^\xi}}\int_0^t\ds \mathbbm{1}_{\ball{x+y}{\zeta'\eta}^\complement}(X_s)\int F(X_s,\dw)g_\eta(X_s+w) \\
		& \hspace{4em} \le \zeta (1 + \lVert x+y\rVert)t\eta^{-\beta} \p^x(T^\xi \le t\wedge S^{\varepsilon,\xi}) \\
		& \hspace{4em} \le \zeta(1 + \lVert x\rVert^2 + \lVert y\rVert^2) t^2\eta^{-\beta}.
	\end{aligned}
\end{gather}

Third: By \cref{D l: Small Time Moments}, we have $\p^x(\Omega_t^{\varepsilon,\xi})\le \zeta(1+\lVert x\rVert^{2})t$. By the Markov property and \cref{D eq: Fg bound}, therefore,
\begin{gather}\label{D eq: L63-3}
	\begin{aligned}
		\lefteqn{\iint_{A_3} \mathrm{d}\!\p^x\ds \mathbbm{1}_{\ball{x+y}{\zeta'\eta}^\complement}(X_s)\int F(X_s,\dw)g_\eta(X_s+w)} \\
		& \hspace{15em} \le \zeta(1 + \lVert x+y\rVert)t\eta^{-\beta}\p^x(\Omega_t^{\varepsilon,\xi}) \\
		& \hspace{15em} \le \zeta(1 + \lVert x\rVert^{3}+\lVert y\rVert^3)t^2\eta^{-\beta}.
	\end{aligned}
\end{gather}
\end{proof}

%:- Proofs of Propositions
{We turn to the proofs of \cref{D p: Small Time Asymptotic,D p: Sojourn Time Approximation}.}
%:   ! Proof: Small Time Asymptotic
\begin{proof}[of \cref{D p: Small Time Asymptotic}]
Let $1 < \zeta' < (\min\{\lVert y\rVert : (x,y)\in A\} / \eta_0)^{1/(m+2)}$, and $\varepsilon,\xi>0$ be given as in the proof of \cref{D l: Small Time Ball Asymptotic}. In this proof, $\zeta<\infty$ may depend on $\eta_0$, $\zeta'$, $\beta$, $m$ and the set $A$, but neither on $x$, $y$, $\Delta$ nor $\eta$.

Let $\eta\le\eta_0$, and $(x,y) \in A$. W.\,l.\,o.\,g., we assume that $g$ is supported on $B_1(0)$. To avoid cumbersome notation, we abbreviate $h_\eta=g^{\eta,x+y}$.
From \cref{D eq: Ito semimartingale} and Itô's formula \cref{D eq: Ito formula}, we obtain $\E^xh_\eta(X_\Delta) = H^{\eta}_\Delta + H^{\prime\eta}_\Delta + H^{\prime\prime\eta}_\Delta$, where
\begin{align*}
	H^{\eta}_\Delta & = \E^x \int_0^\Delta b(X_t)\transpose\nabla\!h_\eta(X_t)\dt + \frac{1}{2}\E^x\int_0^\Delta \tr\big(c(X_t)\nabla^2\!h_\eta(X_t)\big)\dt, \\
	H^{\prime\eta}_\Delta	& = \E^x \int_0^\Delta \dt \mathbbm{1}_{\ball{x+y}{\zeta'\eta}}(X_t) \int F(X_t,\dw) \{h_\eta(X_{t}+ w) - h_\eta(X_{t}) - w\transpose \nabla\!h_\eta(X_{t})\mathbbm{1}_{\lVert w\rVert\le1}\}, \\
	H^{\prime\prime\eta}_\Delta	& = \E^x \int_0^\Delta \dt \mathbbm{1}_{\ball{x+y}{\zeta'\eta}^\complement}(X_t) \int F(X_t,\dw)h_\eta(X_t+w).
\end{align*}
By \cref{D eq: HHprime bound}, we observe 
\begin{gather*}
	\left\lvert H^{\eta}_\Delta \right\rvert + \left\lvert H^{\prime\eta}_\Delta \right\rvert
	\le \frac{\zeta}{\eta^{d+2}} \int_0^\Delta\p^x(X_t\in \ball{x+y}{\zeta'\eta})\dt.
\end{gather*}
By the choice of $\zeta'$, \cref{D l: Small Time Ball Asymptotic} implies
\begin{gather}\label{D eq: Bad Set Bound}
	\left\lvert H^{\eta}_\Delta \right\rvert + \left\lvert H^{\prime\eta}_\Delta \right\rvert
	\le \zeta\left[\frac{\Delta^2}{\eta^{2}}\left(1 + \sum_{k=1}^{m}\frac{\Delta^{k}}{\eta^{2\vee(\beta+d)+2(k-1)}}\right) + \frac{\Delta^{m+1}}{\eta^{2(m+1)+d}}  \right].
\end{gather}
{Suppose}
\begin{gather}
	\begin{aligned}
		\lefteqn{\left\lvert H^{\prime\prime\eta}_\Delta - \int F(x,\dw)h_\eta(x+w)\int_0^\Delta\p^x(X_t \not\in \ball{x+y}{\zeta'\eta})\right\rvert} \\
		& \hspace{20em} \le \zeta(\Delta^{1+(\alpha\wedge1)/2}+ \Delta^2\eta^{-(\beta+d)}).
	\end{aligned}\label{D eq: Good Set Bound}
\end{gather}
Combining \cref{D eq: Bad Set Bound} and \cref{D eq: Good Set Bound}, we obtain \cref{D eq: Small Time Asymptotic}. \Halmos

It remains to prove \cref{D eq: Good Set Bound}. By \cref{D eq: Fg bound}, we observe 
\begin{gather}\label{D eq: Fh bound}
	\int F(z, \dw)h_\eta(z+w) \le
	\begin{cases}
		\zeta\eta^{-(\beta+d)},	&	\text{if } z\in\ball{x+y}{\zeta'\eta}^\complement, \\
		\zeta,	&	\text{if } z\in \ball{x}{3\varepsilon}.
	\end{cases}
\end{gather}
Let the stopping time $S^{\varepsilon,\xi}$, and the event $\Omega_\Delta^{\varepsilon,\xi}$ be given as in the proof of \cref{D l: Small Time Ball Asymptotic}. We split the set $\Omega\times[0,\Delta]$ into $A_1\coloneqq \Omega\times[\![0,\Delta\wedge T^\xi \wedge S^{\varepsilon,\xi}[\![$, $A_2\coloneqq(\Omega_\Delta^{\varepsilon,\xi})^\complement\times[\![T^\xi\wedge \Delta,\Delta]\!]$ and $A_3\coloneqq \Omega_\Delta^{\varepsilon,\xi}\times [\![S^{\eta,\xi},\Delta]\!]$. For convenience, we also abbreviate
\[	
	\tilde f_{x,y}^\eta(z,w) \coloneqq f(z,y+x-z+\eta w) - f(x,y+\eta w).
\]
Then we obtain, first: By the choice of $\varepsilon$, we have that the convex hull of the set 
\[
	\{(z,y+(x-z)+\eta w): (x,y)\in A, \lVert z-x\rVert \le 3\varepsilon, \lVert w\rVert\le 1\}
\] 
is a compact subset of $E\times E^\ast$. By \cref{D a: C-weak} and for all $(z,w)\in\ball{x}{3\varepsilon}\times\ball{0}{1}$, we have $|\tilde f_{x,y}^\eta(z,w)|\le \zeta \lVert z-x\rVert^{\alpha\wedge1}$. By \cref{D l: Small Time Moments}, therefore,
\begin{align*}
	\iint_{A_1}\mathrm{d}\!\p^x\dt \int \dw g(w) \tilde f_{x,y}^\eta(X_t,w) \le \zeta \Delta \E^x\sup_{t\le \Delta}\lVert X^\xi_{t\wedge T^\xi\wedge S^{\varepsilon,\xi}}\rVert \le \zeta \Delta^{1+(\alpha\wedge1)/2}.
\end{align*}

Second and third: We compare \cref{D eq: Fg bound} and \cref{D eq: Fh bound}. In analogy to \cref{D eq: L63-2} and \cref{D eq: L63-3}, respectively, by the Markov property and \cref{D eq: Fh bound}, therefore,
\begin{align*}
	\iint_{A_i}\mathrm{d}\!\p^x\dt \mathbbm{1}_{\ball{x+y}{\zeta'\eta}^\complement}(X_t)\int \dw g(w) \tilde f_{x,y}^\eta(X_t,w)
	& \le \zeta\Delta^2\eta^{-(\beta+d)},
\end{align*}
for $i\in\{2,3\}$. In summary, we proved \cref{D eq: Good Set Bound}. 
\end{proof}
%:   ! Proof: Sojourn Time Approximation
\begin{proof}[of \cref{D p: Sojourn Time Approximation}]
W.\,l.\,o.\,g., we assume $\eta < 1/4$. In this proof, $\zeta<\infty$ may neither depend on $n$, $\Delta$ nor $\eta$.

By Itô's formula \cref{D eq: Ito formula}, we observe
\begin{align*}
	\lefteqn{\frac{\xi_n}{v_{n\Delta}\eta_n^d}	\left\lvert\int_0^{\lfloor sn\rfloor\Delta} h_n(X_r)\dr - \Delta\sum_{k=1}^{\lfloor sn\rfloor} h_n(X_{(k-1)\Delta})\right\rvert
	\le \lvert H^{n}_s\rvert + \lvert H^{\prime n}_s\rvert + \lvert H^{\prime\prime n}_s\rvert + \lvert M^{n}_s\rvert,}\\
% \end{gather*}
\shortintertext{where}
% \begin{align*}
			H^{n}_s
				& \coloneqq \frac{\xi_n}{v_{n\Delta}\eta_n^d}	\sum_{k=1}^{\lfloor sn\rfloor} \int_{(k-1)\Delta}^{k\Delta}\dt\int_{(k-1)\Delta}^t \left(b(X_r)\transpose\nabla\!h_n(X_r) + \frac{1}{2}\tr\big(c(X_r)\nabla^2\!h_n(X_r)\big)\right)\dr, \\
			H^{\prime n}_s
				& \coloneqq \frac{\xi_n}{v_{n\Delta}\eta_n^d}\sum_{k=1}^{\lfloor sn\rfloor} \int_{(k-1)\Delta}^{k\Delta}\dt \int_{(k-1)\Delta}^t \dr  \\
				& \qquad\quad  \int_{\lVert w\rVert\le 1} F(X_r,\dw)\{h_n(X_r+w)-h_n(X_r) - w\transpose \nabla\!h_n(X_r)\}, \\
			H^{\prime\prime n}_s
				& \coloneqq \frac{\xi_n}{v_{n\Delta}\eta_n^d}\sum_{k=1}^{\lfloor sn\rfloor} \int_{(k-1)\Delta}^{k\Delta}\dt \sum_{(k-1)\Delta<r\le t} \mathbbm{1}_{\lVert \Delta X_r\rVert>1} \{h_n(X_{r-}+\Delta X_r)-h_n(X_{r-})\}, \\
\shortintertext{and }
			M^n_s
				& \coloneqq \frac{\xi_n}{v_{n\Delta}\eta_n^d}	\sum_{k=1}^{\lfloor sn\rfloor} \int_{(k-1)\Delta}^{k\Delta}\dt \left( \int_{(k-1)\Delta}^t \nabla\!h_n(X_r)\transpose\sigma(X_r)\mathrm{d}W_r  \right. \\
				& \left.\qquad\quad + \int_{(k-1)\Delta}^t\int_{\lVert w\rVert\le 1}  \{h_n(X_{r-}+w)-h_n(X_{r-})\}(\mathfrak{m}-\mathfrak{n})(\dr,\dw)\right).
\end{align*}
It remains to show:
\begin{enumerate}
	\item Under \cref{D a: BCn,D a: C-weak,C a: HR}, if $v(s)=\bar v(st)$ for some deterministic equivalent $\bar v$ of $X$ and some $t>0$, and if $n\Delta^2\xi_n\to0$, then $H^n_s$, $H^{\prime n}_s$, $H^{\prime\prime n}_s$ and $M^n_s$ converge to zero uniformly on $\{0\le s\le t\}$ in probability.
	\item Under \cref{D a: BCn,D a: C-weak,C a: HR,C a: DK}, if $v$ is the regularly varying function from \cref{D eq: Darling Kac}, and if $(n\Delta)^{1-\delta}\Delta\xi_n\to0$, then $H^n_s$, $H^{\prime n}_s$, $H^{\prime\prime n}_s$ and $M^n_s$ converge to zero uniformly for $\{0\le s\le t\}$ in probability for all $t>0$.
\end{enumerate}

\emph{(a)} Under \cref{D a: BCn}, $b(z)$ and $c(z)$ are bounded in norm by $\zeta(1+\lVert z\rVert^2)$. Moreover, the gradient and Hessian of $h_n$ vanish outside $\ball{x}{\eta_n}$ and satisfy $\lVert \partial_ih_n\rVert\le\zeta\eta_n^{-1}$ and $\lVert \partial_{ij}h_n\rVert\le\zeta\eta^{-2}$, by assumption. Thus,
\[
	\left| b(z)\transpose\nabla\!h_n(z) + \frac{1}{2}\tr\big(c(z)\nabla^2\!h_n(z)\big)\right| \le \zeta(1 + \lVert z\rVert)\eta^{-2}\mathbbm{1}_{\ball{x}{\eta_n}}(z).
\]
By Fubini's theorem, therefore,
\[
	\sup_{r\le s}|H^n_r| \le \zeta(1 + \lVert x\rVert^2)\frac{\Delta\xi_n}{\eta^{2}} S^{\prime n,\Delta,\eta_n}_s,
	\quad\text{where } S^{\prime n,\Delta,\eta}_s = \frac{1}{v_{n\Delta}\eta^{d}} \int_0^{\lfloor sn\rfloor\Delta} \mathbbm{1}_{\ball{x}{\eta}}(X_r)\dr.
\]
In case (i), we deduce from \cref{C lemma: TightFam1} %Lemma\,3.7 of the first part %of \citet{Ueltzhoefer2012}
that the family $\{\mathscr{L}(S_t^{\prime n,\Delta,\eta_n}\mid\p^x):n\in\bbn^\ast\}$ is tight under \cref{C a: HR,D a: C-weak}. As $\Delta\xi_n\eta_n^{-2}\to0$, $\sup_{s\le t} |H^n_s| \to 0$ in probability. In case (ii), we obtain from \cref{C cor: LtSvt} that $S^{\prime n,\Delta,\eta_n}$ converges stably in law to a non-trivial process. As $\Delta\xi_n\eta_n^{-2}\to0$, $\sup_{s\le t} |H^n_s| \to 0$ in probability for all $t>0$.
\Halmos

\emph{(b)} Let $\zeta'>1$ and $\kappa = 1 \wedge 2/(\beta+d)$. Under \cref{D a: BCn}, we have
\begin{gather}\label{D eq: Fh-eta bound}
\begin{aligned}
\lefteqn{\left|\int_{\lVert w\rVert\le 1} F(z,\dw)\{h_n(z+w) - h_n(z) - w\transpose \nabla h_n(z)\}\right|} \\
& \quad \le
		\begin{cases}
			\zeta(1+\lVert z\rVert)\eta_n^{-2}\int_{\lVert w\rVert\le 1} \bar F(\dw)\lVert w\rVert^2, & \text{for } \lVert z-x\rVert \le \zeta'\eta_n^\kappa, \\
			\zeta(1+\lVert z\rVert)\eta_n^{-\kappa\beta}\int_{\lVert w\rVert\le 1} \bar F(\dw)\lVert w\rVert^\beta, & \text{for } \zeta'\eta_n^\kappa<\lVert z -x\rVert\le1+\eta_n, \\
			0, & \text{else.}
		\end{cases}
\end{aligned}
\end{gather}
Again by Fubini's theorem, therefore,
\[
	\sup_{t\le s}|H^{\prime n}_t| \le \zeta(1 + \lVert x\rVert) \left(\frac{\Delta\xi_n\eta_n^{\kappa d}}{\eta_n^{d+2}} S_s^{\prime n,\Delta,\zeta'\eta_n^\kappa} + \frac{\Delta\xi_n}{\eta^{d+\kappa\beta}} S_s^{\prime n,\Delta,1+\eta_n}\right).
\]
In analogy to step (a), since $\Delta\xi_n\eta_n^{-2 - d(1-\kappa)}\to 0$, $H^{\prime n}_s \to 0$ uniformly on $\{0 \le s \le t\}$ in probability in case (i); and for all $t>0$ in case (ii).
\Halmos

\emph{(c)} In analogy to steps (a) and (b), we note 
\begin{align*}
% 		\begin{aligned}
			|H^{\prime\prime n}_s| & \le \xi_n{\Delta}{(v_{n\Delta}\eta_n^d)^{-1}}(|h_n(X_-+w)| + |h_n(X_-)|)\mathbbm{1}_{\lVert w\rVert>1}\star \mathfrak{m}_{\lfloor sn\rfloor\Delta} \\
			& \le |K^n_s| + |N^n_{\lfloor sn\rfloor/n}| + |K^{\prime n}_s| + | N^{\prime n}_{\lfloor sn\rfloor/n}|,
% 		\end{aligned} \\
\\ \shortintertext{where}
% 	\begin{aligned}
	K^n_s  &\coloneqq \xi_n{\Delta}{(v_{n\Delta}\eta_n^d)^{-1}} |h_n(X_-+w)|\mathbbm{1}_{\lVert w\rVert>1}\star \mathfrak{n}_{\lfloor sn\rfloor\Delta},
	\\
	K^{\prime n}_s &\coloneqq \xi_n{\Delta}{(v_{n\Delta}\eta_n^d)^{-1}} |h_n(X_-)|\mathbbm{1}_{\lVert w\rVert>1}\star \mathfrak{n}_{\lfloor sn\rfloor\Delta},	\\
	N^n_s &\coloneqq \xi_n{\Delta}{(v_{n\Delta}\eta_n^d)^{-1}} |h_n(X_-+w)|\mathbbm{1}_{\lVert w\rVert>1}\star (\mathfrak{m}-\mathfrak{n})_{sn\Delta}, \\
% \shortintertext{and}
	N^{\prime n}_s  &\coloneqq \xi_n{\Delta}{(v_{n\Delta}\eta_n^d)^{-1}} |h_n(X_-)|\mathbbm{1}_{\lVert w\rVert>1}\star (\mathfrak{m}-\mathfrak{n})_{sn\Delta}.
% 	\end{aligned}
% \end{gather*}
\end{align*}
% where $K^n_s  \coloneqq \xi_n{\Delta}{(v_{n\Delta}\eta_n^d)^{-1}} |h_n(X_-+w)|\mathbbm{1}_{\lVert w\rVert>1}\star \mathfrak{n}_{\lfloor sn\rfloor\Delta}$,
% $K^{\prime n}_s \coloneqq \xi_n{\Delta}{(v_{n\Delta}\eta_n^d)^{-1}} |h_n(X_-)|\mathbbm{1}_{\lVert w\rVert>1}\star \mathfrak{n}_{\lfloor sn\rfloor\Delta}$,
% $N^n_s \coloneqq \xi_n{\Delta}{(v_{n\Delta}\eta_n^d)^{-1}} |h_n(X_-+w)|\mathbbm{1}_{\lVert w\rVert>1}\star (\mathfrak{m}-\mathfrak{n})_{sn\Delta}$,  and
% $N^{\prime n}_s  \coloneqq \xi_n{\Delta}{(v_{n\Delta}\eta_n^d)^{-1}} |h_n(X_-)|\mathbbm{1}_{\lVert w\rVert>1}\star (\mathfrak{m}-\mathfrak{n})_{sn\Delta}$.
Under \cref{D a: BCn}, since $\int_{\lVert w\rVert>1} F(z,\dw)|h_n(z+w)| = 0$ for $z\in \ball{x}{1-2\eta_n}$, we have
\[
	\int_{\lVert w\rVert>1} F(z,\dw)|h_n(z+w)| \le \zeta(1+\lVert x\rVert).
\]
In both cases (i) and (ii), therefore,
\[
	\sup_{s\le t} |K^n_s| \le  \zeta(1+\lVert x\rVert) \frac{ tn\Delta^2\xi_n}{v_{n\Delta}} \limN 0,
\]
for all $t>0$. Furthermore, we observe that $N^n$ is a martingale \wrt\ the filtration $(\mathscr{F}_{sn\Delta})_{s\ge0}$. Its predictable quadratic variation satisfies
\begin{align*}
	\langle N^n,N^n \rangle_s & = \frac{\Delta^2\xi_n^2}{v_{n\Delta}^2} |h_n(X_-+w)|^2\mathbbm{1}_{\lVert w\rVert>1} \star \mathfrak{n}_{sn\Delta} \le \zeta(1+\lVert x\rVert) \frac{sn\Delta^3\xi_n^2}{v^2_{n\Delta}\eta_n^d} \limN 0.
\end{align*}
Since $\lfloor sn\rfloor/n\to s$, $N^n_{\lfloor sn\rfloor/n} \to 0$ uniformly on $\{0 \le s \le t\}$ in probability for all $t>0$.

In addition, we recall that $F(z,\{\lVert w\rVert>1\} \le \zeta(1+\lVert z\rVert)$ under \cref{D a: BCn}. Thus,
\begin{align*}
	\sup_{s\le t} |K^{\prime n}_s| \le \zeta(1+ \lVert x\rVert) \xi_n\Delta S_t^{\prime n,\Delta,\eta_n} \limpPiN 0
\end{align*}
in case (i); and for all $t>0$ in case (ii). Again, we observe that $N^{\prime n}$ is a martingale \wrt\ the filtration $(\mathscr{F}_{sn\Delta})_{s\ge0}$. Its predictable quadratic variation satisfies
\begin{align*}
	\langle N^{\prime n}, N^{\prime n} \rangle_s & = \frac{\Delta^2\xi_n^2}{v_{n\Delta}^2\eta_n^{2d}} |h_n(X_-)|^2\mathbbm{1}_{\lVert w\rVert>1} \star \mathfrak{n}_{sn\Delta}
	 \le \frac{\zeta(1+\lVert x\rVert)\Delta^2\xi_n^2}{v_{n\Delta}\eta_n^d} S^{\prime n,\Delta,\eta_n}_s \limN 0.
\end{align*}
Thus, $N^{\prime n}_{\lfloor sn\rfloor/n} \to 0$ uniformly on $\{0 \le s \le t\}$ in probability in case (i); and for all $t>0$ in case (ii).
\Halmos

\emph{(d)} Let $(M^{\prime n}_s)_{s\ge0}$ and $(M^{\prime\prime n}_s)_{s\ge0}$ denote the $\mathscr{F}_{sn\Delta}$-martingales given by
\begin{align*}
	M^{\prime n}_s	& \coloneqq \frac{\xi_n}{v_{n\Delta}\eta_n^d} \int_0^{sn\Delta} \varphi_\Delta(r) \nabla h_n(X_r)\transpose\sigma(X_r)\mathrm{d}W_r, \\
	M^{\prime\prime n}_s & \coloneqq \frac{\xi_n}{v_{n\Delta}\eta_n^d}  \varphi_\Delta(r)(h_n(X_-+w)-h_n(X_-))\mathbbm{1}_{\lVert w\rVert\le1} \star (\mathfrak{m}-\mathfrak{n})_{sn\Delta},
\end{align*}
where $\varphi_\Delta(r) \coloneqq \Delta - (r-\lfloor r/\Delta\rfloor\Delta)$. The predictable quadratic variation of $M^{\prime n}$ satisfies
\begin{align*}
	\langle M^{\prime n}, M^{\prime n} \rangle_s & = \frac{\xi_n^2}{v_{n\Delta}^2\eta_n^{2d}} \int_0^{sn\Delta} \varphi_\Delta(r)^2 \nabla h_n(X_r)\transpose c(X_r)\nabla h_n(X_r)\dt \\
	&	 \le \frac{\zeta(1+ \lVert x\rVert^2)\Delta^2\xi_n^2}{v_{n\Delta}\eta_n^{d+2}} S^{\prime n,\Delta,\eta_n}_s.
\end{align*}
As $\Delta\xi_n\eta_n^{-2} \to 0$ and $v_{n\Delta}\eta_n^d\to\infty$, $M^{\prime n}_s \to 0$ uniformly on $\{0 \le s \le t\}$ in probability in case (i); and for all $t>0$ in case (ii).

In addition, the predictable quadratic variation of $M^{\prime\prime n}$ satisfies
\begin{align*}
	\langle M^{\prime\prime n}, M^{\prime\prime n} \rangle_s & = \frac{\xi_n^2}{v_{n\Delta}^2\eta_n^{2d}} \varphi_\Delta(r)^2 (h_n(X_-+w)-h_n(X_-))^2 \mathbbm{1}_{\lVert w\rVert\le1} \star \mathfrak{n}_{sn\Delta} \\
		& \le \frac{\Delta^2\xi_n^2}{v_{n\Delta}^2\eta_n^{2d}} \int_0^{sn\Delta} \dr \int_{\lVert w\rVert\le1} F(X_r,\dw) (h_n(X_r+w)-h_n(X_r))^2.
\end{align*}
Let $\zeta'>1$ and $\kappa=1\wedge 2/(\beta+d)$ be as in step (b). By \cref{D eq: Fh-eta bound},
\begin{align*}
	\lefteqn{\left|\int_{\lVert w\rVert\le1} F(z,\dw) (h_n(z+w)-h_n(z))^2\right|} \\
	& \quad \le 
	\begin{cases}
		\zeta(1+\lVert z\rVert)\eta_n^{-2}\int_{\lVert w\rVert\le 1} \bar F(\dw)\lVert w\rVert^2, & \text{for } \lVert z-x\rVert \le \zeta'\eta_n^\kappa, \\
			\zeta(1+\lVert z\rVert)\eta_n^{-\kappa\beta}\int_{\lVert w\rVert\le 1} \bar F(\dw)\lVert w\rVert^\beta, & \text{for } \zeta'\eta_n^\kappa<\lVert z -x\rVert\le1+\eta_n, \\
			0, & \text{else.}
	\end{cases}
\end{align*}
Therefore,
\begin{align*}
	\langle M^{\prime\prime n}, M^{\prime\prime n} \rangle_s 
	\le \frac{\zeta(1 + \lVert x\rVert)\Delta\xi_n}{v_{n\Delta}\eta_n^{d}} \left(\frac{\Delta\xi_n\eta_n^{\kappa d}}{\eta_n^{d+2}} S_s^{\prime n,\Delta,\zeta'\eta_n^\kappa} + \frac{\Delta\xi_n}{\eta_n^{d+\kappa\beta}} S_s^{\prime n,\Delta,1+\eta_n}\right).
\end{align*}
Again since $\Delta\xi_n\eta_n^{-2-d(1-\kappa)}\to 0$, $M^{\prime\prime n}_s \to 0$  uniformly on $\{0 \le s \le t\}$ in probability in case (i); and for all $t>0$ in case (ii).
\end{proof}

%:------------------------------------------------------------------------
\subsection{Auxiliary martingale limit theorem}\label{D ProofMartingale}
The theorem presented in this subsection serves as a preliminary result for the proof of our central limit theorem (\cref{D t: CLT,D c: CLT}). It is a non-standard limit theorem for a triangular, martingale array scheme. 

Here, we work on the extension \cref{C eq:ExtendedPSpace} of the probability space,  $L$ denotes the Mittag-Leffler process of order $0<\delta\le1$, and $W=(W^i)_{i\in I}$ denotes an $I$-dimensional standard Wiener process such that $L$, $W$ and $\mathscr{F}$ are independent.
%:! Theorem: Auxilliary Martingale CLT
\begin{Theorem}\label{D p: AuxMart}
	For $n\in \bbn^\ast$, let $(\mathscr{G}^n_s)_{s>0}$ be the filtration given by $\mathscr{G}^n_s\coloneqq \mathscr{F}_{\lfloor sn\rfloor\Delta}$, and $I$ be a finite index set. Moreover, let $h_n:E\times E\to \bbr^I$ be such that $\lVert h_n \rVert_\infty \to 0$ as $\nto$.
	Grant \cref{C a: HR,C a: DK}, and suppose that the process $M^n$ given by
\begin{gather}\label{D eq:Martingale form}
	M^n_s \coloneqq  \sum_{k=1}^{\lfloor sn\rfloor} h_n(X_{(k-1)\Delta}, \Delta^n_kX)
\end{gather}
is a $\mathscr{G}^n_s$-martingale such that the predictable quadratic co-variation $\langle M^{ni},M^{nj}\rangle$ is identically zero for every $i\neq j$ and all $n$ large enough. 
If $(\langle M^{ni},M^{ni}\rangle)_{i\in I}$ converges stably in law in $\mathcal{D}(\bbr^I)$ to $(\varsigma_i^2L)_{i\in I}$, then
\[
	M^n \LimdstN (\varsigma_i W^i_L)_{i\in I}.
\]
\end{Theorem}
\begin{proof}
	Let $\delta = 1$. Then we have $L_s = s$. Therefore, the convergence of $M^n$ to $(\varsigma_i^2W^i)_{i\in I}$ follows directly from standard results (see section VIII.3c of \citet{jacodshir}).

	For the remainder, let $0<\delta<1$. 	We consider the processes $L^{ni}$, $\bar L^n$, $K^n$ and $N^n$ given by
\begin{gather*}
	L^{ni}_s \coloneqq \langle M^{ni},M^{ni}\rangle_s = \sum_{k=1}^{\lfloor sn\rfloor} \E^{X_{(k-1)\Delta}} h^{i}_n(X_{(k-1)\Delta}, \Delta^n_kX)^2, \\
	\bar L^{n}_s \coloneqq \sum_{i\in I}L^{ni}_s, \quad
	K^n_u \coloneqq \inf\left\{s>0: \bar L^n_s > u\right\} \mAnd  N^n_s \coloneqq M^n_{K^n_s}.
\end{gather*}
We emphasise that $N^n(\bar L^n_s) = M^n_s + \Delta M^n_{K^n(\bar L^n_s)}$ holds for all $s$. %since $\bar L^n$ is strictly increasing.
As $\lVert \Delta M^n\rVert\le\lVert h_n\rVert\to 0$, it is sufficient to prove that we have the following stable convergence in law in $\mathcal{D}(\bbr\times\bbr^I)$:
\begin{gather}
	(\bar L^n,N^n) \LimdstN \left(\bar{\varsigma}^2L, \left((\varsigma_i/\bar\varsigma)W^{i}\right)_{i\in I}\right), \quad\text{where } \bar\varsigma^2 \coloneqq \sum_{i\in I} \varsigma_i^2.
\end{gather}
% and $L$, $W$, and $\mathscr{F}$ are independent.

First, by the continuous mapping theorem, we obtain
\begin{gather}\label{D eq: Lntilde Convergence}
	\left( \bar L^n, \left( L^{ni}\right)_{i\in I} \right) \LimdstN \left( \bar\varsigma^2L, \left( \varsigma_i^2L \right)_{i\in I} \right).
\end{gather}
Second, we remark that $K^n_u$ is a predictable $\mathscr{G}^n_s$-stopping time for all $u\ge0$. Thus, $N^n$ is a martingale \wrt\ the time-changed filtration $\mathscr{H}^n_s\coloneqq \mathscr{G}^n_{K^n_s}$. Moreover, we observe that its predictable quadratic variation satisfies
\[
	\langle N^{ni}, N^{ni} \rangle_s = L^{ni}_{K^n_s}.
\]
By \cref{D eq: Lntilde Convergence}, we have that $|L^{ni} - (\varsigma_i/\bar\varsigma)^2\bar L^n| \to 0$ uniformly on compacts in probability for all $i\in I$. We note that the (scaled) Mittag-Leffler process $\bar\varsigma^2L$ is a.\,s.\ continuous. Its right-inverse $K$ given by $K_u\coloneqq \inf\{s>0: \bar\varsigma^2L_s > u\}$ is a (deterministically time-changed) $\delta$-stable L\'evy process, hence, without fixed time of discontinuity. By (3.2) of \citet{Hoepfner1990}, therefore,
$L^{ni}_{K^n_s} \to (\varsigma_i/\bar\varsigma)^2s$ in law for every $s\ge0$; hence, in probability. By construction, we have that $\lVert \Delta N^n_s\rVert$ is bounded above by $\lVert h_n\rVert_\infty$. This bound converges to zero. By standard results (see above), consequently,
\begin{gather}\label{D eq: Nn Convergence}
	N^n \LimdstN \left((\varsigma_i/\bar\varsigma)W^{i}\right)_{i\in I}.
\end{gather}
In analogy to the proof of \cref{C lemma:JointCLT} and Steps 6 and 7 on pp.\,122--124 of \citet{Hoepfner1990}, we obtain that the pair $(\bar L^n,N^n)$ converges in law in $\mathcal{D}(\bbr\times\bbr^I)$ to $(\bar{\varsigma}^2L, \left((\varsigma_i/\bar\varsigma)W^{i}\right)_{i\in I})$. Finally, the stable convergence in law and the independence from $\mathscr{F}$ follows in analogy to \cref{C lemma:Stable}.
\end{proof}

%:------------------------------------------------------------------------
\subsection{Proof of \texorpdfstring{\cref{D t: LLN}}{Theorem 2.9}}\label{D ProofLLN}
Throughout the remainder of \cref{D Proofs}, we work under the law $\p^\pi$ for some inital probability $\pi$ on $E$, and we denote $E_\oplus \coloneqq \{x\in E: \mu'(x)>0, F(x,E)>0\}$. %We recall that, under \cref{C a: HR}, $v:\bbrp\to\bbrp$ denotes an arbitrary deterministic equivalent of $X$ whereas, under the stronger \cref{C a: DK}, $v$ denotes the regularly varying function given in \cref{D eq: Darling Kac}. 

We consider the processes $G^{n,\Delta,\eta}$ and $R^{n,\Delta,\eta}$ given by
\begin{align}
	G^{n,\Delta,\eta}_s(x,y) &\coloneqq \frac{1}{v_{n\Delta}}\sum_{k=1}^{\lfloor sn\rfloor} g_1^{\eta,x}(X_{(k-1)\Delta})g_2^{\eta,y}(\Delta^n_kX), \label{D eq:Gn}\\
	R^{n,\Delta,\eta}_s(x) & \coloneqq \frac{\Delta}{v_{n\Delta}}\sum_{k=1}^{\lfloor sn\rfloor}g_1^{\eta,x}(X_{(k-1)\Delta}). \label{D eq:Sn}
\end{align}

%:! Lemma: Rn tight
\begin{Lemma}\label{D l: P1}
	Grant \cref{D a: BCn,D a: C-weak,C a: HR}. Let $\eta_n=\eta_{1,n}$ be such that \cref{D eq: LLN} holds, and let $x\in E_\oplus$.
	\begin{enumerate}
		\item If $n\Delta^2\to0$, then,
		\begin{gather}\label{D eq:FamTight0}
			\text{the family } \left\{ \mathscr{L}\left(R^{n,\Delta,\eta_n}_1(x) \mid \p^\pi\right): n\in\bbn^\ast   \right\} \text{ is tight.}
		\end{gather}
		\item Grant \cref{C a: DK} in addition. If $(n\Delta)^{1-\delta}\Delta\to 0$, then, \cref{D eq:FamTight0} holds as well.
	\end{enumerate}
	In both cases, each limit point of the family in \cref{D eq:FamTight0} is the law $\mathscr{L}(\mu'(x)\tilde L)$ for some positive random variable $\tilde L$.
\end{Lemma}
\begin{proof}
	Let $S^{t,\eta}_s(x) \coloneqq v_{t}^{-1}\int_0^{st}g^{\eta,x}_2(X_r)\dr$. By \cref{C lemma: TightFam1}, the family $\{\mathscr{L}(S^{n\Delta,\eta_n}_1(x) \mid \p^\pi): n\in\bbn^\ast\}$ is tight; moreover, each of its limit points is the law $\mathscr{L}(\mu'(x)\tilde L)$ for some random variable $\tilde L>0$. In both cases (i) and (ii), since $\eta_n$ is such that \cref{D eq: LLN} holds, we have
	\begin{gather*}%\label{D eq:Prop2.5-cor}
		\left|S^{n\Delta,\eta_{n}}_1(x) - R^{n,\Delta,\eta_n}_1(x)\right| \limpPiN 0
	\end{gather*}
by \cref{D p: Sojourn Time Approximation}. Consequently, the family $\{\mathscr{L}(R^{n,\Delta,\eta}_1(x) \mid \p^\pi): n\in\bbn^\ast\}$ is tight; moreover, each of its limit points is a limit point of the family $\{\mathscr{L}(S^{t,\eta}_1(x) \mid \p^\pi): t>0\}$, hence, the law $\mathscr{L}(\mu'(x)\tilde L)$ for some random variable $\tilde L>0$.
\end{proof}

%:! Lemma: Small time asymptotic Applied
\begin{Lemma}\label{D l: P2a}
	Grant \cref{D a: BCn,D a: C-weak}. Let $\eta_n=(\eta_{1,n},\eta_{2,n})$ be such that $\eta_{1,n}\to 0$, $\eta_{2,n}\to0$ and $\Delta\eta_{2,n}^{-2\vee (\beta+d)} \to 0$. Moreover, let $(x,y) \in E_\oplus \times E^\ast$, and let $g$ be a $\mathcal{C}^2$-function with compact support. Then
	\begin{gather}\label{D eq: Small Time Applied}
		\lim_\nto \sup_{z\in\ball{x}{\eta_{1,n}}} \left\lvert \frac{1}{\Delta} \E^zg^{\eta_n,y}(\Delta^n_1X) - f(x,y)\int g(w)\dw \right\rvert = 0.
	\end{gather}
\end{Lemma}
\begin{proof}
	First, by \cref{D p: Small Time Asymptotic} -- where we choose $m$ large enough -- we have
	\[
		\lim_\nto \sup_{z\in\ball{x}{\eta_{1,n}}} \left|\Delta^{-1}\E^{z}g^{\eta,x}(\Delta^n_1X) - Fg^{\eta,y}(z)\right| = 0.
	\]
	Second, under \cref{D a: C-weak}, $f\in\mathcal{C}^\alpha_\loc(E\times E^\ast)$ for some $\alpha>0$. Therefore,
% Hence, it is uniformly continuous on $B_1(x) \times B_{\vartheta_0}(y)$ for all $\vartheta_0 < \lVert y\rVert$. By Lebesgue's dominated convergence theorem, therefore, the family $\{Fg^{\vartheta',y} : \vartheta'<\vartheta_0 \}$ is uniformly equi-con\-tin\-u\-ous on $B_\eta(x)$. Thus,
	\[
		\lim_\nto \sup_{z\in B_{\eta_{1,n}}(x)}  \big|Fg^{{\eta_n},y}(z) - Fg^{\eta_n,y}(x)\big| \le \lim_{\nto} \zeta\eta_{1,n}^{\alpha\wedge1} = 0.
	\]
	Third, by Lebesgue's differentiation theorem, we observe 
	\[
		\lim_\nto \left\lvert Fg^{\eta_n,y}(x) - f(x,y)\int g(w)\dw \right\rvert = 0.
	\]
% 	In summary, we proved \cref{D eq: Small Time Applied}.
\end{proof}

%:! Lemma: GnRn tight
\begin{Lemma}\label{D l: P2}
	Grant \cref{D a: BCn,D a: C-weak,C a: HR}. Let $\eta_n=(\eta_{1,n},\eta_{2,n})$ be such that \cref{D eq: LLN} holds. Moreover, let $(x,y) \in E_\oplus \times E^\ast$. Then, in both cases as in \cref{D l: P1},
	\begin{gather}\label{D eq:FamTight1}
			\text{the family } \left\{\mathscr{L}\left(G^{n,\Delta,\eta_n}_1(x,y), R^{n,\Delta,\eta_n}_1(x) \mid \p^\pi \right): n\in\bbn^\ast\right\} \text{ is tight.}
	\end{gather}
Moreover, each limit point of the family in \cref{D eq:FamTight1} is the law $\mathscr{L}(f(x,y)\mu'(x)\tilde L,\mu'(x)\tilde L)$ for some positive random variable $\tilde L$.
\end{Lemma}
\begin{proof}
	We note that $G^{n,\Delta,\eta}_s(x,y) = f(x,y)R^{n,\Delta,\eta}_s(x) + H^{n,\Delta,\eta}_s(x,y) + M^{n,\Delta,\eta}_s(x,y),$ where
	\begin{align}
		H^{n,\Delta,\eta}_s(x,y) & = \frac{1}{v_{n\Delta}} \sum_{k=1}^{\lfloor sn\rfloor} g_1^{\eta,x}(X_{(k-1)\Delta})\left(\E^{X_{(k-1)\Delta}}[g_2^{\eta,x}(\Delta^n_1X)] - \Delta f(x,y)\right), \label{D eq:Hn}\\
		M^{n,\Delta,\eta}_s(x,y) & = \frac{1}{v_{n\Delta}} \sum_{k=1}^{\lfloor sn\rfloor} g_1^{\eta,x}(X_{(k-1)\Delta})\left(g_2^{\eta,y}(\Delta^n_kX) - \E^{X_{(k-1)\Delta}}[g_2^{\eta,x}(\Delta^n_1X)]\right). \label{D eq:Mn}
	\end{align}
	By \cref{D l: P1}, it is sufficient to prove that $H^{n,\Delta,\eta_n}_1(x,y)$ and $M^{n,\Delta,\eta_n}_1(x,y)$ converge to zero in probability as $\nto$.

\bigskip\emph{(H)} We observe 
\begin{gather}\label{D eq:ineq1}
	\big\lvert H^{n,\Delta,\eta}_1(x,y)\big\rvert \le \sup_{z\in B_{\eta_{1}}(x)} \left\lvert \Delta^{-1}\E^{z}[g_2^{\eta,x}(\Delta^{n}_1X)] - f(x,y) \right\rvert v_{n\Delta}^{-1}\sum_{k=1}^{n} \Delta h^{\eta,x}(X_{(k-1)\Delta}),
\end{gather}
where $h$ is a $\mathcal{C}^2$-function dominating $|g_1|$. The sequence $(v_{n\Delta}^{-1}\sum_{k=1}^{n} \Delta h^{\eta_n,x}(X_{(k-1)\Delta}))_{n\in\bbn^\ast}$ is tight in analogy to \cref{D l: P1}.
As $\sup_{z\in B_{\eta_{1,n}}(x)}  \lvert \Delta^{-1}\E^{z}[g_2^{\eta_n,x}(\Delta^n_1X)] - f(x,y)  \rvert \to 0$ by \cref{D l: P2a}, we have $H^{n,\Delta,\eta_n}_1(x,y)\to0$ in law, hence, in probability. \Halmos

\emph{(M)} We observe that $M^{n,\Delta,\eta}$ is an $\mathscr{F}_{\lfloor sn\rfloor\Delta}$-martingale. We note $\sup_{s\le1} \lVert \Delta M^{n,\Delta,\eta_n}_s\rVert \le (v_{n\Delta}\eta_{1,n}^d\eta_{2,n}^d)^{-1}\lVert g_1\rVert_\infty\lVert g_2\rVert_\infty \to 0$ by \cref{D eq: LLN}. By Theorem VIII.2.4 of \citet{jacodshir}, thus, it is sufficient to show that the predictable quadratic variation of $M^{n,\Delta,\eta_n}$ at time one, denoted $\langle M^{n,\Delta,\eta_n},M^{n,\Delta,\eta_n} \rangle_1$, converges to zero in probability.

We observe 
\begin{gather*}
	\big\langle M^{n,\Delta,\eta},M^{n,\Delta,\eta} \big\rangle_1 \le \frac{\lVert g_1\rVert_\infty}{v_{n\Delta}\eta_{1}^d\eta_{2}^d}\sup_{z\in B_{\eta_{1}}(x)}\left|\frac{\eta_{2}^d}{\Delta} \E^z g_2^{\eta,y}(\Delta^n_1X)^2\right| v_{n\Delta}^{-1}\sum_{k=1}^{n} \Delta h^{\eta,x}(X_{(k-1)\Delta}).
% \frac{1}{v^2_{n\Delta}} \sum_{k=1}^{n} \E^{X_{(k-1)\Delta}}\left(g_2^{\eta,y}(\Delta^n_kX) - \E^{X_{(k-1)\Delta}}[g_2^{\eta,x}(\Delta^n_1X)]\right)^2g_1^{\eta,x}(X_{(k-1)\Delta})^2 \nonumber\\
% 	& \le \frac{1}{v_{n\Delta}\eta^d\vartheta^d} \sup_{z\in B_\eta(x)} \left\vert \int f(z,y + \vartheta w)g_1(w)^2\dw \right\vert S^{n,\Delta,\eta}_1(x). \label{D eq:MnQ}
\end{gather*}
By \cref{D l: P2a}, $\sup_{z\in B_{\eta_{1,n}}(x)}  \lvert \Delta^{-1}\E^{z}\eta_{2,n}^dg_2^{\eta_n,x}(\Delta^n_1X)^2| \to f(x,y)\int g_1(w)^2\dw$. In analogy to step (H), since $v_{n\Delta}\eta_{1.n}^d\eta_{2,n}^d\to\infty$, we have $\langle M^{n,\Delta,\eta_n},M^{n,\Delta,\eta_n} \rangle_1\to0$ in law, hence, in probability.
\end{proof}

%:- Proof of Theorem 5.10
\begin{proof}[of \cref{D t: LLN}]
We recall the results from \cref{D l: P2}. Let $\tilde L>0$ be a random variable such that the law $\mathscr{L}(f(x,y)\mu'(x)\tilde L, \mu'(x)\tilde L)$ is a limit point of the family in \cref{D eq:FamTight1}, and let $(n_k)_{k\in\bbn^\ast}$ be a sequence such that
\[
	\left(G^{n_k,\Delta,\eta_{n_k}}_1(x,y), R^{n_k,\Delta,\eta_{n_k}}_1(x)\right) \limdK \left(f(x,y)\mu'(x)\tilde L,\mu'(x)\tilde L\right).
\]
Since $\mu'(x)>0$, by the continuous mapping theorem, we conclude 
\begin{gather*}
	\hat f_{n_k}^{\Delta,\eta_{n_k}}(x,y) = \frac{G^{n_k,\Delta,\eta_{n_k}}_1(x,y)}{R^{n_k,\Delta,\eta_{n_k}}_1(x)} \limdK f(x,y).
\end{gather*}
As this limit is unique and independent of the particular limit point of the family in \cref{D eq:FamTight1}, we have that $\hat f_{n}^{\Delta,\eta_n}(x,y)$ converges to $f(x,y)$ in law, hence, in probability.
\end{proof}

%:------------------------------------------------------------------------
\subsection{Proofs of \texorpdfstring{\cref{D t: CLT,D c: CLT}}{Theorem 2.10 and Corollary 2.11}}\label{D ProofCLT}
Throughout this subsection, %$v$ denotes the regularly varying function from \cref{D eq: Darling Kac}. 
we work on the extension \cref{C eq:ExtendedPSpace} of the probability space, $L$ denotes the Mittag-Leffler process of order $0<\delta\le1$, and $W=(W^i)_{i\in I}$ denotes an $I$-dimensional standard Wiener process such that $L$, $W$ and $\mathscr{F}$ are independent.
%we work on the extension \cref{D eq: Extended Probability Space} of the probability space, $L$ denotes the Mittag-Leffler process of order $0<\delta\le1$, $V=(V(x,y))_{x\in E,y\in E^\ast}$ denotes a standard Gaussian white noise random field, and $W=(W^i)_{i\in I}$ denotes an $I$-dimensional standard Wiener process such that $L$, $W$ and $\mathscr{F}$, on the one hand, and $L$, $V$ and $\mathscr{F}$, on the other hand, are independent.

We consider the processes $G^{n,\Delta,\eta}$ and $R^{n,\Delta,\eta}$ given by \cref{D eq:Gn} and \cref{D eq:Sn}, and the processes $U^{n,\Delta,\eta}$ and $R^{\prime n,\Delta,\eta}$ given by
\begin{align}
	U^{n,\Delta,\eta}_s(x,y) & \coloneqq \sqrt{v_{n\Delta}\eta_1^d\eta_2^d} \left(G^{n,\Delta,\eta}_s(x,y) - \frac{\mu(g_1^{\eta,x}Fg_2^{\eta,y})}{\mu(g_1^{\eta,x})} R_s^{n,\Delta,\eta}(x) \right) \\
	R^{\prime n,\Delta,\eta}_s(x) & \coloneqq \frac{\Delta}{v_{n\Delta}} \sum_{k=1}^{\lfloor sn\rfloor} \eta_1^d g_1^{\eta,x}(X_{(k-1)\Delta})^2.
% \\
% 	M^{n,\Delta,\eta}_s(x,y) &\coloneqq \sqrt{\frac{\eta^d\vartheta^d}{v_{n\Delta}}} \sum_{k=1}^{\lfloor sn\rfloor}\left(g_2^{\eta,y}(\Delta^n_kX) - \E^{X_{(k-1)\Delta}}g_2^{\eta,y}(\Delta^n_1X)\right)g_1^{\eta,x}(X_{(k-1)\Delta}).
\end{align}
We recall that, under Darling--Kac's condition, we have Théo\-rème\,3 of \citet{Touati1987} at hand (see \cref{C prop:ML-Convergence}).
%; in addition, we obtain the following extensions of \cref{D l: P2a,D l: P2}.
First, we obtain an extension of \cref{D l: P1}.
%:! Lemma: RnR'n convergence
\begin{Lemma}\label{D l: P1b}
Grant \cref{D a: BCn,D a: C-weak,C a: HR,C a: DK}. 
Let $\eta_n=\eta_{1,n}$ be such that \cref{D eq: LLN} and \cref{D eq: LLNb} hold, and let $(x_i)_{i\in I}$ be a family of pairwise distinct points in $E_\oplus$. If $(n\Delta)^{1-\delta}\Delta \to 0$, then, under any law $\p^\pi$, we have the following stable convergence in law in $\mathcal{D}(\bbr^{2I})$:
\begin{gather}\label{D eq: P1b}
	\left( R^{n,\Delta,\eta_n}(x_i), R^{\prime n,\Delta,\eta_n}(x_i) \right)_{i\in I} \LimdstN \left(\mu'(x_i)L, \mu'(x_i)\int g_2(w)^2\dw L \right)_{i\in I}.
\end{gather}
\end{Lemma}
\begin{proof}
	Let $S^{t,\eta}_s(x) \coloneqq v_{t}^{-1}\int_0^{st}g^{\eta,x}_1(X_r)\dr$ and $S^{\prime t,\eta}_s(x) \coloneqq v_{t}^{-1}\int_0^{st}\eta^dg^{\eta,x}_1(X_r)^2\dr$. We note that $\mu(g^{\eta_n,x}_1) \to \mu'(x)$ and $\mu(\eta_n^d(g_1^{\eta_n,x})^2) \to \mu'(x)\int g_1(w)^2\dw$ for all $x$. By \cref{C lemma:AdditiveFunctional-LLN,C prop:ML-Convergence}, %Theorem 3.1 %of \citet{Ueltzhoefer2012}
%	and \cref{C prop:ML-Convergence}, 
	we deduce -- in analogy to \cref{C cor:Joint-CLT} %Corollary 3.13 %\marginpar{Update ref} of \citet{Ueltzhoefer2012}
	-- that
	\[
		\left(S^{n\Delta,\eta_n}(x_i), S^{\prime n\Delta,\eta_n}(x_i) \right)_{i\in I} \LimdstT \left(\mu'(x_i)L, \mu'(x_i)\int g_2(w)^2\dw L \right)_{i\in I}.
	\]
	For every $x$, moreover, we deduce from \cref{D p: Sojourn Time Approximation} that
	\[
		\left\lvert R^{n,\Delta,\eta_n}(x) -  S^{n\Delta,\eta_n}(x)\right\rvert \LimucpN 0 \mAnd  \left\lvert R^{\prime n,\Delta,\eta_n}(x) -  S^{\prime n\Delta,\eta_n}(x) \right\rvert \LimucpN 0.
	\]
	Consequently, we obtain \cref{D eq: P1b}.
\end{proof}
In view of \cref{D p: AuxMart}, we obtain the following preliminary result.
\begin{Lemma}\label{D l: P4}
	Grant \cref{D a: BCn,D a: C-weak,C a: HR,C a: DK,D a: C-strong}. Let $\eta_n=(\eta_{1,n},\eta_{2,n})$ be such that \cref{D eq: LLN} and \cref{D eq: Discrete} hold, and let $(x_i,y_i)_{i\in I}$ be a finite family of pairwise distinct points in $E_\oplus\times E^\ast$. If $(n\Delta)^{1-\delta}\Delta \to 0$, then, under any law $\p^\pi$, we have the following stable convergence in law in $\mathcal{D}(\bbr^I)$:
	\begin{gather}\label{D eq:P4}
		\left( R^{n,\Delta,\eta_n}(x_i), U^{n,\Delta,\eta_n}(x_i,y_i) \right)_{i\in I} \LimdstN \big(\mu'(x_i)L,  \sigma(x_i,y_i)\mu'(x_i)W^i_L\big)_{i\in I},
	\end{gather}
where $\sigma(x,y)^2$ is given by \cref{D eq: Variance}.
\end{Lemma}
\begin{proof}
	Let $(\mathscr{G}^n_s)_{s\ge0}$ be given by $\mathscr{G}^n_s = \mathscr{F}_{\lfloor sn\rfloor\Delta}$, and let the process $M^{n,\Delta,\eta}$ be  given by
	\[
		M^{n,\Delta,\eta}_s(x,y) \coloneqq \sqrt{\frac{\eta_1^d\eta_2^d}{v_{n\Delta}}} \sum_{k=1}^{\lfloor sn\rfloor}g_1^{\eta,x}(X_{(k-1)\Delta})\left(g_2^{\eta,y}(\Delta^n_kX) - \E^{X_{(k-1)\Delta}}g_2^{\eta,y}(\Delta^n_1X)\right).
	\]
	We note that $M^{n,\Delta,\eta}$ is a $\mathscr{G}^n_s$-martingale of the form \cref{D eq:Martingale form}. The proof is divided into four steps:
	First, we prove
	\begin{gather}\label{D eq:P4-1}
		\left\lvert U^{n,\Delta,\eta_n}(x,y) - M^{n,\Delta,\eta_n}(x,y) \right\rvert \LimucpN 0.
	\end{gather}
	Second, we show that the predictable quadratic variation of $M^{n,\Delta,\eta}(x,y)$ satisfies 
	\begin{gather}\label{D eq:P4-2}
		\left(\left\langle M^{n,\Delta,\eta_n}(x_i,y_i), M^{n,\Delta,\eta_n}(x_i,y_i) \right\rangle\right)_{i\in I} \LimdstN \big([\sigma(x_i,y_i)\mu'(x_i)]^2L\big)_{i\in I}
	\end{gather}
	in $\mathcal{D}(\bbr^I)$. Third, we show that $\langle M^{n,\Delta,\eta_n}(x_i,y_i),M^{n,\Delta,\eta_n}(x_j,y_j) \rangle$ vanishes for all $n$ large enough if $i\neq j$.
	Last, we argue
	\begin{gather*}\label{D eq:P4-3}
		\left(R ^{n,\Delta,\eta_n}(x_i), \left\langle M^{n,\Delta,\eta_n}(x_i,y_i),M^{n,\Delta,\eta_n}(x_i,y_i) \right\rangle\right)_{i\in I} \LimdstN \left(\mu'(x_i)L, [\sigma(x_i,y_i)\mu'(x_i)]^2L\right)_{i\in I}
	\end{gather*}
	in $\mathcal{D}(\bbr^{2I})$. By \cref{D p: AuxMart} and (3.5) of \citet{Hoepfner1990}, we then have \cref{D eq:P4}.

\bigskip\emph{(i)}
We note $U^{n,\Delta,\eta}_s(x,y) - M^{n,\Delta,\eta}_s(x,y) = H^{n,\Delta,\eta}_s(x,y) + H^{\prime n,\Delta,\eta}_s(x,y)$ with
\begin{align*}
	H^{n,\Delta,\eta}_s(x,y) & \coloneqq \sqrt{v_{n\Delta}\eta_{1}^d\eta_{2}^d}\frac{\Delta}{v_{n\Delta}}\sum_{k=1}^{\lfloor sn\rfloor}g_1(X_{(k-1)\Delta}) \left(Fg_2^{\eta,y}(X_{(k-1)\Delta}) - \frac{g_1^{\eta,x}Fg_2^{\eta,y}}{\mu(g_1^{\eta,x})} \right), \\
	|H^{\prime n,\Delta,\eta}_s(x,y)| &\le  \sqrt{v_{n\Delta}\eta_{1}^d\eta_{2}^d} \sup_{z\in B_{\eta_{1}}(x)} \left\lvert  \frac{1}{\Delta}\E^z g_2^{\eta,y}(\Delta^n_1X) - Fg_2^{\eta,y}(z)\right\rvert R^{\prime\prime n, \Delta,\eta}_s(x),
% 	\lefteqn{\lvert U^{n,\Delta,\eta}_s(x,y) - M^{n,\Delta,\eta}_s(x,y) \rvert} \\
% 	& \qquad \le \sqrt{v_{n\Delta}\eta^d\vartheta^d}\sup_{z\in B_\eta(x)}\left\lvert  \frac{1}{\Delta}\E^z g_2^{\eta,y}(\Delta^n_1X) - f(x,y)\right\rvert R^{n,\Delta,\eta}_s(x).
\end{align*}
where $R^{\prime\prime n, \Delta,\eta}_s(x) = \Delta v_{n\Delta}^{-1}\sum_{k=1}^{\lfloor sn\rfloor}h^{\eta,x}(X_{(k-1)\Delta})$ for some $\mathcal{C}^2$-function $h$, dominating $|g_1|$. Under \cref{D a: C-strong}, $Fg_2^{\eta,y}$ is twice continuously differentiable. Since \cref{D eq: Discrete} holds, by \cref{D p: Sojourn Time Approximation} and step (i) in the proof of \cref{C lemma:UnCLT}, $H^{n,\Delta,\eta_n}(x,y) \Rightarrow 0$ in ucp as $\nto$. By \cref{D p: Small Time Asymptotic} --~where we choose $m$ large enough~-- we have
	\[
		\sup_{z \in B_{\eta_{1}}(x)} \left \lvert \frac{1}{\Delta}\E^z g_2^{\eta,y}(\Delta^n_1X) - Fg_2^{\eta,y}(z) \right \rvert \le \zeta\left(\sqrt{\Delta} +  \Delta\eta_2^{-2\vee(\beta+d)}\right)
	\]
	since \cref{D eq: LLNb} holds. 
Since, moreover, \cref{D eq: Discrete} holds, therefore,
\begin{gather}\label{D eq: P4X}
	\sqrt{v_{n\Delta}\eta_{1,n}^d\eta_{2,n}^d} \sup_{z\in B_{\eta_{1,n}}(x)} \left\lvert  \frac{1}{\Delta}\E^z g_2^{\eta_n,y}(\Delta^n_1X) - Fg_2^{\eta_n,y}(z)\right\rvert \limN 0.
\end{gather}
In analogy to \cref{D l: P1b}, $R^{\prime\prime n,\Delta,\eta_n}(x)$ converges stably in law. Thus, $|H^{\prime n,\Delta,\eta}_s(x,y)|\Rightarrow 0$ in ucp as $\nto$. Consequently, \cref{D eq:P4-1} holds.
\Halmos

\emph{(ii)} We note $\langle M^{n,\Delta,\eta}(x,y),M^{n,\Delta,\eta}(x,y) \rangle_s = K^{n,\Delta,\eta}_s(x,y) - K^{\prime n,\Delta,\eta}_s(x,y)$, where
\begin{align*}
	K^{n,\Delta,\eta}_s(x,y) & = \frac{\eta_1^d\eta_2^d}{v_{n\Delta}} \sum_{k=1}^{\lfloor sn \rfloor} g_1^{\eta,x}(X_{(k-1)\Delta})^2\left(\E^{X_{(k-1)\Delta}}g_2^{\eta,y}(\Delta^n_1X)^2\right), \\
\shortintertext{and}
	|K^{\prime n,\Delta,\eta}_s(x,y)| &\le \sup_{z\in B_{\eta_1}(x)} \left|\frac{1}{\Delta^2}\left(\E^{X_{(k-1)\Delta}}g_2^{\eta,y}(\Delta^n_1X)\right)^2\right| \Delta\eta_2^d R^{\prime n,\Delta,\eta}_s(x).
\end{align*}
By \cref{D l: P2a} and the continuous mapping theorem, 
\[
	\sup_{z\in B_{\eta_{1,n}}(x)} \left| \frac{1}{\Delta^2}\left(\E^{z}g_2^{\eta_n,y}(\Delta^n_1X)\right)^2 \right| \limN  f(x,y)^2.
\]
By \cref{D l: P1b}, $R^{\prime n,\Delta,\eta_n}_s(x)$ converges stably in law. Since $\Delta\eta_{2,n}^d\to0$, we observe that $|K^{\prime n,\Delta,\eta_n}_s(x,y)|$ converges to zero uniformly on compacts in probability as $\nto$.
% \[
% 	|K^{\prime n}_s(x,y)| \LimucpN 0.
% \]

Again by \cref{D l: P2a},
\[
	\sup_{z\in B_{\eta_{1,n}}(x)} \left| \frac{\eta_{2,n}^d}{\Delta}\E^{X_{(k-1)\Delta}}g_2^{\eta_n,y}(\Delta^n_1X)^2 - f(x,y)\int g_2(w)^2\dw \right| \limN 0.
\]
In analogy to $K^{\prime n,\Delta,\eta}(x,y)$, therefore,
\begin{gather}\label{D eq:P4-2a}
	\left| K^{n,\Delta,\eta_n}(x,y) - f(x,y)\int g_1(w)^2\dw R^{\prime n,\Delta,\eta_n}(x) \right| \LimucpN 0.
\end{gather}
By \cref{D l: P1b}, consequently,
\[
	\big(K^{n,\Delta,\eta_n}(x_i,y_i)\big)_{i\in I} \LimdstN \left(f(x_i,y_i)\int g_1(w)^2\dw \mu'(x_i)\int g_2(z)^2\dz L\right)_{i\in I};
\]
hence, \cref{D eq:P4-2} holds. \Halmos

\emph{(iii)} Let $i,j\in I$. We note that for all $n$ large enough such that $\eta_{1,n},\eta_{2,n}$ are small enough, we have $g_1^{\eta_n,x_i}g_1^{\eta_n,x_j}\equiv 0$ whenever $x_i\neq x_j$, and $g_2^{\eta_n,y_i}g_2^{\eta_n,y_j}\equiv 0$ whenever $y_i\neq y_j$. For all $\omega$ and $n$ large enough, thus, $ \langle M^{n,\Delta,\eta_n}(x_i,y_i), M^{n,\Delta,\eta_n}(x_j,y_j)  \rangle_s \equiv 0$ if $i\neq j$.
\Halmos

\emph{(iv)} By \cref{D l: P1b} and \cref{D eq:P4-2a}, we obtain the joint convergence of $(R^{n,\Delta,\eta_n}(x_i))_{i\in I}$ and $\langle M^{n,\Delta,\eta_n}(x_i,y_i),M^{n,\Delta,\eta_n}(x_i,y_i)\rangle)_{i\in I}$ to the required limit.
\end{proof}

%:- Proof of Theorem 5.11
\begin{proof}[of \cref{D t: CLT}]
	For every $n$, and $(x,y)\in E_\oplus\times E^\ast$, we have
	\[
		\sqrt{v_{n\Delta}\eta_{1,n}^d\eta_{2,n}^d}\left(\hat f^{\Delta,\eta_n}_n(x,y) - \bar f^{\eta_n}(x,y)\right)
			=
		\frac{U^{n,\Delta,\eta_n}_1(x,y)}{R^{n,\Delta,\eta_n}_1(x)},
	\]
where $\bar f^{\eta}(x,y)\coloneqq {\mu(g_1^{\eta,x}Fg_2^{\eta,y})}/{\mu(g_1^{\eta,x})}$.
	Since $L$ and $W$ are independent, $V(x_i,y_i)\coloneqq L_1^{-1/2}W^i_{L_1}$ defines an $I$-dimensional standard Gaussian random vector such that $L$, $V$ and $\mathscr{F}$ are independent.
	By the continuous mapping theorem and \cref{D l: P4}, consequently,
	\[
		\sqrt{v_{n\Delta}\eta_{1,n}^d\eta_{2,n}^d}\left(\hat f^{\Delta,\eta_n}_n(x_i,y_i) - \bar f^{\eta_n}(x_i,y_i)\right)_{i\in I}
			\limdstN
		\left({\sigma(x_i,y_i)}V(x_i,y_i)L_1^{-1/2}\right)_{i\in I},
	\]
	where $\sigma(x,y)^2$ is given by \cref{D eq: Variance}.
	
	In addition, let $\eta_n=(\eta_{1,n},\eta_{2,n})$ be such that \cref{D eq: CLT} holds as well. It remains to prove $({v_{n\Delta}\eta_{1,n}^d\eta_{2,n}^d})^{1/2}(\bar f^{\eta_n}(x,y)-f(x,y))\to \gamma(x,y)$. This, however, follows in analogy to the proof of \cref{C theo:CLT}.
\end{proof}
%:- Proof of Corollary 5.12
\begin{proof}[of \cref{D c: CLT}]
In analogy to the proof of \cref{D t: CLT}, by \cref{D l: P4} it remains to show that $({v_{n\Delta}\eta_{1,n}^d\eta_{2,n}^d})^{1/2}\hat\gamma^{\eta_n}_n(x,y)$ is a consistent estimator for $\gamma(x,y)$. This, however, follows in analogy to the proof of \cref{C cor:CLT}.
\end{proof}

\appendix
\section{On the auxiliary Markov chains \emph{Z} and \emph{Z'}}\label{C OnZ}
In this appendix, we derive an explicit representation for the transition kernel $\Phi$ of the auxiliary process $Z^\prime$, and  (in-)equalities for expectations of the form $\E^x(\int_0^{T_1}h(X_s)\ds)^k$. In addition, we derive representations for the stationary probability measures $\psi$ and $\varphi$ of the processes $Z$ and $Z^\prime$.

%We recall that the auxiliary chains $Z$ and $Z^\prime$ are given by $Z_{k} =  (s\mapsto X_{(1-s)T_{k-1}+sT_{k}}, T_{k}-T_{k-1}, X_{T_{k}-})$ and $Z^\prime_{k} = X_{T_{k}-}$. 
We invoke technical results on resolvents of semi-groups. The \emph{resolvent} $({R}_{\lambda})_{\lambda>0}$ of a semi-group $(P_t)_{t\ge0}$ is given by ${R}_{\lambda} \coloneqq \int_{0}^\infty \exp(-\lambda t)P_{t}\dt$. For bounded measurable functions $h$, the generalised resolvent kernel ${R}_{h}$ is given by
\[
	{R}_{h}(x,A) \coloneqq \E^x \int_{0}^\infty e^{-\int_{0}^t h(X_{s})\ds} \mathbbm{1}_{A}(X_{t})\dt \qquad \forall x\in E, A\in\mathscr{E}.
\]
These kernels were first introduced by \citet{Neveu1972}. For a comprehensive interpretation, we refer to section 4 of \citet{Down1995}.

%:! Lemma: Upsilon Darstellung
\begin{Lemma}\label{C lemma:ModifiedResolvent}
Let $(R_\lambda)_{\lambda>0}$ be the resolvent of $X$, and let $(R^\ast_\lambda)_{\lambda>0}$ be given by
\begin{gather}\label{C eq:Rast}
	R^\ast_\lambda \coloneqq R_\lambda \sum_{k=0}^\infty\big((\Iota_{ q}-\Iota_{ q}\bar\Pi){R}_{\lambda}\big)^k, \quad\text{where}\quad \Iota_q h(x)\coloneqq q(x)h(x).
\end{gather}
Then $(R^\ast_\lambda)_{\lambda>0}$ is the resolvent of a positive contraction semi-group. For its corresponding process $X^\ast$, we have that the laws of $X^\ast\mathbbm{1}_{[\![0,T_{1}[\![}$ and $X\mathbbm{1}_{[\![0,T_{1}[\![}$ are equal.
\end{Lemma}
\begin{proof}
Since $\Iota_q\bar\Pi$ is a bounded kernel, $(R^\ast_\lambda)_{\lambda>0}$ is the resolvent of a positive contraction semi-group by Theorem 4.2 of \citet{Bass1979}. It follows from \citet{Sawyer1970} and Chapter\,6 of \citet{Bass1979} that, for the process $X^\ast$ (corresponding to $({R}^\ast_{\lambda})_{\lambda>0}$), we have that the laws of $X^\ast\mathbbm{1}_{[\![0,T_{1}[\![}$ and $X\mathbbm{1}_{[\![0,T_{1}[\![}$ are equal.
\end{proof}
\begin{Lemma}\label{C p:Representation}
Let $h$ be a measurable function on $E$. Then
\begin{gather}\label{C eq:Representation}
	\E^x h(Z^\prime_1) = R_q^\ast I_q h(x)
	\quad\text{and}\quad
	\E^x \int_0^{T_1} h(X_s)\ds = R_q^\ast h(x) ,
\end{gather}
where $R_q^\ast$ denotes the generalised resolvent kernel associated with the modified resolvent $(R^\ast_\lambda)_{\lambda>0}$ and the function $q$. For every $\lambda_{q}\ge\lVert  q\rVert_{\infty}$, we have $R_q^\ast = \sum_{k=0}^\infty R^\ast_{\lambda_{q}}(\Iota_{\lambda_{q}- q}R^\ast_{\lambda_{q}})^k$.
\end{Lemma}
\begin{proof}
We recall that the laws of $X^\ast\mathbbm{1}_{[\![0,T_{1}[\![}$ and $X\mathbbm{1}_{[\![0,T_{1}[\![}$ are equal. The expectation of $h(Z^\prime_1)$ under $\p^x$, therefore, coincides with the expectation of $h(X^\ast)$ sampled at an independent killing time according to the multiplicative functional $\exp(-\int_{0}^\cdot  q(X^\ast_{s})\ds)$. In formulas, we have 
\begin{equation*}
	\E^x h(Z^\prime_1) = \E^x \int_{0}^\infty e^{-\int_{0}^t q(X^\ast_{s})\ds} q(X^\ast_{t})h(X^\ast_{t})\dt.
\end{equation*}
By eq.\ (19) of \citet{Down1995}, hence, $\E^x h(Z^\prime_1) = R_q^\ast \Iota_q h(x)$, where $R_q^\ast$ denotes the generalised resolvent kernel associated with the modified resolvent $(R_\lambda^\ast)_{\lambda>0}$. 
By Chapter~7 of \citet{Neveu1972}, $R_q^\ast = \sum_{k=0}^\infty R^\ast_{\lambda_{q}}(\Iota_{\lambda_{q}- q}R^\ast_{\lambda_{q}})^k$ holds for every $\lambda_{q}\ge\lVert  q\rVert_{\infty}$.

Similarly, we observe 
\begin{gather}\label{C eq:RepA}
	\E^x \int_0^{T_1} h(X_s)\ds = \E^x \int_{0}^\infty e^{-\int_{0}^t q(X^\ast_{u})\du} q(X^\ast_{t})\int_{0}^t h(X^\ast_{s})\ds  \dt.
\end{gather}
By Fubini's theorem (cf., eq. (20) of \citet{Down1995}), consequently,
\[
	\E^x \int_0^{T_1} h(X_s)\ds = \E^x \int_{0}^\infty e^{-\int_{0}^t q(X^\ast_{s})\ds}  h(X^\ast_{t})\dt = R_q^\ast h(x).
\]
\end{proof}
\begin{Remark}
It is immediate from  \cref{C lemma:OnYZ} that $\Phi=\bar\Pi R^\ast_q\Iota_q$.
\end{Remark}
We obtain two corollaries:
\begin{Corollary}\label{C cor:Representation}
Let $h_1,\dotsc,h_k$ be measurable functions on $E$. Then
\begin{gather}\label{C eq:Prod-hk}
	\E^x \prod_{j=1}^k \int_0^{T_1} h_j(X_s)\ds =
	\sum_{j=1}^k \E^x \int_{0}^\infty e^{-\int_{0}^t q(X^\ast_{u})\du} h_j(X^\ast_t) \prod_{l\neq j} \int_0^{t} h_l(X^\ast_s)\ds\dt.
\end{gather}
\end{Corollary}
\begin{proof}
In analogy to \cref{C eq:RepA}, we observe
\[
	\E^x \prod_{j=1}^k \int_0^{T_1} h_j(X_s)\ds =  \E^x \int_{0}^\infty e^{-\int_{0}^t q(X^\ast_{u})\du} q(X^\ast_{t})\prod_{j=1}^k \int_0^{t} h_j(X^\ast_s)\ds\dt.
\]
By the Leibniz rule, moreover,
\[
	\prod_{j=1}^k \int_0^{t} h_j(X^\ast_s)\ds = \sum_{j=1}^k \int_0^t h_j(X^\ast_s) \prod_{l\neq j} \int_0^{s} h_l(X^\ast_r)\dr\ds.
\]
By Fubini's theorem, therefore, we have \cref{C eq:Prod-hk}.
\end{proof}
\begin{Corollary}\label{C cor:kthMoment-Inequality}
Let $h$ be a bounded measurable function on $E$. For all $k\in\bbn^\ast$, if $\inf_{x\in\supp(h)}q(x)>0$, then
\begin{gather}\label{C eq:R2}
	\E^x\left(\int_0^{T_1} h(X_s)\ds\right)^k \le \frac{k! \lVert h\rVert_\infty^{k-1}}{(\inf_{x\in\supp(h)}q(x))^{k-1}} R^\ast_q |h|(x).
\end{gather}
\end{Corollary}
\begin{proof}[by induction]
By \cref{C p:Representation}, we immediately have \cref{C eq:R2} for $k=1$. We assume that \cref{C eq:R2} holds for some $k\in\bbn^\ast$. Then we deduce from \cref{C cor:Representation} and $|h| \le q \lVert h\rVert_\infty/(\inf_{x\in\supp(h)}q(x))$ that \cref{C eq:R2} holds for $k+1$.
\end{proof}

\begin{Lemma}\label{C lemma:Mu-Invariant}
$\mu \Iota_{ q}\bar\Pi {R}^\ast_{ q} = \mu$.
\end{Lemma}
\begin{proof}
By Theorem 4.2 of \citet{Bass1979} and Section 7 of \citet{Neveu1972}, we have 
\[
	(\Iota_{ q}\bar\Pi - (\Iota - {R}_{1}^{-1}) ){R}^\ast_{ q} = \Iota,
\]
where the formal inverse of $R_1$ is defined by $R_1^{-1} \coloneqq \sum_{k=0}^\infty (\Iota-R_1)^k$.
Since $\mu$ is invariant \wrt\ $(P_{t})_{t\ge0}$, we also have $\mu {R}_{1} = \mu$ and $\mu = \mu {R}_{1}^{{-1}}$. Hence, $\mu \Iota_{ q}\bar\Pi = \mu(\Iota_{ q}\bar\Pi -  (\Iota - {R}_{1}^{-1})).$
Therefore, $\mu \Iota_{ q}\bar\Pi {R}^\ast_{ q} = \mu$.
\end{proof}
\begin{Corollary}\label{C cor:upsilon-psi-invariant}
The measures $\varphi \coloneqq (\mu(q))^{-1}\mu\Iota_q$ and $\psi\coloneqq \varphi\Psi$ are the invariant probability measures \wrt\ $\Phi$ and $\Psi$.
\end{Corollary}
\begin{proof}
Since $\Phi=\bar\Pi {R}^\ast_{ q}\Iota_{ q}$, we observe  $\mu \Iota_{ q} \Phi = \mu \Iota_{ q}$. By \cref{C eq:PsiK}, $\varphi\Psi^{k+1}=\varphi\Phi^k\Psi=\varphi\Psi$ in addition.
\end{proof}

{
\section{Simulations}\label{sec:Simul}
In this appendix, we present a small simulation study for our estimator. We implemented a numerical simulation scheme for a process with infinite activity. In particular, we considered the univariate Itō semi-martingale with characteristics $(B,C,\mathfrak{n})$ given by 
$\mathrm{d}B_t = -bX_t\dt$, $\mathrm{d}C_t = c\dt$, and $\mathfrak{n}(\dt,\dy) = f(X_t,y)\dt\dy$, 
%\begin{gather}\label{eq:ItoChar}
%	\mathrm{d}B_t = -bX_t\dt, \quad \mathrm{d}C_t = c\dt, \mAnd \mathfrak{n}(\dt,\dy) = f(X_t,y)\dt\dy,
%\end{gather}
where $b,c>0$ and the density of the Lévy kernel is a stable density with state-dependent intensities; in particular,
\begin{gather}\label{eq:IALevy}
	f(x,y) \coloneqq \Big(\zeta_+(x)\mathbbm{1}_{\bbr_+^\ast}(y) + \zeta_-(x)\mathbbm{1}_{\bbr_-^\ast}(y)\Big)|y|^{-1-\alpha}, 
\end{gather}
where $\zeta_+(x) = 2$ if $x \in {} ]{-\infty},{-\xi}]$, $\zeta_+(x) = 2 - (1 + \cos(\pi x/\xi))/2$ if $x \in {} ]{-\xi},0]$, $\zeta_+(x) = (1 + \cos(\pi x/\xi))/2$ if $x \in {} ]0, \xi]$, $\zeta_+(x) = 0$ if $x \in {} ]\xi, \infty[$, and $\zeta_-(x) \coloneqq 2 - \zeta_+(x)$.
%\begin{align*}
%	\zeta_+(x) & \coloneqq 
%\begin{cases}
%	2, & \text{if } x \in {} ]{-\infty},{-\xi}], \\
%	2 - (1 + \cos(\pi x/\xi))/2 & \text{if } x \in {} ]{-\xi},0], \\
%	(1 + \cos(\pi x/\xi))/2 & \text{if } x \in {} ]0, \xi], \\
%	0 & \text{if } x \in {} ]\xi, \infty[,
%\end{cases} \\
%	\zeta_-(x) & \coloneqq 2 - \zeta_+(x).
%\end{align*}
We emphasise the singularity on the set $\bbr\times\{0\}$; also, we note that $f$ is not twice continuously differentiable for $x \in \{-\xi,\xi\}$ -- since we do not estimate close to these points in the sequel, this has no impact on our simulations.

We chose the parameters of the process as follows: $b=c=1$, $\xi=3$, and $\alpha = 0.9$; and investigated the scenarios 
%\begin{enumerate}[~~d1)]
%	\item 
d1) $t_1 = 1000$ and $\Delta_1 = 0.01$, that is 100\,000 observations; 
%	\item 
d2)	$t_2 = 1000$ and $\Delta_2 = 0.0025$, that is 400\,000 observations; and
%	\item 
d3)	$t_3 = 2500$ and $\Delta_3 = 0.0025$, that is 1\,000\,000 observations.
%\end{enumerate}
We simulated the process with the Euler scheme; as step length, we chose 1/10-th of the observation time-lag $\Delta$. Given the value $X_{k\Delta/10}$, we simulated a stable increment with Lévy density $y\mapsto f(X_{k\Delta/10},y)$ and a Brownian increment with drift $-bX_{k\Delta/10}$ and volatility $c$. Iteratively, we obtained an approximate sample $X_0, X_{\Delta/10},\dotsc, X_{n\Delta}$. Finally, we only kept every tenth observation.

We have implemented the kernel density estimator \cref{D eq: Estimator} using the so-called \emph{bi-weight kernel} $g(z) \coloneqq 0.9375(1-z^2)^2 \mathbbm{1}_{[-1, 1]}(z).$ %Its roughness is given by $\xi_g = \int g(z)^2\dz= 5/7$; its second moment by $\int z^2g(z)\dz = 1/7$. 
To calculate asymptotic confidence intervals derived from \cref{D c: CLT} which are non-negative, we invert a test-statistic following, for instance, \citet[p.\ 24]{Hansen2009}. We compare our estimates $\hat f^{\Delta,\eta}_n(x,y)$ in terms of their functional properties: 

We observe a significant influence of the bandwidth choice. In scenario d1, for instance, we observe that $\eta_1>0.2$ (resp., $\eta_1>0.3$) is necessary to obtain reasonable estimates at $x=0$ (resp., at $x=2$). On the neighbourhoods $\{|y| \le \eta_2 + 0.3\}$ of the origin, the bias due to discretisation is dominant. At $x=0$, we obtain good estimates on the sets $\{0.5 \le |y| \le 1\}$ and $\{0.75 \le |y| \le 4\}$ for the bandwidth choices $\eta=(0.2,0.2)$ and $\eta = (0.4, 0.4)$, respectively. At $x=2$, we obtain good estimates on the sets $\{-3.5 < |y| < -0.75\}$ and $\{0.75 < y < 1.5\}$ for $\eta=(0.4,0.4)$. In scenario d2, where the observation time-lag is one quarter of the time-lag of scenario d1, first, we observe that the bias due to discretisation is dominant on the set $\{|y| \le \eta_2 + 0.2\}$. Apart from the improvement for $|y|$ small, the estimates in scenario d2 are similar to those of scenario d1. Finally, we observe that, for scenarios d2 and d3 where the observation time-lag is equal, the set on which the bias due to discretisation is dominant coincides. Nevertheless, the estimation for $|y|$ large improves significantly. At $x=0$, we obtain very good estimates on the sets $\{0.4 < |y| < 3\}$ and $\{0.6 < |y| < 5\}$ for $\eta=(0.4,0.2)$ and $\eta=(0.2,0.4)$, respectively. At $x=2$, we obtain very good estimates on the sets $\{-4 < y < -0.6\}$ and $\{0.6 < y < 2\}$ for $\eta=(0.2,0.4)$. We present our results for scenario d3 in Figure~\ref{fig:d3}. 

In summary, on the one hand, we have seen that larger bandwidths give better 
estimates in terms of variability and the degree of smoothing for 
$|y|$ large. On the 
other hand, smaller bandwidths allow for more reasonable estimates closer to zero 
than larger ones. Moreover, increasing the number of observations without reducing 
the observation time-lag does not give better estimates close to zero. 
For further details and a study of the finite activity case, we refer to \citet{UeltzhoeferPhD}.
}

%:====================================================
\section*{Acknowledgment}
I am grateful to Jean Jacod for introducing me to this topic, illuminating discussions, and helpful comments on earlier versions of this manuscript. {I would also like to thank an anonymous referee for his valuable remarks which improved this presentation.}
%With the support of the Technische Universit\"at M\"unchen -- Institute for Advanced Study, funded by the German Excellence Initiative; and support provided by the TUM International Graduate School of Science and Engineering (IGSSE).

%:====================================================
%:- Referenzen
\bibliography{library}
\bibliographystyle{DissertationAY}
\begin{figure}
\centering
\includegraphics[width=.98\textwidth]{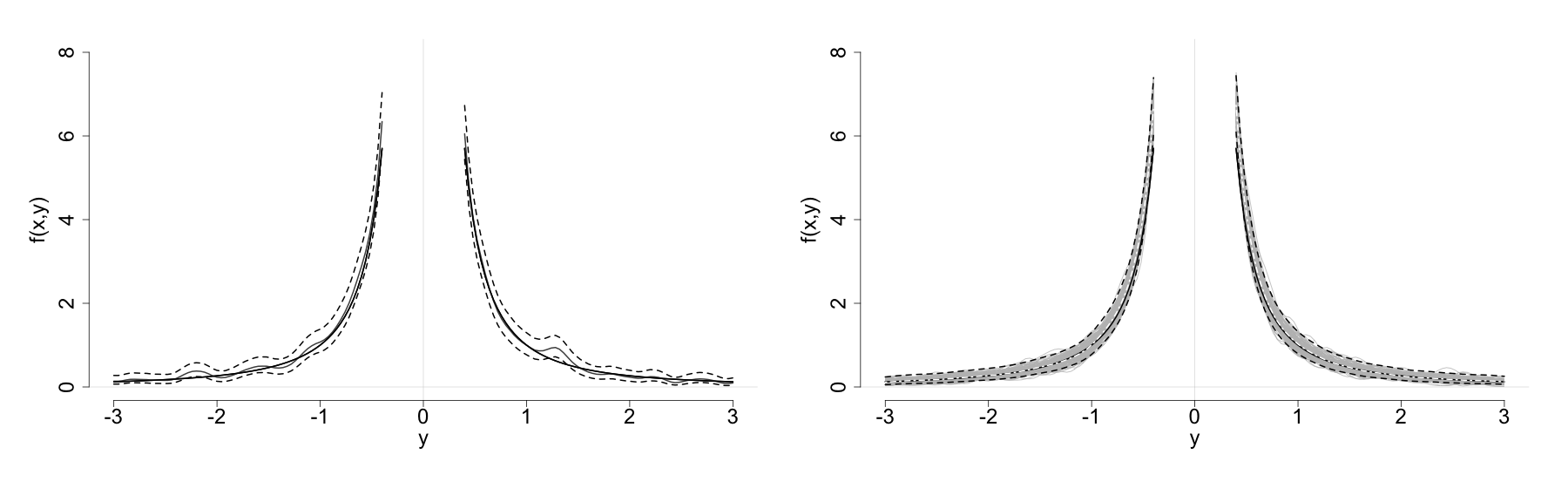}
\includegraphics[width=.98\textwidth]{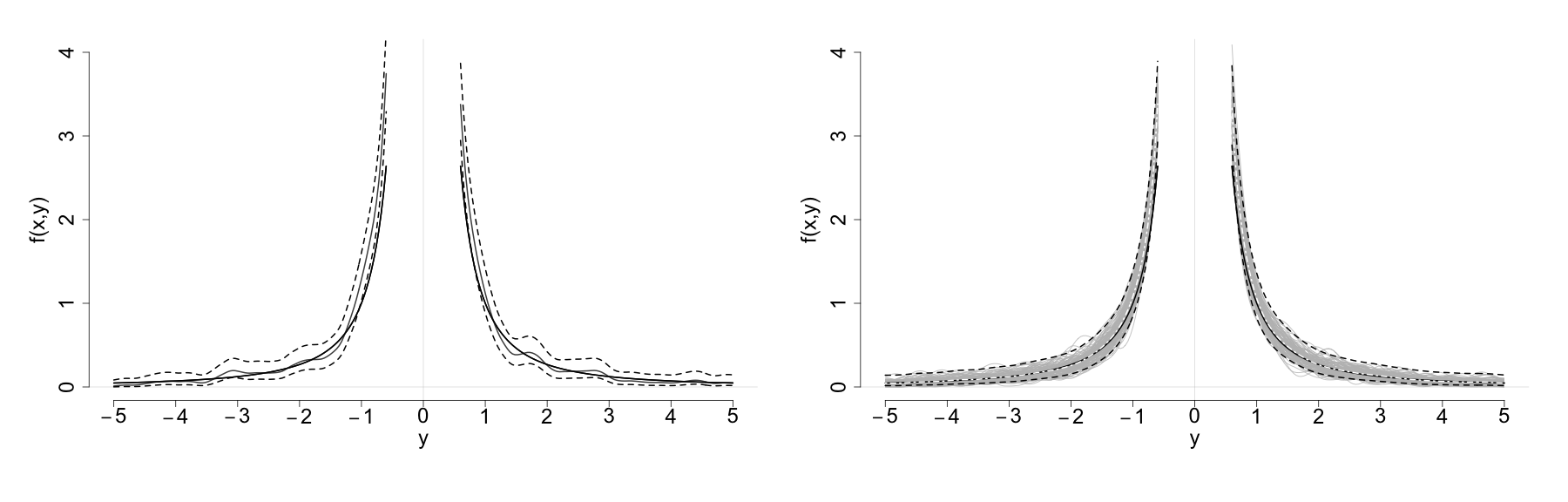}
\includegraphics[width=.98\textwidth]{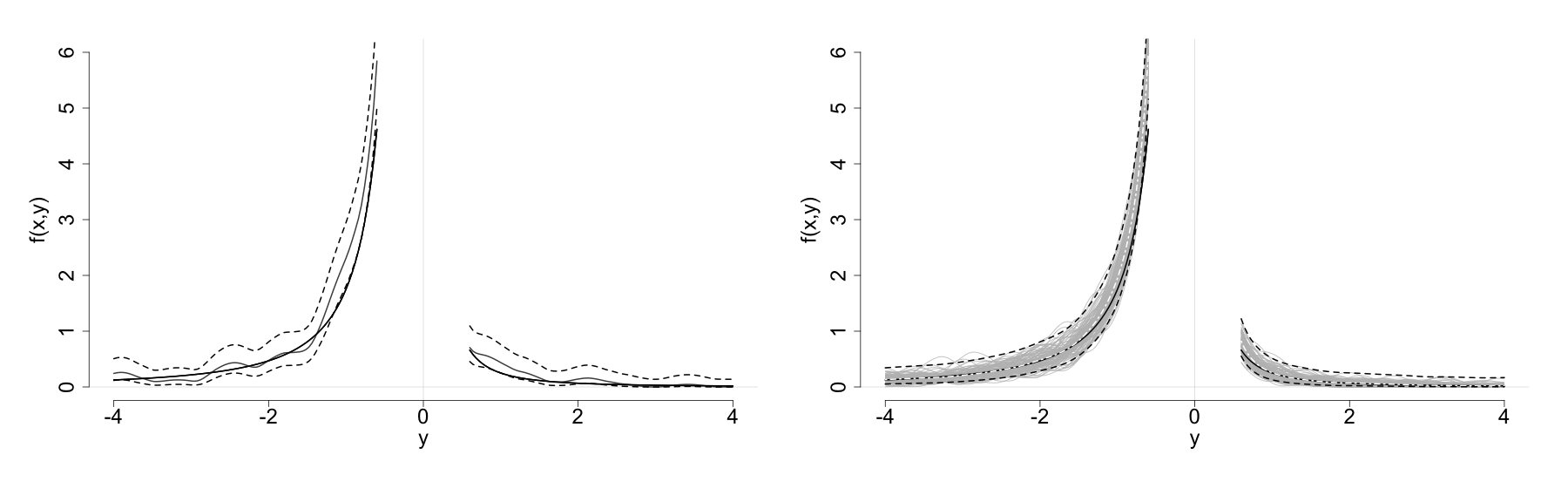}
\caption[Scenario d3 --- Estimation at $x=0$ and $x=2$]{Scenario d3 --- Estimation of the Lévy density $f(x,y)$ given by \cref{eq:IALevy} at $x=0$ with $\eta=(0.4,0.2)$ (top row), at $x=0$ with $\eta=(0.2,0.4)$ (middle row), and at $x=2$ with $\eta=(0.2,0.4)$ (bottom row) based on discrete observations with lag $\Delta=0.0025$ up to time $n\Delta=2500$. Left: One typical estimate (grey) is compared to the true Lévy density (black). The upper and lower bounds of the estimated (pointwise) 95\%\,-confidence intervals (dashed lines) are shown. Right: Estimates based on 100 trajectories (grey) are compared to the true Lévy density (black). The (pointwise) mean of the estimates (white dashed/dotted line) and mean of the upper and lower bounds of the 95\%\,-confidence intervals (black dashed lines) are shown.}\label{fig:d3}
\end{figure}
\end{document}